\documentclass[11pt]{article}
\usepackage{amsmath,amssymb,amsthm}
\usepackage[width=16.5cm,height=23cm]{geometry}
\usepackage[small,PostScript=dvips,heads=curlyvee]
{diagrams}

\usepackage{tikz}

\theoremstyle{plain}
\newtheorem{theorem}{Theorem}[section]
\newtheorem*{theorem*}{Theorem}
\newtheorem{lemma}[theorem]{Lemma}
\newtheorem{proposition}[theorem]{Proposition}
\newtheorem*{proposition*}{Proposition}
\newtheorem{corollary}[theorem]{Corollary}
\newtheorem*{corollary*}{Corollary}
\newtheorem{remark}[theorem]{Remark}
\newtheorem{remark*}{Remark}

\newtheorem{defn*}{Definition}
\newtheorem{example}[theorem]{Example}

\newarrowmiddle{o}\hho\hho\circ\circ

\newarrow{Closed}{hooka}{-}{+}{-}{>}
\newarrow{Open}{hooka}{-}{o}{-}{>}
\newarrow{Zigzag}{<}{-}{-}{-}{>}
\newarrow{Bundle}{-}{-}{-}{-}{>}
\newarrow{Etale}{triangle}{-}{-}{-}{>}
\newarrow{Cof}{>}{-}{-}{-}{>}
\newarrow{Fib}{-}{-}{-}{-}{>>}
\newarrow{Lift}{}{dash}{}{dash}{>}
\newarrow{ClosedLine}{hooka}{-}{+}{-}{-}
\newarrow{IntoLine}{hooka}{-}{-}{-}{-}
\newarrow{Equal}{=}{=}{=}{=}{=}

\let\smash=\wedge
\let\iso=\cong
\let\tensor=\otimes

\let\directsum=\oplus
\let\bigdirectsum=\bigoplus

\def\minus{\!\smallsetminus\!}

\def\colim{\operatornamewithlimits{colim}}
\def\hocolim{\operatornamewithlimits{hocolim}}

\def\hofib{\operatornamewithlimits{hofib}}
\def\cone{\operatorname{cone}}

\def\pr{\operatorname{pr}}
\def\Char{\operatorname{char}}

\def\AA{\mathbb{A}}
\def\PP{\mathbb{P}}

\def\ZZ{\mathbb{Z}}

\def\QQ{\mathbb{Q}}
\def\RR{\mathbb{R}}

\def\FF{\mathbb{F}}
\def\GL{\mathbb{GL}}

\def\op{\mathrm{op}}

\def\Spec{\mathrm{Spec}}

\def\units{\times}
\def\Hom{\operatorname{Hom}}

\def\id{\mathrm{id}}
\def\Id{\mathrm{Id}}
\def\Ker{\mathrm{Ker}}
\def\Img{\mathrm{Im}}
\def\tr{\mathrm{tr}}

\def\Sm{\mathrm{Sm}}
\def\DM{\mathrm{DM}}

\def\eff{\mathrm{eff}}

\def\sSet{\mathbf{sSet}}
\def\Spt{\mathbf{Spt}}

\def\OO{\mathcal{O}}

\newcommand{\KGL}{\mathbf{KGL}}
\newcommand{\MGL}{\mathbf{MGL}}
\newcommand{\can}{\mathrm{can}}

\newcommand{\KT}{\mathbf{KT}}
\newcommand{\KQ}{\mathbf{KQ}}

\newcommand{\MZZ}{\mathbf{MZ}}
\newcommand{\Mmu}{\mathbf{M}\mu}
\newcommand{\Mnu}{\mathbf{M}\nu}
\newcommand{\Sq}{\mathsf{Sq}}
\newcommand{\Qop}{\mathsf{Q}}
\newcommand{\One}{\mathbf{1}}
\newcommand{\SHH}{\mathbf{SH}}
\newcommand{\E}{\mathsf{E}}
\newcommand{\D}{\mathsf{D}}
\newcommand{\F}{\mathsf{F}}
\newcommand{\f}{\mathsf{f}}
\newcommand{\s}{\mathsf{s}}
\newcommand{\cd}{\mathsf{cd}}
\newcommand{\vcd}{\mathsf{vcd}}

\newcommand{\sliced}{\mathbf{d}}

\begin{document}

\title{Slices of hermitian $K$-theory and \\ Milnor's conjecture on quadratic forms}
\author{Oliver R{\"o}ndigs \;\;\; Paul Arne {\O}stv{\ae}r}
\date{\today}
\maketitle

\begin{abstract}
We advance the understanding of $K$-theory of quadratic forms by computing the slices of the motivic spectra 
representing hermitian $K$-groups and Witt-groups.
By an explicit computation of the slice spectral sequence for higher Witt-theory, 
we prove Milnor's conjecture relating Galois cohomology to quadratic forms via the filtration of the Witt ring by its fundamental ideal.
In a related computation we express hermitian $K$-groups in terms of motivic cohomology.
\end{abstract}
\tableofcontents

\newpage

\section{Introduction}

Suppose that $F$ is a field of characteristic $\Char(F)\neq 2$.
In \cite{Milnor} the Milnor $K$-theory of $F$ is defined in terms of generators and relations by 
\[
K^{M}_{\ast}(F)
=
T^{\ast}F^{\times}/(a\otimes (1-a));
\,\, a\neq 0,1.
\]
Here $T^{\ast}F^{\times}$ is the tensor algebra of the multiplicative group of units $F^{\times}$. 
In degrees zero, one and two these groups agree with Quillen's $K$-groups, but for higher degrees they differ in general. 
Milnor \cite{Milnor} proposed two conjectures relating $k^{M}_{\ast}(F)=K^{M}_{\ast}(F)/2K^{M}_{\ast}(F)$ to the mod-$2$ Galois cohomology ring $H^{\ast}(F;\ZZ/2)$ and the 
graded Witt ring $GW_{\ast}(F)=\oplus_{q\geq 0}I(F)^{q}/I(F)^{q+1}$ for the fundamental ideal $I(F)$ of even dimensional forms, via the two homomorphisms:
\begin{diagram}
&& k^{M}_{\ast}(F) \\
& \ldTo^{s^{F}_{\ast}} && \rdTo^{h^{F}_{\ast}} \\
GW_{\ast}(F) &&&& H^{\ast}(F;\ZZ/2)
\end{diagram}

The solutions of the Milnor conjectures on Galois cohomology \cite{Voevodsky:Z/2} 
-- $h^{F}_{\ast}$ is an isomorphism -- 
and on quadratic forms \cite{Orlov-Vishik-Voevodsky} 
-- $s^{F}_{\ast}$ is an isomorphism -- 
are two striking applications of motivic homotopy theory.
For background and influence of these conjectures and also for the history of their proofs we refer to \cite{MR1215959}, \cite{MR1627118}, \cite{MR1627119}, \cite{MR1432056}, 
\cite{MR2841244}, \cite{MR1600334}, \cite{MR1769021}, \cite{MR1700298}.
In this paper we give an alternate proof of Milnor's conjecture on quadratic forms by explicitly computing the slice spectral sequence for the higher Witt-theory spectrum $\KT$. 
Our method of proof applies also to smooth semilocal rings containing a field of characteristic zero.
\vspace{0.1in}

Let $X\in\Sm_{F}$ be a smooth scheme of finite type over a field $F$.
We refer to \cite{GraysonKhandbook} for a survey of the known constructions of the first quadrant convergent spectral sequence relating 
motivic cohomology to algebraic $K$-theory
\begin{equation}
\label{equation:algebraicktheoryss}
\MZZ^{\star}(X)
\Longrightarrow
\KGL_{\ast}(X).
\end{equation}
From the viewpoint of motivic homotopy theory \cite{Voevodsky:open}, 
the problem of constructing (\ref{equation:algebraicktheoryss}) reduces to identifying the slices $\s_{q}(\KGL)$ of the motivic spectrum 
$\KGL$ representing algebraic $K$-theory.
Voevodsky introduced the slice spectral sequence and conjectured that 
\begin{equation}
\label{equation:algebraicktheoryslices}
\s_{q}(\KGL)
\cong
\Sigma^{2q,q}\MZZ.
\end{equation}
The formula (\ref{equation:algebraicktheoryslices}) was proven for fields
of characteristic zero by 
Voevodsky \cite{Voevodsky:motivicss}, 
\cite{Voevodsky:zero-slice}, and (invoking different methods) for perfect fields by Levine \cite{Levine:slices}.
By base change the same holds for all fields.
\vspace{0.1in}

When $\Char(F)\neq 2$ we are interested in the analogues of (\ref{equation:algebraicktheoryss}) and (\ref{equation:algebraicktheoryslices}) 
for the motivic spectra $\KQ$ and $\KT$ representing hermitian $K$-groups and Witt-groups on $\Sm_{F}$, 
respectively \cite{Hornbostel}.
Theorem \ref{thm:slices-ko} shows the slices of hermitian $K$-theory are given by infinite wedge product decompositions
\begin{equation}
\label{equation:hermitianktheoryslices}
\s_{q}(\KQ)
\cong
\begin{cases}
\Sigma^{2q,q}\MZZ
\vee  
\bigvee_{i<\frac{q}{2}}\Sigma^{2i+q,q}\MZZ/2 & q\equiv 0\bmod 2 \\
\bigvee_{i<\frac{q+1}{2}}\Sigma^{2i+q,q}\MZZ/2 & q\equiv 1\bmod 2. \\
\end{cases}
\end{equation}
Moreover,
in Theorem \ref{thm:slices-witt-theory} we compute the slices of higher Witt-theory, 
namely  
\begin{equation}
\label{equation:wittheoryslices}
\s_{q}(\KT)
\cong
\bigvee_{i\in\ZZ}\Sigma^{2i+q,q}\MZZ/2.
\end{equation} 
The summand $\Sigma^{2q,q}\MZZ$ in \eqref{equation:hermitianktheoryslices} is detected by showing that $\s_{q}(\KGL)$ is a retract of $\s_{q}(\KQ)$ if
$q$ is even.
We deduce that $\s_{q}(\KQ)$ is a wedge sum of $\Sigma^{2q,q}\MZZ$ and some $\MZZ$-module, 
i.e., a motive, 
which we identify with the infinite wedge summand in (\ref{equation:hermitianktheoryslices}). 
Our first results show there is an additional ``mysterious summand'' $\Sigma^{2q,q}\Mmu$ of $\s_{q}(\KQ)$.
We show $\Mmu$ is trivial by using base change and the solution of the homotopy fixed point problem for hermitian $K$-theory of the
prime fields \cite{BK}, \cite{Friedlander}, 
cf.~\cite{BKSO}, \cite{Hu-Kriz-Ormsby}.
As conjectured in Hornbostel's foundational paper \cite{Hornbostel}, 
Theorem \ref{thm:cofiber-seq} shows there is a homotopy cofiber sequence relating the algebraic and hermitian $K$-theories:
\begin{equation}
\label{equation:Hopfsequence}
\Sigma^{1,1}\KQ
\rTo^{\eta}
\KQ
\rTo
\KGL
\end{equation}
The stable Hopf map $\eta$ is induced by the canonical map $\AA^{2}\smallsetminus\{0\}\rightarrow\PP^{1}$.
We show this over any finite dimensional regular noetherian base scheme $S$ equipped with the trivial involution and with $2$ invertible 
in its ring of regular functions, 
i.e., $\frac{1}{2}\in\Gamma(S,\mathcal{O}_{S})$. 
A closely related statement is obtained in \cite{Schlichting}.
The sequence (\ref{equation:Hopfsequence}) is employed in our computation of the slices of $\KQ$.
\vspace{0.1in}

The algebraic $K$-theory spectrum $\KGL$ affords an action by the stable Adams operation $\Psi^{-1}$. 
For the associated homotopy orbit spectrum $\KGL_{hC_{2}}$ there is a homotopy cofiber sequence
\begin{equation}
\label{equation:KGLorbitsKOKTsequence}
\KGL_{hC_{2}}
\rTo
\KQ
\rTo
\KT.
\end{equation}
In (\ref{equation:KGLorbitsKOKTsequence}), 
$\KGL_{hC_{2}}\rightarrow\KQ$ is induced by the hyperbolic map $\KGL\rightarrow\KQ$, 
while $\KQ\rightarrow\KT$ is the natural map from hermitian $K$-theory into the homotopy colimit of the tower
\begin{equation}
\label{equation:KQKTtower}
\KQ
\rTo^{\eta}
\Sigma^{-1,-1}\KQ
\rTo^{\Sigma^{-1,-1}\eta}
\Sigma^{-2,-2}\KQ
\rTo^{\Sigma^{-2,-2}\eta}
\cdots .
\end{equation}
We use the formulas $\s_{0}(\Psi^{-1})=\id$ and $\Sigma^{2,1}\Psi^{-1}=-\Psi^{-1}$ to identify the slices 
\begin{equation}
\label{equation:homotopyorbitktheoryslices}
\s_{q}(\KGL_{hC_{2}})
\cong
\begin{cases}
\Sigma^{2q,q}\MZZ
\vee  
\bigvee_{i\geq 0}\Sigma^{2(i+q)+1,q}\MZZ/2
&
q\equiv 0\bmod 2 \\
\bigvee_{i\geq 0} \Sigma^{2(i+q),q}\MZZ/2 
&
q\equiv 1\bmod 2. \\
\end{cases}
\end{equation}
By combining the slice computations in (\ref{equation:hermitianktheoryslices}) and (\ref{equation:homotopyorbitktheoryslices}) with the 
homotopy cofiber sequence in (\ref{equation:KGLorbitsKOKTsequence}) we deduce the identification of the slices of the Witt-theory spectrum 
$\KT$ in (\ref{equation:wittheoryslices}).
Alternatively, 
this follows from (\ref{equation:hermitianktheoryslices}), (\ref{equation:KQKTtower}), 
and Spitzweck's result that slices commutes with homotopy colimits \cite{Spitzweck:slices}.
\vspace{0.1in}

Our next goal is to determine the first differentials in the slice spectral sequences as maps of motivic spectra.
Because of the special form the slices of $\KQ$ and $\KT$ have, 
this involves the motivic Steenrod squares $\Sq^{i}$ constructed by Voevodsky \cite{Voevodsky:Z/2} and further elaborated on in \cite{HKPAO}.
According to (\ref{equation:wittheoryslices}) the differential
\[ 
\sliced_{1}^{\KT}(q)
\colon 
\s_{q}(\KT) 
\rTo
\Sigma^{1,0}\s_{q+1}(\KT) 
\]
is a map of the form
\[ 
\bigvee_{i\in\ZZ}\Sigma^{2i+q,q}\MZZ/2 
\rTo 
\Sigma^{2,1}\bigvee_{j\in\ZZ}\Sigma^{2j+q,q}\MZZ/2. 
\]
Let $\sliced_{1}^{\KT}(q,i)$ denote the restriction of $\sliced_{1}^{\KT}(q)$ to the $i$th summand $\Sigma^{2i+q,q}\MZZ/2$ of $\s_{q}(\KT)$. 
By comparing with motivic cohomology operations of weight one, 
it suffices to consider 
\[ 
\sliced_{1}^{\KT}(q,i)
\colon 
\Sigma^{2i+q,q}\MZZ/2 
\rTo 
\Sigma^{2i+q+4,q+1}\MZZ/2 
\vee
\Sigma^{2i+q+2,q+1}\MZZ/2
\vee
\Sigma^{2i+q,q+1}\MZZ/2. 
\]
In Theorem \ref{thm:diff-witt-theory} we show the closed formula
\begin{equation}
\label{equation:d1KTdifferentials}
\sliced_{1}^{\KT}(q,i) 
= 
\begin{cases} 
(\Sq^{3}\Sq^{1},\Sq^{2},0) & i-2q\equiv 0\bmod 4 \\
(\Sq^{3}\Sq^{1},\Sq^{2}+\rho\Sq^{1},\tau) & i-2q \equiv 2\bmod 4.
\end{cases} 
\end{equation}

The classes $\tau\in h^{0,1}$ and $\rho\in h^{1,1}$ are represented by 
$-1\in F$; here $h^{p,q}$ is shorthand for the mod-$2$ motivic 
cohomology group in degree $p$ and weight $q$.
We denote integral motivic cohomology groups by $H^{p,q}$.
This sets the stage for our proof of Milnor's conjecture on quadratic forms formulated in \cite[Question 4.3]{Milnor}.
For fields of characteristic zero this conjecture was shown by 
Orlov, Vishik and Voevodsky in \cite{Orlov-Vishik-Voevodsky},
and by Morel \cite{MorelCRAS}, \cite{MorelRSMUP} using different approaches.
\vspace{0.1in}

According to (\ref{equation:wittheoryslices}) the slice spectral sequence for $\KT$ fills out the upper half-plane.
A strenuous computation using (\ref{equation:d1KTdifferentials}), 
Adem relations, 
and the action of the Steenrod squares on the mod-$2$ motivic cohomology ring $h^{\star}$ of $F$ shows that it collapses.
We read off the isomorphisms
\begin{equation*}
E^{2}_{p,q}(\KT)
=
E^{\infty}_{p,q}(\KT)
\cong
\begin{cases}
h^{q,q}
&
p\equiv 0\bmod 4 \\
0
&
\text{otherwise}. \\
\end{cases}
\end{equation*}
To connect this computation with quadratic form theory, 
we show the spectral sequence converges to the filtration of the Witt ring $W(F)$ by the powers of the fundamental ideal $I(F)$ 
of even dimensional forms.
By identifying motivic cohomology with Galois cohomology for fields we arrive at the following result.

\begin{theorem}
\label{thm:main}
If $\Char(F)\neq 2$ the slice spectral sequence for $\KT$ converges and furnishes a complete set of invariants 
\[ 
\overline{e}^{q}_{F}
\colon
I(F)^{q}/I(F)^{q+1}
\rTo^\cong
H^{q}(F;\ZZ/2)
\]
for quadratic forms over $F$ with values in the mod-$2$ Galois cohomology ring.
\end{theorem}

If $X\in\Sm_F$ is a semilocal scheme and $F$ a field of characteristic zero, 
our computations and results extend to the Witt ring $W(X)$ with fundamental ideal $I(X)$ and the mod-$2$ motivic cohomology of $X$.
Our reliance on the Milnor conjecture for Galois cohomology \cite{Voevodsky:Z/2} can be replaced by Levine's generalized Milnor conjecture on \'etale cohomology 
of semilocal rings \cite{Levine:milnorK},  
as shown in \cite[\S2.2]{Hoobler} and \cite[Theorem 7.8]{Kerz}, 
cf.~\cite{Lichtenbaum}.  
We refer to the end of \S\ref{section:milnorconjecture} for further details.

In Section \ref{section:hk-groups} we perform computations in the slice spectral sequence for $\KQ$ in low degrees.
We formulate our computation of the second orthogonal $K$-group, 
and refer to the main body of the paper for the more complicated to state results on other hermitian $K$-groups.
\begin{theorem}
\label{thm:mainhermitian}
If $\Char(F)\neq 2$ there is a naturally induced isomorphism
\[ 
KO_{2}(F)
\rTo^{\cong}
\mathrm{ker}(\tau\circ\pr+\Sq^{2}\colon H^{2,2}\directsum h^{0,2}\rTo h^{2,3}).
\]
\end{theorem}

Throughout the paper we employ the following notation.

\begin{tabular}{l|l}
$S$ & finite dimensional regular and separated noetherian base scheme \\
$\Sm_{S}$ & smooth schemes of finite type over $S$ \\
$S^{m,n}$, $\Omega^{m,n}$, $\Sigma^{m,n}$ & motivic $(m,n)$-sphere, $(m,n)$-loop space, $(m,n)$-suspension  \\
$\SHH$, $\SHH^\mathrm{eff}$ & the motivic and effective motivic stable homotopy categories \\ 
$\E$, $\One=S^{0,0}$  & generic motivic spectrum, the motivic sphere spectrum \\
\end{tabular}
\vspace{0.1in}

{\bf Acknowledgements:}  We gratefully acknowledge both the editor and the referee for the two detailed and helpful referee reports on this paper.

\section{Slices and the slice spectral sequence}
\label{section:slices}
Let $i_{q}\colon \Sigma^{2q,q} \SHH^\mathrm{eff}\rInto\SHH$ be the full inclusion of the localizing subcategory generated by 
$\Sigma^{2q,q}$-suspensions of smooth schemes. 
We denote by $r_{q}$ the right adjoint of $i_{q}$ and set $\f_{q}= i_{q}\circ r_{q}$. 
The $q$th slice of $\E$ is characterized up to unique isomorphism by the distinguished triangle 
\begin{equation}\label{eq:slice-def}
\f_{q+1}(\E) 
\rTo 
\f_{q}(\E) 
\rTo 
\s_{q}(\E) 
\rTo 
\Sigma^{1,0}  \f_{q+1}(\E)
\end{equation}
in $\SHH$ \cite{Voevodsky:open}.
When it is helpful to emphasize the base scheme $S$ we shall write $\f^{S}_{q}(\E)$ and $\s^{S}_{q}(\E)$.
Smashing with the motivic sphere $\Sigma^{1,1}$ has the following effect on the slice filtration.
\begin{lemma}
\label{lemma:shift-slices}
For all $q\in\ZZ$ there are natural isomorphisms
\[ 
\f_{q}(\Sigma^{1,1}  \E)
\rTo^{\iso} 
\Sigma^{1,1}  \f_{q-1}(\E)
\text{ and }
\s_{q}(\Sigma^{1,1}  \E) 
\rTo^{\iso} 
\Sigma^{1,1}  \s_{q-1}(\E). 
\]
which are compatible with the natural transformations
occurring in~(\ref{eq:slice-def}).
\end{lemma}

\begin{proof}
  Let $m$ and $q$ be integers. The suspension
  functor $\Sigma^{m,m}$ restricts to a functor 
  \[ \Sigma^{m,m}_q - \colon \Sigma^{2q,q} \SHH^\mathrm{eff}
  \rTo \Sigma^{2(q+m),q+m} \SHH^\mathrm{eff} \]
  satisfying $ \Sigma^{m,m}  i_q(\E) = i_{q+m} \circ (\Sigma^{m,m}_q \E)$
  for all $\E\in \Sigma^{2q,q} \SHH^\mathrm{eff}$. This equality induces
  a unique natural isomorphism on the respective right adjoints:
  \[ r_q\bigl(\Sigma^{-m,-m}  \E\bigr)\iso \Sigma^{-m,-m}_{q+m} r_{q+m}(\E) \]
  In particular, there results a natural isomorphism:
  \[ \f_q(\Sigma^{1,1}  \E)=i_q r_q(\Sigma^{1,1}  \E) \iso
  i_q\Sigma^{1,1}_{q-1} r_{q-1}(\E) = \Sigma^{1,1}  i_{q-1}r_{q-1}(\E)
  = \Sigma^{1,1}  \f_q(\E) \]
  In order to conclude the same for $\s_q$, observe 
  that the inclusion $i_{q+1}$ factors as $i_q\circ i^{q}_{q+1}$, where
  \[ i^q_{q+1}\colon \Sigma^{2(q+1),q+1} \SHH^\mathrm{eff} \rTo 
  \Sigma^{2q,q} \SHH^\mathrm{eff} \]
  is the natural inclusion. The functor $i^q_{q+1}$ has a right 
  adjoint $r^q_{q+1}$ for the same reason that $i_q$ does, and
  the natural transformation $\f_{q+1}\rTo \f_q$ is obtained
  from the counit $i^q_{q+1}\circ r^q_{q+1} \rTo \Id$. As before, the
  equality 
  $\Sigma^{m,m}_q i^q_{q+1}(\E) = i^{q+m}_{q+m+1} \circ (\Sigma^{m,m}_{q+1} \E)$
  induces a unique natural isomorphism on right adjoints, which
  serves to show that the diagram
  \begin{diagram}
    \f_{q+1}(\Sigma^{1,1}  \E) & \rTo^\iso & \Sigma^{1,1}  \f_q(\E) \\
    \dTo & & \dTo \\
    \f_{q}(\Sigma^{1,1}  \E) & \rTo^\iso & \Sigma^{1,1}  \f_{q-1}(\E) 
  \end{diagram}
  commutes. The remaining statements follow.
\end{proof}

Let $\eta\colon S^{1,1}\rTo \One$ denote the Hopf map induced by the canonical map $\AA^2\minus\{0\}\rTo \PP^1$. 
Since every motivic spectrum $\E$ is a module over the motivic sphere spectrum $\One$, 
multiplication with $\eta$ defines, 
by abuse of notation, 
a map $\eta\colon \Sigma^{1,1}  \E\rTo \E$.

\begin{lemma}\label{lem:hopf-slices}
For every $\E$ and $q\in \ZZ$ there is a naturally induced commutative diagram:
\begin{diagram}
\f_{q+1}(\Sigma^{1,1}  \E) & \rTo^\iso & \Sigma^{1,1}  \f_q(\E) \\
\dTo^{\f_{q+1}(\eta)} & & \dTo_{\eta} \\
\f_{q+1}\E & \rTo & \f_q \E
\end{diagram}
\end{lemma}

\begin{proof}
  This follows from Lemma \ref{lemma:shift-slices} by naturality.
\end{proof}

\begin{example}
\label{example:KT}
The Hopf map induces a periodicity isomorphism $\eta\colon \Sigma^{1,1}  \KT\rTo^\iso \KT$, 
cf.~the definition \eqref{KTdef}.
In particular, 
the vertical map on the left hand side in the diagram of Lemma~\ref{lem:hopf-slices} is an isomorphism.
It also implies the isomorphism $\s_{q}(\KT)\cong \Sigma^{q,q}  \s_{0}(\KT)$.
\end{example}

The slices $\s_{q}(\E)$ are modules over the motivic ring spectrum $\s_{0} (\One)$,  
cf.~\cite[Theorem 3.6.22]{Pelaez} and \cite[\S 6(v)]{GRSO}. 
If the base scheme $S$ is a perfect field, 
then $\s_{0}(\One)$ is the Eilenberg-MacLane spectrum $\MZZ$ by the works of Levine \cite{Levine:slices} and Voevodsky \cite{Voevodsky:zero-slice}.
We set $D^{1}_{p,q,n}=\pi_{p,n}\f_{q}(\E)$ and $E^{1}_{p,q,n}=\pi_{p,n}\s_{q}(\E)$.
The exact couple 
\begin{diagram}[width=2cm,height=0.7cm]
D^{1}_{\ast} && \rTo^{(0,-1,0)} && D^{1}_{\ast}\\
& \luTo_{(-1,1,0)} && \ldTo_{(0,0,0)} & \\
&& E^{1}_{\ast} &&
\end{diagram}
gives rise to the slice spectral sequence
\begin{equation}
\label{equation:splicespectralsequence2}
E^{1}_{p,q,n} 
\Longrightarrow
\pi_{p,n}(\E).
\end{equation}
Our notation does not connote any information on the convergence of (\ref{equation:splicespectralsequence2}). 
The $d_{1}$-differential 
\[  
d_{1}^{\E}(p,q,n)
\colon
\pi_{p,n}\s_{q}(\E) 
\rTo 
\pi_{p-1,n}\s_{q+1}(\E) 
\]
of (\ref{equation:splicespectralsequence2}) is induced on homotopy groups
$\pi_{p,n}$ by the composite map
\[ 
\sliced_{1}^{\E}(q)
\colon
\s_{q}(\E) 
\rTo 
\Sigma^{1,0}  \f_{q+1}(\E)
\rTo 
\Sigma^{1,0}  \s_{q+1}(\E)
\] 
of motivic spectra. The $r$th differential has tri-degree $(-1,r,0)$.
By construction of the $q$th slice, 
if $n>q$ the maps in 
\begin{equation}
\label{eq:seq-filt}
\dotsm 
\rTo 
\pi_{p,n}\f_{q+1}(\E) 
\rTo 
\pi_{p,n}\f_{q}(\E) 
\rTo 
\dotsm 
\end{equation}
are isomorphisms, 
and $\pi_{p,n}\s_{q}(\E)$ is trivial.
Thus only finitely many nontrivial differentials enter each tri-degree, 
so that (\ref{equation:splicespectralsequence2}) is a half-plane spectral sequence with entering differentials.
Let $\f_{q}\pi_{p,n}(\E)$ denote the image of $\pi_{p,n}\f_{q}(\E)$ in $\pi_{p,n}(\E)$.
The terms $\f_{q}\pi_{p,n}(\E)$ form an exhaustive filtration of $\pi_{p,n}(\E)$.
Moreover, 
$\E$ is called convergent with respect to the slice filtration if 
\[ 
\bigcap_{i\geq 0} 
\f_{q+i}\pi_{p,n}\f_{q}(\E)
=
0 
\]
for all $p,q,n\in\ZZ$ \cite[Definition 7.1]{Voevodsky:open}. 
That is, the filtration $\{\f_{q}\pi_{p,n}(\E)\}$ of $\pi_{p,n}(\E)$ is Hausdorff.
When $\E$ is convergent, 
the spectral sequence (\ref{equation:splicespectralsequence2}) converges \cite[Lemma 7.2]{Voevodsky:open}.
Precise convergence properties of the slice spectral sequence are unclear in general \cite{Levine:sss}.
\vspace{0.1in}

When comparing slices along field extensions we shall appeal to the following base change property of the slice filtration.
\begin{lemma}
\label{lem:slice-smooth-base-change}
Let $\alpha\colon X\rTo Y$ be a smooth map. 
For $q\in\ZZ$ there are natural isomorphisms
\[ 
\f_{q}^{X} \alpha^{\ast} 
\rTo^{\iso} 
\alpha^{\ast} \f_{q}^{Y},  
\;
\s_{q}^{X} \alpha^{\ast} 
\rTo^{\iso} 
\alpha^{\ast} \s_{q}^{Y}.
\]
\end{lemma}
\begin{proof}
Any map $\alpha\colon X\rTo Y$ between base schemes yields a commutative diagram: 
\begin{diagram}
\Sigma^{2q,q}  \SHH^{\eff}(Y) & 
\rTo^{\alpha^{\ast}} & 
\Sigma^{2q,q} \SHH^{\eff}(X) \\
\dTo^{i_{q}^{Y}} & & \dTo_{i_{q}^{X}} \\
\SHH(Y) & 
\rTo^{\alpha^{\ast}} & 
\SHH(X) 
\end{diagram}
Let $\alpha_{\ast}$ be the right adjoint of $\alpha^{\ast}$.  
By uniqueness of adjoints, 
up to unique isomorphism, 
there exists a natural isomorphism of triangulated functors
\[ 
r^{Y}_{q} \alpha_{\ast} 
\rTo^{\iso} \alpha_{\ast} r_{q}^{X}.  
\]
Since $\alpha$ is smooth,  
the functor $\alpha^{\ast}$ has a left adjoint $\alpha_\sharp$ and there is a commutative diagram:
\begin{diagram}
\Sigma^{2q,q}  \SHH^{\eff}(X) & 
\rTo^{\alpha^\sharp} & 
\Sigma^{2q,q} \SHH^{\eff}(Y) \\
\dTo^{i_{q}^{X}} & & \dTo_{i_{q}^{Y}}\\  
\SHH(X) & \rTo^{\alpha^\sharp} & \SHH(Y)
\end{diagram}
By uniqueness of adjoints,  
there is an isomorphism
\[ 
r_{q}^{X} \alpha^{\ast} 
\rTo^{\iso}
\alpha^{\ast} r_{q}^{Y}, 
\]
and hence
\begin{equation}
\label{equation:fbasechange}
\f_{q}^{X} \alpha^{\ast} 
= 
r_{q}^{X} i_{q}^{X} \alpha^{\ast} 
= 
r_{q}^{X} \alpha^{\ast} i_{q}^{Y} 
\iso 
\alpha^{\ast} r_{q}^{Y} i_{q}^{Y} 
= 
\alpha^{\ast}\f_{q}^{Y}.
\end{equation}
The desired isomorphism for slices follows since, 
by uniqueness of adjoints, 
(\ref{equation:fbasechange}) is compatible with the natural transformation $\f_{q+1}\rTo\f_{q}$. 
\end{proof}

\begin{theorem}
\label{thm:slice-ess-smooth-base-change}
Suppose $\mathcal{I}$ is a filtered partially ordered set and
$
D
\colon 
\mathcal{I}^{\op}
\rTo 
\Sm_{Y}, 
i
\rMapsto X_{i}
$
is a diagram of $Y$-schemes with affine bonding maps.
Let $\alpha\colon X\equiv\lim_{i\in \mathcal{I}} X_{i}\rTo Y$ be
the naturally induced morphism. 
For every $q\in \ZZ$ there are natural isomorphisms
\[  
\s_{q}^{X} \alpha^{\ast} 
\rTo^{\iso} 
\alpha^{\ast} \s_{q}^{Y}
\text{ and }
\f_{q}^{X} \alpha^{\ast} 
\rTo^{\iso} 
\alpha^{\ast} \f_{q}^{Y}.
\]
\end{theorem}
\begin{proof}
Let $\mathcal{G}$ be a compact generator of the triangulated category $\Sigma^{2q+2,q+1} \SHH^{\eff}(X)$ and $\E\in\SHH(Y)$.
According to \cite[Theorem 2.12]{Pelaez:base} it suffices to show $[\mathcal{G},\alpha^{\ast}\s_{q}^{Y}(\E)]$ is trivial. 
Note that $\mathcal{G}$ is of the form $\alpha(j)^{\ast}\mathcal{G}_{j}$ for $\mathcal{G}_{j}$ a compact generator of the triangulated 
category $\Sigma^{2q+2,q+1} \SHH^{\eff}(X_{j})$. 
We consider the overcategory of an arbitrary element $j\in\mathcal{I}$ and the composite diagram
\[
D\downarrow j
\colon  
\mathcal{I}^{\op}\downarrow j
\rTo 
\mathcal{I}^{\op} 
\rTo^{D} 
\Sm_{Y}, 
(j\rightarrow i)
\mapsto 
X_{i}.
\]
We claim the functor $\Phi\colon\mathcal{I}^{\op}\downarrow j\rTo\mathcal{I}^{\op}$ yields an isomorphism 
$\colim D\downarrow j\rTo\colim D$.
For $i\in\mathcal{I}$ there exists an object $i^{\prime}$ and maps $i\rightarrow i^{\prime} \leftarrow j$, 
since $\mathcal{I}$ is filtered. 
Thus $\Phi\downarrow i$ is nonempty.
For zig-zags $i\rightarrow i^{\prime}\leftarrow j$ and $i\rightarrow i^{\prime\prime}\leftarrow j$ there exist maps 
$i^{\prime}\rightarrow\ell\leftarrow i^{\prime\prime}$ such that the induced maps from $i$ to $\ell$ coincide, 
and similarly for $j$ ($\mathcal{I}$ is a partially ordered set.)
Hence $\Phi\downarrow i$ is connected. 
For $j\rightarrow i$ in $\mathcal{I}^{\op}\downarrow j$, 
let $e(i)\colon X_{i}\rTo X_{j}$ be the structure map of the diagram. 
Thus the map
\[ 
\hocolim_{j\downarrow \mathcal I} e(i)_{\ast} e(i)^{\ast}\F
\rTo 
\alpha(j)_{\ast} \alpha(j)^{\ast}\F 
\]
is an isomorphism in $\SHH(X_{j})$ for all $\F\in \SHH(X_{j})$. 
Since $\mathcal{G}_{j}$ is compact and slices commute with base change along smooth maps by Lemma \ref{lem:slice-smooth-base-change},
the group
\begin{align*}
[\mathcal{G},\alpha^{\ast} \s_{q}^{Y}(\E)] 
& = 
[\alpha(j)^{\ast} \mathcal{G}_{j},\alpha^{\ast} \s_{q}^{Y}(\E)]  \\
& \cong 
[\mathcal{G}_{j},\alpha(j)_{\ast} \alpha(j)^{\ast} \beta(j)^{\ast} \s_{q}^{Y}(\E)] \\
& \cong
[\mathcal{G}_{j},\hocolim_{j\downarrow \mathcal{I}} e(i)_{\ast} e(i)^{\ast} \beta(j)^{\ast} \s_{q}^{Y}(\E)] \\
& \cong 
\colim_{j\downarrow \mathcal{I}} [e(i)^{\ast} \mathcal{G}_{j},e(i)^{\ast} \beta(j)^{\ast} \s_{q}^{Y}(\E)]  
\end{align*}
is trivial.
With reference to \cite[Remark 2.13]{Pelaez:base} the second isomorphism follows in the same way.
\end{proof}

\begin{lemma}
\label{lemma:slice-ess-smooth-base-change}
If $\alpha\colon A\rTo B$ is a regular ring map 
then $\s_{q}^{B}\alpha^{\ast}\iso\alpha^{\ast}\s_{q}^{A}$.
\end{lemma}
\begin{proof}
By Popescu's general N{\'e}ron desingularization theorem \cite{Popescu} the regularity assumption on $\alpha$ is 
equivalent to $B$ being the colimit of a filtered diagram of smooth $A$-algebras of finite type.
The result follows now from Theorem~\ref{thm:slice-ess-smooth-base-change}.
\end{proof}
\begin{corollary}\label{cor:slice-ess-smooth-base-change}
If $\alpha\colon F\rTo E$ is a separable field extension then $\s_{q}^{E}\alpha^{\ast}\iso\alpha^{\ast}\s_{q}^{F}$.
\end{corollary}

\section{Algebraic and hermitian $K$-theory}
Let $\KGL=(K,\dotsc)$ denote the motivic spectrum representing algebraic $K$-theory \cite[\S 6.2]{Voevodsky:icm} over the base scheme $S$.
Here, 
$K\colon\Sm_{S}\rTo\Spt$ sends a smooth $S$-scheme to its algebraic $K$-theory spectrum.
The structure maps of $\KGL$ are given by the Bott periodicity operator
\[
\beta
\colon 
K
\rTo^{\sim} 
\Omega^{2,1}K.
\]
Suppose  $2$ is invertible in the ring of regular functions on $S$. 
Let $KO\colon \Sm_{S}\rTo \Spt$ denote the functor sending a smooth $S$-scheme to the (non-connective) spectrum representing the 
hermitian $K$-groups for the trivial involution on $S$ and sign of symmetry $\varepsilon=+1$. 
Similarly, 
when $\varepsilon = -1$, 
we use the notation $KSp\colon \Sm_{S}\rTo \Spt$.

There are  natural maps $f_{0}\colon KO \rTo K$ and $f_{2}\colon KSp\rTo K$ induced by the forgetful functors.
Taking the homotopy fibers of these forgetful maps yields the following
homotopy fiber sequences:
\begin{equation}\label{eq:v}
\begin{diagram}[width=1.3cm]
\Omega^{1,0}K& \rTo^{h_3^\prime} & VQ &\rTo^\can & KO & \rTo^{f_0} & K \\ 
\Omega^{1,0}K& \rTo^{h_1^\prime} & VSp &\rTo^\can & KSp & \rTo^{f_2} & K 
\end{diagram}
\end{equation}
Moreover, 
there are natural maps
\begin{diagram}
h_{0}\colon K & \rTo & KO 
\text{ and }
h_{2}\colon K & \rTo & KSp
\end{diagram}
induced by the hyperbolic functors.
Taking the homotopy fibers of these hyperbolic maps yields the following
homotopy fiber sequences:
\begin{equation}\label{eq:u}
\begin{diagram}[width=1.3cm]
\Omega^{1,0}KO& \rTo^\can & UQ &\rTo^{f_3} & K & \rTo^{h_0} & KO \\ 
\Omega^{1,0}KSp& \rTo^\can & USp &\rTo^{f_1} & K & \rTo^{h_2} & KSp 
\end{diagram}
\end{equation}
Karoubi's fundamental theorem in hermitian $K$-theory \cite{Karoubi} can be formulated as follows.
\begin{theorem}[Karoubi]
\label{thm:master}
There are natural weak equivalences
\begin{diagram} 
\phi
\colon 
\Omega^{1,0} USp & \rTo^{\sim} & VQ 
\text{ and }
\psi
\colon 
\Omega^{1,0} UQ & \rTo^{\sim} & VSp.
\end{diagram}
\end{theorem}

In their foundational paper on localization in hermitian $K$-theory of rings \cite[Section 1.8]{Hornbostel-Schlichting}, 
Hornbostel-Schlichting show the following result.
\begin{theorem}[Hornbostel-Schlichting]
\label{thm:hornbostel-schlichting}
The homotopy cofiber of the map 
\[ 
KO 
\rTo 
KO(\AA^{1}\minus \{0\}\times_{S} -)
\quad
\mathrm{resp.}
\quad
KSp 
\rTo 
KSp(\AA^{1}\minus \{0\}\times_{S} -)
\] 
induced by the map $\AA^{1}\minus \{0\}\rTo S$ is naturally weakly equivalent to $\Sigma^{1,0}UQ$ resp.~$\Sigma^{1,0}USp$.
\end{theorem}

The homotopy cofiber sequences in Theorem \ref{thm:hornbostel-schlichting} split by the unit section $1\in \AA^1\minus\{0\}(S)$.
Since $\Omega^{1,1}$ is the homotopy fiber with respect to the unit section, 
Theorem~\ref{thm:hornbostel-schlichting} implies the natural weak equivalences
\[ 
\Omega^{2,1}KO \sim \Omega^{1,0}\Sigma^{1,0}UQ\lTo^\sim UQ
\quad \mathrm{and} \quad
\Omega^{2,1}KSp \sim \Omega^{1,0}\Sigma^{1,0}USp\lTo^\sim USp.
\]
In \cite{Hornbostel} Hornbostel shows that hermitian $K$-theory is represented by the motivic spectrum 
\[
\KQ
=
(KO,USp,KSp,UQ,KO,USp,\dotsc).
\]
Our notation emphasizes the connection between hermitian $K$-theory and quadratic forms.
The structure maps of $\KQ$ are the adjoints of the weak equivalences:
\begin{equation}
\label{equation:adjointweakequivalences}
\begin{diagram}[height=1cm]
KO & \rTo^{\sim}  & \Omega^{2,1} USp & & USp & \rTo^{\sim} & \Omega^{2,1}KSp \\
KSp & \rTo^{\sim}  & \Omega^{2,1} UQ & & UQ & \rTo^{\sim} & \Omega^{2,1}KO 
\end{diagram}
\end{equation}

\begin{proposition}
There are commutative diagrams:
\begin{diagram}
KO & \rTo^{f_{0}} & K  & &   USp & \rTo^{f_{1}} &  K \\
\dTo_{\sim} & & \dTo_{\beta}^{\sim} && \dTo_{\sim} & & \dTo_{\beta}^{\sim} \\ 
\Omega^{2,1}USp & \rTo^{\Omega^{2,1}f_{1}} & \Omega^{2,1}K & 
& \Omega^{2,1}KSp & \rTo^{\Omega^{2,1}f_{2}} & \Omega^{2,1}K \\
KSp & \rTo^{f_{2}} & K  & &   UQ & \rTo^{f_{3}} &  K \\
\dTo_{\sim} & & \dTo_{\beta}^{\sim} && \dTo_{\sim} & & \dTo_{\beta}^{\sim} \\ 
\Omega^{2,1}UQ & \rTo^{\Omega^{2,1}f_{3}} & \Omega^{2,1}K & 
& \Omega^{2,1}KO & \rTo^{\Omega^{2,1}f_{0}} & \Omega^{2,1}K \\
\end{diagram}
\end{proposition}

The forgetful map $f\colon \KQ\rTo \KGL$ is given by the sequence
$(f_0,f_1,f_2,f_3,f_0,\dotsc)$ of maps of pointed motivic spaces
displayed in diagrams~(\ref{eq:v}) and~(\ref{eq:u}).
Similarly, the hyperbolic map $h\colon \KGL\rTo \KQ$ is given by the
sequence $(h_0,h_1,h_2,h_3,h_0,\dotsc)$. Here $h_0$ and $h_2$ have
been introduced before, and $h_1$ and $h_3$ are defined by the
weak equivalences from Theorem~\ref{thm:master} and the
canonical maps $h_1^\prime$ and $h_3^\prime$ introduced 
with the construction
of $VQ$ and $VSp$ in diagram~(\ref{eq:v}).
By inspection of the structure maps of $\KQ$ determined by (\ref{equation:adjointweakequivalences}), 
devising a map $\KQ \rTo \Omega^{1,1} \KQ$ of motivic spectra is tantamount to giving a compatible sequence of maps between motivic spaces
\[ 
KO \rTo \Sigma^{1,0}UQ, 
\;\;
USp \rTo \Sigma^{1,0} KO,
\;\;
KSp \rTo \Sigma^{1,0}USp,
\;\;
UQ \rTo \Sigma^{1,0} KSp,\dotsc . 
\]

Next we show the homotopy cofiber sequence (\ref{equation:Hopfsequence}) relating algebraic and hermitian $K$-theory via 
the stable Hopf map, 
as reviewed in the introduction. 
A closely related statement is obtained in \cite[Theorem 6.1]{Schlichting}.
\begin{theorem}
\label{thm:cofiber-seq}
The stable Hopf map and the forgetful map yield a homotopy cofiber sequence 
\[ 
\Sigma^{1,1} \KQ
\rTo^{\eta} 
\KQ 
\rTo^{f}
\KGL \rTo \Sigma^{2,1}  \KQ. 
\]
The connecting map factors as $\KGL \rTo^{\iso} \Sigma^{2,1} \KGL\rTo^{\Sigma^{2,1}  h} \Sigma^{2,1} \KQ$.
\end{theorem}

\begin{proof}
We show the map $\KQ\rTo\Omega^{1,1} \KQ$ induced by $\eta$ is determined by the canonical maps 
\[ 
KO\rTo\Sigma^{1,0}UQ
\quad 
USp\sim\Sigma^{1,0}VQ\rTo\Sigma^{1,0} KO 
\]
\[  
KSp\rTo\Sigma^{1,0}USp 
\quad 
UQ\sim\Sigma^{1,0}VSp\rTo\Sigma^{1,0} KSp. 
\]
To that end, 
it suffices to describe the maps
\[
KO(\PP^{1})\rTo KO(\AA^{2}\minus\{0\})  
\quad 
USp(\PP^{1})\rTo USp(\AA^{2}\minus\{0\}) 
\]
\[
KSp(\PP^{1})\rTo KSp(\AA^{2}\minus\{0\})  
\quad 
UQ(\PP^{1})\rTo UQ(\AA^{2}\minus \{0\}) 
\]
induced by the unstable Hopf map $\AA^{2}\minus \{0\}\rTo \PP^{1}$, 
or equivalently, 
see \cite[p.~73 in \S 3.3 and Example 7.26]{Morel:2052},
by the Hopf construction applied to the map
\[
\Upsilon
\colon 
(\AA^{1}\minus \{0\})
\times_{S} 
(\AA^{1}\minus \{0\}) 
\rTo 
\AA^{1}\minus \{0\}, 
(x,y)\rMapsto xy^{-1}.
\]
We note there is an isomorphism of schemes
\[
\theta\colon 
(\AA^{1}\minus\{0\})
\times_{S} 
(\AA^{1}\minus\{0\}) 
\rTo 
(\AA^{1}\minus \{0\})
\times_{S}  
(\AA^{1}\minus \{0\}), 
(x,y)\rMapsto (xy,y).
\]
Now the composite $\Upsilon\theta$ is the projection map on the first factor. 
Thus for $KO$, $USp$, $KSp$ and $UQ$, 
the homotopy cofibers of the maps induced by $\Upsilon$ and the projection map coincide up to weak equivalence. The latter homotopy cofibers
are given in Theorem~\ref{thm:hornbostel-schlichting}.
\end{proof}

\begin{lemma}
\label{lem:c-ring-map}
The unit map $\One \rTo \KGL$ factors as $\One\rTo \KQ\rTo^f\KGL$.
\end{lemma}

\begin{proof}
The unit map $\One\rTo \KGL$ is given by the trivial line bundle over the base scheme $S$.  
The latter is obtained by forgetting the standard nondegenerate quadratic form on it, 
which provides the factorization.
\end{proof}

Let $\epsilon\colon\One\rTo\One$ be the endomorphism of the sphere spectrum induced by the commutativity isomorphism 
on the smash product $S^{1,1}\smash  S^{1,1}$.
\begin{lemma}
\label{lemma:multiplicationbyfh}
The composition $\KQ\rTo^{f}\KGL\rTo^{h}\KQ$ coincides with multiplication by $1-\epsilon$.
\end{lemma}
\begin{proof}
Since $1-\epsilon = 1+\langle-1\rangle$ is the hyperbolic plane \cite[p.~53]{Morel:2052} the unit map for $\KQ$ induces a commutative diagram:  
\begin{diagram}
\One & \rTo & \KQ \\
\dTo^{1-\epsilon} & & \dTo_{h\circ f} \\
\One & \rTo & \KQ
\end{diagram}
By smashing with $\KQ$ and employing its multiplicative structure we obtain the diagram:
\begin{diagram}
\KQ & \rTo & \KQ\smash \KQ & \rTo & \KQ\\
\dTo^{\KQ\smash (1-\epsilon)} & & \dTo
& & \dTo_{h\circ f} \\
\KQ & \rTo & \KQ \smash \KQ & \rTo & \KQ
\end{diagram}
The middle vertical map is $\KQ\smash (h\circ f)$, 
while the two composite horizontal maps coincide with the identity on $\KQ$. 
The right hand square commutes because $h\circ f$ is a map of $\KQ$-modules. 
\end{proof}

\section{Slices}

This section contains a determination of
the slices of hermitian $K$-theory and Witt theory
over any field of characteristic not two. These are the first 
examples of non-orientable motivic spectra 
for which all slices are explicitly known.
Our starting point is the computation of the
slices of algebraic $K$-theory.

\subsection{Algebraic $K$-theory}
\label{section:slicesofktheory}
We recall and augment previous work on the slices of algebraic $K$-theory.
\begin{theorem}[Levine, Voevodsky]
\label{thm:slices-k}
The unit map $\One\rTo\KGL$ induces an isomorphism of zero slices
\[ 
\s_{0}(\One)\rTo \s_{0}(\KGL). 
\]
Hence there is an isomorphism $\s_q(\KGL)\iso \Sigma^{2q,q} \MZZ$ for all $q\in \ZZ$.
\end{theorem}
\begin{proof}
By \cite{Levine:slices}, 
\cite{Voevodsky:motivicss}, 
and \cite{Voevodsky:zero-slice}, 
and base change to any field as in \cite{HKPAO},
the unit $\One\rTo\KGL$ induces a map 
\[ 
\MZZ 
\iso 
\s_{0}(\One)
\rTo 
\s_{0}(\KGL) 
\iso 
\MZZ 
\]
corresponding to multiplication by an integer $i\in\MZZ_{0,0}\iso\ZZ$. 
Now $\s_{0}(\One)\rTo \s_{0}(\KGL)$ is a map of ring spectra by multiplicativity of the slice filtration \cite[Theorem 5.19]{GRSO}, \cite{Pelaez}.
It follows that $i=1$.
Moreover, 
the $q$th power $\beta^{q}\colon S^{2q,q}\rTo\KGL$ of the Bott map induces an isomorphism
\[ 
\Sigma^{2q,q} \MZZ
\iso 
\s_{q}(S^{2q,q}) 
\rTo 
\s_{q}(\KGL). 
\]
\end{proof}

\subsection{Homotopy orbit $K$-theory}
\label{section:slicesofhomotopyorbitktheory}
Let $\PP^{1}$ be pointed at $\infty$.
The general linear group scheme $\GL(2n)$ acquires an involution given by
\[ 
\begin{pmatrix} 
A & B 
\\ C & D 
\end{pmatrix} 
\rMapsto 
\begin{pmatrix} 
D^t & B^t 
\\ C^t & A^t 
\end{pmatrix}^{-1}. 
\]
Geometrically, 
this corresponds to a strictification of the pseudo-involution obtained by sending a vector bundle of rank $n$ to its dual.
These involutions induce the inverse-transpose involution on the infinite general linear group scheme $\GL$ and its classifying 
space $\mathrm{B}\GL$ \cite{Thomason}. 
Letting $C_{2}$ operate trivially on the first factor in $\ZZ\times\mathrm{B}\GL$, 
the involution coincides sectionwise with the unstable Adams operation $\Psi^{-1}$ on the motivic space representing algebraic $K$-theory.
The stable Adams operation
\[
\Psi^{-1}_{\mathrm{st}}
\colon 
\KGL
\rTo 
\KGL
\]
is determined by the structure map $K\rTo\Omega^{2,1}K=\hofib(K(-\times\PP^{1})\rTo^{\infty^{\ast}} K(-))$ obtained from multiplication by the 
class $1-[\OO(-1)]\in K_{0}(\PP^{1})$ of the hyperplane section. 
Note that 
\[ 
\Psi^{-1}(1-[\OO(-1)])
= 
1-[\OO(1)]
= 
1-(1+(1-[\OO(-1)])) 
=
[\OO(-1)]-1 
= 
-(1-[\OO(-1)]). 
\]
Thus, 
up to homotopy,
the stable operation $\Psi^{-1}_{\mathrm{st}}\colon\KGL\rTo\KGL$ is given levelwise on motivic spaces by the formula
\[
\Psi^{-1}_{\mathrm{st},n} 
= 
\begin{cases} 
 \Psi^{-1} & n \equiv 0 \bmod 2 \\
-\Psi^{-1} & n \equiv 1 \bmod 2.
\end{cases} 
\]
Since smashing with $S^{2,1}$ shifts motivic spectra by one index we get
\begin{equation}\label{eq:stable-adams}
\Sigma^{2,1}  \Psi^{-1}_{\mathrm{st}} 
= 
-\Psi^{-1}_{\mathrm{st}}.
\end{equation}
When forming homotopy fixed points and homotopy orbits of $\KGL$ we implicitly make use of a naive $C_{2}$-equivariant motivic spectrum,
i.e., 
a non-equivariant motivic spectrum with a $C_{2}$-action 
(given by $\Psi^{-1}_{\mathrm{st}}$) 
which maps by a levelwise weak equivalence to $\KGL$. 
Recall the Witt-theory spectrum $\KT$ is the homotopy colimit of the 
sequential diagram
\begin{equation}
\label{KTdef}
\KQ 
\rTo^{\eta} 
\Sigma^{-1,-1}\KQ 
\rTo^{\Sigma^{-1,-1}\eta} 
\Sigma^{-2,-2}\KQ 
\rTo^{\Sigma^{-2,-2}\eta} 
\dotsm .
\end{equation}
Note that $\KT$ is a motivic ring spectrum equipped with an evident $\KQ$-algebra structure.
Our notation follows \cite{Hornbostel} and reminds us of the fact that $\KT$ is an example of a "Tate spectrum" or more precisely "geometric fixed point spectrum" via the homotopy cofiber sequence relating it to $\KQ$ and homotopy orbit algebraic $K$-theory \cite{Kobal}.
\begin{theorem}[Kobal]
\label{theorem:kobalsequence}
There is a homotopy cofiber sequence
\[ 
\KGL_{hC_{2}}
\rTo
\KQ 
\rTo 
\KT.
\]
\end{theorem}
The involution on $\KGL$ induces an $\MZZ$-linear involution on the slices $\s_{q}(\KGL)\iso \Sigma^{2q,q}  \MZZ$. 
Suspensions of $\MZZ$ allow only two possible involutions, 
namely the identity (trivial involution) and the multiplication by $-1$ map (nontrivial involution).
This follows because $\MZZ_{0,0}^{\times}\iso\{\pm 1\}$.
\begin{proposition}
\label{prop:involution-slices}
Let $(\KGL,\Psi^{-1}_{\mathrm{st}})$ be the motivic $K$-theory spectrum with its Adams involution.
The induced involution on $\s_{q}(\KGL)\iso \Sigma^{2q,q} \MZZ$ is nontrivial if $q$ is odd and trivial if $q$ is even.
\end{proposition}
\begin{proof}
This follows from (\ref{eq:stable-adams}) and the equality $\s_{0}(\Psi^{-1})=\id_{\MZZ}$. 
For the latter, 
note that the Adams involution and the unit map of $\KGL$ yield commutative diagrams:
\begin{diagram}[width=2cm,height=0.5cm]
& & \KGL & & & \s_{0}(\KGL) \\
& \ruTo^{\iota} & & & \ruTo^{\s_{0}(\iota)} & \\
\One & & \dTo_{\Psi^{-1}_{\mathrm{st}}} & \s_{0}(\One) & & \dTo_{\s_{0}(\Psi^{-1}_{\mathrm{st}})} \\
& \rdTo_{\iota} & & & \rdTo_{\s_{0}(\iota)} & \\
& & \KGL & & & \s_{0}(\KGL)
\end{diagram}
By reference to Theorem~\ref{thm:slices-k} it follows that $\s_{0}(\Psi^{-1}_{\mathrm{st}})=\id_{\MZZ}$.
\end{proof}
 
Next we consider the composite of the hyperbolic and forgetful maps 
\begin{equation}\label{eq:fh}
f\circ h 
= 
\Psi^{1}_{\mathrm{st}}+\Psi^{-1}_{\mathrm{st}}
\colon\KGL\rTo\KGL.
\end{equation}
\begin{proposition}
\label{prop:hyper-forget-slices}
The endomorphism $\s_{q}(fh)$ of $\s_{q}(\KGL)\iso \Sigma^{2q,q}  \MZZ$ 
is multiplication by $2$ if $q$ is even and the trivial map if $q$ is odd.
\end{proposition}
\begin{proof}
Apply Proposition \ref{prop:involution-slices}, 
the additivity of $\s_{q}$ and $\s_{q}(\Psi^{1}_{\mathrm{st}})=\s_{q}(\id_{\KGL})=\id_{\s_{q}(\KGL)}$.
\end{proof}

\begin{lemma}
\label{lemma:sliceshomotopycolimits}
The slice functor $\s_{q}$ commutes with homotopy colimits for all $q\in\ZZ$.
\end{lemma}
\begin{proof}
By a general result for Quillen adjunctions between stable pointed model categories shown by Spitzweck \cite[Lemma 4.4]{Spitzweck:slices},
the conclusion follows because $\s_{q}$ commutes with sums.
\end{proof}

The idea is now to combine Theorem~\ref{thm:slices-k}, 
Proposition~\ref{prop:involution-slices} and Lemma \ref{lemma:sliceshomotopycolimits} to identify the slices of homotopy orbit $K$-theory.
To start with, 
we compute the homotopy orbit spectra of $\MZZ$ for the trivial
involution and the unique nontrivial involution
over a base scheme essentially smooth over a  
field. 
These computations are parallel to the corresponding
computations for the topological integral Eilenberg-MacLane
spectrum $H\ZZ$, as the proofs suggest. 
In the case where $F$ is a subfield of the complex
numbers, complex realization -- which is compatible with
homotopy colimits, and sends $\MZZ$ to $H\ZZ$ by \cite[Lemma 5.6]{Levine:comparison}
 -- maps
the motivic computation to the topological one.
\begin{lemma}
\label{lem:mz-triv-inv}
With the trivial involution on motivic cohomology there is an isomorphism
\[ 
(\MZZ,\id)_{hC_{2}}  
\iso 
\MZZ\vee\bigvee_{i=0}^{\infty} \Sigma^{2i+1,0}  \MZZ/2. 
\] 
\end{lemma}
\begin{proof}
If $\E$ is equipped with an involution then $\E_{hC_{2}}\iso\bigl(\E\smash (EC_{2})_{+}\bigr)_{C_{2}}$, 
where $EC_{2}$ is a contractible simplicial set with a free $C_{2}$-action considered as a constant motivic space.
In the case of $(\MZZ,\id)$, 
\[ 
(\MZZ,\id)_{hC_{2}} 
\iso 
\bigl(\MZZ\smash (EC_{2})_{+}\bigr)_{C_{2}} 
\iso 
\MZZ\smash (EC_{2})_{+}/C_{2}
\iso  
\MZZ\smash\mathbb{RP}^{\infty}_{+}.
\]
Since $S^{n,0}\rTo \mathbb{RP}^{n}\rTo\mathbb{RP}^{n+1}$ is a homotopy cofiber sequence, 
so is 
\begin{equation}
\label{RPncofibration} 
\Sigma^{n,0} \MZZ
\rTo 
\MZZ\smash \mathbb{RP}^{n} 
\rTo 
\MZZ\smash \mathbb{RP}^{n+1}.
\end{equation}
Clearly $\MZZ\smash\mathbb{RP}^{2}\cong \Sigma^{1,0} \MZZ/2$.
Proceeding by induction on (\ref{RPncofibration}) using Lemma \ref{lemma:MZtoMZ/2weightzero},
we obtain
\[
\MZZ\smash \mathbb{RP}^{n}
\cong
\begin{cases} 
\bigvee_{i=0}^{n-2} \Sigma^{2i+1,0}  \MZZ/2 & n \equiv 0 \bmod 2 \\
\Sigma^{n,0}  \MZZ\vee \bigvee_{i=0}^{n-3} \Sigma^{2i+1,0}  \MZZ/2 & n \equiv 1 \bmod 2.
\end{cases} 
\]
\end{proof}

\begin{lemma}
\label{lem:mz-nontriv-inv}
With the nontrivial involution $\sigma$ on motivic cohomology there is an isomorphism
\[ 
(\MZZ,\sigma)_{hC_{2}} 
\iso 
\bigvee_{i=0}^{\infty} \Sigma^{2i,0}  \MZZ/2. 
\]
\end{lemma}
\begin{proof}
The nontrivial involution $\sigma$ on $\MZZ$ is determined levelwise by the nontrivial involution on 
\[
\MZZ_{n}  
= 
K(\ZZ,2n,n) 
= 
\ZZ^{\tr}(S^{2n,n}).
\]
Here, 
$\ZZ^\tr$ sends a motivic space $\mathcal{X}$ to the motivic space with transfers freely generated by $\mathcal{X}$,
cf.~\cite[Example 3.4]{DRO:motivic}.
On the level of motivic spaces the nontrivial involution is induced by a degree $-1$ pointed map of the simplicial circle $S^{1,0}$. 
Since $\ZZ^{\tr}$ commutes with homotopy colimits \cite[\S 2]{RO:mz} we are reduced to identify
\[
\hocolim_{C_{2}} S^{2n,n}.
\]
When $n=0$ the homotopy colimit is contractible.
When $n>0$ there are isomorphisms
\[
\hocolim_{C_{2}} S^{2n,n}
\iso 
\hocolim_{C_{2}} \Sigma^{2n-1,n}  S^{1,0}
\iso 
\Sigma^{2n-1,n}  (\hocolim_{C_{2}} S^{1,0})
\iso 
\Sigma^{2n-1,n} \mathbb{RP}^{\infty}. 
\]
By passing to the sphere spectrum the above yields an isomorphism
\[ 
\hocolim_{C_{2}}\One 
\iso 
\Sigma^{-1,0} \mathbb{RP}^{\infty}.
\]
Finally, 
applying transfers and arguing as in the proof of Lemma \ref{lem:mz-triv-inv}, 
we deduce the isomorphism
\[ 
\hocolim_{hC_{2}}(\MZZ,\sigma) 
\cong 
\bigvee_{i=0}^{\infty} \Sigma^{2i,0}  \MZZ/2. 
\]
\end{proof}

\begin{theorem}
\label{thm:slices-homotopy-orbit-K}
Suppose $F$ is a field equipped 
with the trivial involution and $\Char(F)\neq 2$. 
The slices of homotopy orbit $K$-theory are given by
\[ 
\s_{q}(\KGL_{hC_{2}}) 
\iso 
\begin{cases} 
\Sigma^{2q,q}  \MZZ \vee \bigvee_{i=\frac{q}{2}}^{\infty} \Sigma^{q+2i+1,q}  \MZZ/2 & q\equiv 0\bmod 2 \\
\bigvee_{i=\frac{q-1}{2}}^{\infty} \Sigma^{q+2i+1,q}  \MZZ/2 & q\equiv 1\bmod 2.
\end{cases}
\]
\end{theorem}
\begin{proof}
This follows from Theorem~\ref{thm:slices-k}, Proposition \ref{prop:involution-slices} and Lemmas \ref{lemma:sliceshomotopycolimits}, \ref{lem:mz-triv-inv}, \ref{lem:mz-nontriv-inv}.
\end{proof}

\subsection{Hermitian $K$-theory}
\label{section:slic-herm-algebr}
Throughout this section $F$ is a field of $\Char(F)\neq 2$.
\begin{corollary}
\label{cor:split-zeroslice-ko}
There is a splitting of $\MZZ$-modules
\[ 
\s_{0}(\KQ)
\iso 
\s_{0}(\KGL)\vee\mu
\]
which identifies $\s_0(f)$ with the projection map onto $\s_0(\KGL)$.
\end{corollary}
\begin{proof}
Lemma~\ref{lem:c-ring-map} shows the unit of $\KQ$ and the forgetful map to $\KGL$ furnish a factorization 
\[
\One\rTo^{\iota}\KQ\rTo^{f} \KGL
\]
of the unit of $\KGL$. 
By Theorem~\ref{thm:slices-k} the composite
\[ 
\s_{0}(\One)
\rTo^{\s_{0}(\iota)} 
\s_{0}(\KQ)
\rTo^{\s_{0}(f)} 
\s_{0}(\KGL)
\]
is an isomorphism.
The desired splitting follows since retracts are direct summands in $\SHH$.
\end{proof}
\begin{proposition}
\label{prop:mult-2} 
The composite $\KQ\rTo^{f}\KGL\rTo^{h}\KQ$ induces the multiplication by $2$ map on $\s_{q}(\KQ)$ for all $q\in\ZZ$.
\end{proposition}
\begin{proof}
The unit of $\MZZ$ induces an isomorphism on zero slices by Theorem \ref{thm:slices-k}.
Since $\MZZ\in\SHH^\mathrm{eff}$ is an effective motivic spectrum, 
the counit $\f_{0}(\MZZ)\rTo\MZZ$ is an isomorphism. 
Thus there is a canonical isomorphism of ring spectra
\[
\MZZ
\iso 
\f_{0}(\MZZ)
\rTo 
\s_{0}(\MZZ)
\iso 
\MZZ.
\]
It follows that smashing with $\MZZ$,
i.e., 
passing to motives \cite{RO:cr}, \cite{RO:mz}, 
induces a canonical ring map 
\[ 
[\One,\One]
\rTo^{}_{} 
{[\s_{0}(\One),\s_{0}(\One)]}
\iso [\MZZ,\MZZ].
\]
Lemma \ref{lemma:multiplicationbyfh} shows the composite $hf$ coincides with $1-\epsilon$.
Its image in motives is multiplication by $2$ because the twist isomorphism of $\ZZ(1)\tensor \ZZ(1)$ is the identity map 
\cite[Corollary 2.1.5]{Voevodsky:dm}.
\end{proof}
  
\begin{corollary}
In the splitting $\s_{0}(\KQ)\iso \s_{0}(\KGL)\vee\mu$ the homotopy groups $\pi_{s,t}\mu$ are modules over $\pi_{0,0}\MZZ/2\iso \FF_{2}$.
\end{corollary}
\begin{proof}
This follows from Corollary \ref{cor:split-zeroslice-ko} and Proposition \ref{prop:mult-2}.
\end{proof}

The slice functors are triangulated. 
Thus the homotopy cofiber sequence in Theorem \ref{thm:cofiber-seq} induces homotopy cofiber sequences of slices. 

\begin{lemma}
\label{lem:s1-ko}
There is a distinguished triangle of $\MZZ$-modules
\[ 
\Sigma^{1,1}  \MZZ 
\rTo^{\Sigma^{1,1}  (2,0)} 
\Sigma^{1,1}  (\MZZ\vee \mu) 
\rTo \s_{1}(\KQ)
\rTo 
\Sigma^{2,1}  \MZZ, 
\]
where $2\in\ZZ\iso\MZZ_{0,0}$ and $0\in\pi_{0,0}\mu$.
\end{lemma}
\begin{proof}
The distinguished triangle follows from Theorems \ref{thm:cofiber-seq}, 
\ref{thm:slices-k},
Lemma \ref{lemma:shift-slices} and Corollary \ref{cor:split-zeroslice-ko}.
Since induced maps between slices are module maps over $\s_{0}(\One)\iso \MZZ$,
we have
\[ 
\Hom_{\MZZ}(\Sigma^{1,1}  \MZZ,\Sigma^{1,1}  (\MZZ\vee \mu)) 
\iso
\Hom_{\SHH}(\One,\MZZ\vee \mu) 
\iso 
\MZZ_{0,0}\directsum\pi_{0,0}\mu. 
\]
This shows $\MZZ\rTo \MZZ\vee \mu$ is of the form $(a,\alpha)\in\ZZ\directsum \pi_{0,0}\mu$. 
By the proof of Corollary~\ref{cor:split-zeroslice-ko},
$\s_{0}(f)$ is the projection map onto the direct summand $\s_{0}(\KGL)$. 
The connecting map in the distinguished triangle identifies with the $(1,1)$-shift of the hyperbolic map $h$. 
It follows that $\s_{-1}(h)=0$. 
By Proposition~\ref{prop:mult-2},
$\s_{0}(hf)$ is multiplication by $2$. 
Thus $a=2$ and $\pi_{0,0}\s_{0}(\KGL)\rTo \pi_{0,0}\s_{0}(\KQ)$ is injective. 
Consider the short exact sequence
\[ 
0 
\rTo 
\pi_{0,0}\s_{0}(\KGL) 
\rTo 
\pi_{0,0}\s_{0}(\KQ) 
\rTo 
\pi_{0,0}\cone({2,\alpha})
\rTo 
\pi_{-1,0}\s_{0}\KGL=0. 
\]
Proposition~\ref{prop:mult-2} shows that $\s_{1}(\KQ)\rTo \s_{1}(\KGL) \rTo \s_{1}(\KQ)$ is multiplication by $2$.
Hence the same holds for the induced map $\pi_{0,0}\cone({2,\alpha})\rTo \pi_{0,0} \Sigma^{1,0}  \MZZ = 0\rTo\pi_{0,0}\cone({2,\alpha})$.
It follows that $\pi_{0,0}\cone({2,\alpha})$ is an $\FF_{2}$-module. 
We note that the image of $(1,\alpha)$ has order $4$ unless $\alpha=0$:
Write $\pi_{0,0}\cone({2,\alpha})$ as the cokernel of $(2,\alpha)\colon\ZZ\rTo\ZZ\directsum A$, 
$A$ an $\FF_{2}$-module. 
Its subgroup generated by $\overline{(1,\alpha)}$ is
$\{\overline{(1,\alpha)},\overline{(2,0)},\overline{(3,\alpha)},\overline{(4,0)}=2\overline{(2,\alpha)} = 0\}$, 
and if $\alpha\neq 0$, 
then $\overline{(2,0)}\neq 0$.
\end{proof}

\begin{corollary}
\label{cor:s1-ko}
There is an isomorphism $\s_{1}(\KQ) \iso \Sigma^{1,1} (\MZZ/2\vee \mu)$ and $\s_1(f)$ coincides with the composite map
\[ 
\s_{1}(\KQ)\iso \Sigma^{1,1}  (\MZZ/2 \vee \mu) 
\rTo^{\pr} 
\Sigma^{1,1} \MZZ/2 \rTo^{\delta} \Sigma^{2,1}  \MZZ 
\iso 
\s_{1}(\KGL).\]
\end{corollary}
\begin{proof}
This follows from Lemma~\ref{lem:s1-ko}.
\end{proof}

\begin{lemma}\label{lem:s1h}
The map $\s_{1}(h)\colon \s_1(\KGL)\rTo \s_{1}\KQ$ is trivial.
\end{lemma}
\begin{proof}
The first component of
\[ 
\s_{1}(h)\colon 
\s_1(\KGL)
\rTo 
\s_{1}(\KQ)\iso \Sigma^{1,1}  (\MZZ/2 \vee \mu)
\]
is trivial for degree reasons by Lemma \ref{lemma:MZtoMZ/2weightzero}.
Thus $\s_{1}(h)$ corresponds to the transpose of the matrix $\begin{pmatrix} 0 & \beta \end{pmatrix}$.
By Corollary~\ref{cor:s1-ko}, 
$\s_{1}(f)$ corresponds to the matrix $\begin{pmatrix} \delta & 0\end{pmatrix}$.
Proposition~\ref{prop:mult-2} implies that
\[ 
\s_{1}(hf) 
= 
\begin{pmatrix} 
0 & 0 \\ 0 & \beta\delta 
\end{pmatrix}
\]
is the multiplication by $2$ map,
hence trivial on $\Sigma^{1,1}  (\MZZ/2\vee \mu)$. 
We conclude that $\beta\delta =0$.
Thus $\beta$ factors as
\[ 
\Sigma^{2,1}  \MZZ 
\rTo^{2} \Sigma^{2,1}  \MZZ 
\rTo^{\beta^\prime} 
\Sigma^{1,1}  \mu, 
\]
which coincides with
\[ 
\Sigma^{2,1}  \MZZ 
\rTo^{\beta^\prime} 
\Sigma^{1,1}  \mu 
\rTo^2 \Sigma^{1,1}  \mu.
\]
Since the multiplication by $2$ map on $\mu$ is trivial, 
the result follows.
\end{proof}

\begin{corollary}
\label{cor:s2-ko}
There is an isomorphism
\[  
\s_{2}(\KQ)
\rTo^{\iso} 
\Sigma^{4,2}  \MZZ\vee \Sigma^{2,2}  (\MZZ/2\vee \mu) 
\]
which identifies $\s_2(f)$ with the projection map onto $\Sigma^{4,2}  \MZZ$.
\end{corollary}
\begin{proof}
Theorem \ref{thm:slices-k} and Corollary \ref{cor:s1-ko} show that, 
up to isomorphism,
Theorem \ref{thm:cofiber-seq} gives rise to the distinguished triangle
\[ 
\Sigma^{3,2}  \MZZ 
\rTo^{\Sigma^{1,1}  \s_1(h)} 
\Sigma^{2,2}  (\MZZ/2\vee \mu) 
\rTo 
\s_{2}(\KQ) 
\rTo^{\s_2(f)} 
\Sigma^{4,2}  \MZZ. \]
Lemma \ref{lem:s1h} implies that  $\s_1(h)$ is the trivial map.
\end{proof}

\begin{lemma}
\label{lem:s2h}
The map $\s_{2}(h)\colon \s_{2}(\KGL)\rTo \s_{2}(\KQ)$ coincides with the composite 
\[ 
\Sigma^{4,2}  \MZZ
\rTo^{(2,0)}
\Sigma^{4,2}  \MZZ\vee \Sigma^{2,2}  (\MZZ/2\vee \mu). 
\]
\end{lemma}
\begin{proof}
Corollary \ref{cor:s2-ko} identifies $\s_{2}(f)$ with the projection map onto $\Sigma^{4,2}  \MZZ$. 
Since $\s_2(fh)=2$ by Proposition~\ref{prop:hyper-forget-slices}, 
$\s_{2}(h)=(2,\gamma_1,\gamma_2)$ where 
\[ 
(\gamma_1,\gamma_2)
\colon 
\Sigma^{4,2} \MZZ
\rTo 
\Sigma^{2,2} (\MZZ/2\vee \mu).
\]
Note that $\gamma_1=0$ by Lemma~\ref{lemma:MZtoMZ/2weightzero}. 
Similarly $\s_{2}(hf)=2$ by Proposition~\ref{prop:mult-2}. 
Using the matrix 
\[ 
\s_{2}(hf) 
= 
\begin{pmatrix} 
2 & 0 & 0 \\ 0 & 0 & 0 \\ \gamma_2 & 0 & 0 
\end{pmatrix}\]
we conclude that $\gamma_2=0$.
\end{proof}

\begin{corollary}\label{cor:s3-ko}
There is an isomorphism
\[ 
\s_{3}(\KQ)
\iso  
\Sigma^{5,3}  \MZZ/2 \vee \Sigma^{3,3}  (\MZZ/2\vee \mu). 
\]
Moreover,
$\s_{3}(f)$ coincides with the composite map
\[ 
\s_{3}(\KQ)
\iso 
\Sigma^{5,3}  \MZZ/2 \vee \Sigma^{3,3}  (\MZZ/2 \vee \mu) 
\rTo^{\pr} 
\Sigma^{5,3}  \MZZ/2 
\rTo^{\delta} \Sigma^{6,3}  \MZZ 
\iso 
\s_{3}(\KGL).\]
\end{corollary}
\begin{proof}
Apply $\s_{3}$ to the Hopf cofiber sequence in Theorem \ref{thm:cofiber-seq} and use Lemma~\ref{lem:s2h}.
\end{proof}

\begin{theorem}
\label{thm:slices-ko}
The slices of the hermitian $K$-theory spectrum $\KQ$ are given by
\[  
\s_{q}(\KQ)
=
\begin{cases}
\bigl(\Sigma^{2q,q}  \MZZ\bigr) 
\vee 
\bigl(\Sigma^{q,q}  \Mmu\bigr) 
\vee
\bigvee_{i<\frac{q}{2}} \Sigma^{2i+q,q}  \MZZ/2 & q\equiv 0 \bmod 2 \\
\bigl(\Sigma^{q,q}  \Mmu\bigr) 
\vee 
\bigvee_{i<\frac{q+1}{2}} \Sigma^{2i+q,q}  \MZZ/2 & q\equiv 1 \bmod 2.
\end{cases}
\]
Here $\Mmu\cong \Sigma^{4,0} \Mmu$ and $\Mmu_{s,t}$ is an $\FF_{2}$-module for all integers $s$ and $t$.
\end{theorem}
\begin{proof}
Corollary \ref{cor:s3-ko} identifies the third slice
\[ 
\s_{3}(\KQ)
\iso  
\Sigma^{5,3}  \MZZ/2 \vee \Sigma^{3,3}  (\MZZ/2\vee \mu). 
\]
On the other hand, 
Karoubi periodicity and Corollary~\ref{cor:split-zeroslice-ko} imply the isomorphism
\[ 
\s_{3}(\KQ)
\iso 
\s_{3}(\Sigma^{8,4}  \KQ) 
\iso 
\Sigma^{8,4}  s_{-1}(\KQ)
\iso 
\Sigma^{7,3}  \mu.
\]
It follows that $\Sigma^{-2,0}  \MZZ \vee \Sigma^{-4,0}  \MZZ$ is a direct summand of $\mu$. 
Iterating the procedure furnishes an isomorphism
\[ 
\mu 
\iso 
\Mmu \vee \bigvee_{i<0} \Sigma^{-2i,0}  \MZZ/2, 
\]
where $\Mmu$ is simply the complementary summand. The result follows.
\end{proof}

In the next section we show the mysterious summand $\Mmu$ of $\s_{0}(\KQ)$ is trivial.
The following summarizes some of the main observations in this section.
\begin{proposition}
\label{prop:slices-forget-hyper}
The map $\s_{2q}(f)\colon \s_{2q}(\KQ) \rTo \s_{2q}(\KGL)$ is the projection onto $\Sigma^{4q,2q}  \MZZ$,
$\s_{2q+1}(f)\colon \s_{2q+1}(\KQ) \rTo \s_{2q+1}(\KGL)$ is the projection map onto $\Sigma^{4q+1,2q+1}  \MZZ/2$ composed with 
$\delta\colon \Sigma^{4q+1,2q+1}  \MZZ/2\rTo \Sigma^{4q+2,2q+1}  \MZZ$.
The map $\s_{i}(h)\colon \s_{i}(\KGL) \rTo \s_{i}(\KQ)$ is multiplication by $2$ composed with the inclusion of $\Sigma^{4q,2q} \MZZ$ if $i=2q$, 
and trivial if $i$ is odd.
\end{proposition}

\subsection{The mysterious summand is trivial}
\label{section:mysterious-summand}
To prove $\Mmu$ is trivial we use the 
solution of the homotopy limit problem for hermitian $K$-theory of prime fields due to 
Berrick-Karoubi \cite{BK} and Friedlander \cite{Friedlander}.
Their results have been generalized by Hu-Kriz-Ormsby \cite{Hu-Kriz-Ormsby} and 
Berrick-Karoubi-Schlichting-{\O}stv{\ae}r \cite{BKSO},
where the following is shown.
Here $\vcd_{2}(F)=\cd_{2}(F(\sqrt{-1}))$ denotes the virtual mod-$2$ cohomological dimension of $F$.
\begin{theorem}
\label{thm:homotopy-limit-prime-fields}
If $\vcd_{2}(F)<\infty$ there is a canonical weak equivalence $\KQ/2\rTo\KGL^{hC_{2}}/2$.
\end{theorem}

To make use of Theorem \ref{thm:homotopy-limit-prime-fields} we need to understand how slices compare 
with mod-$2$ reductions of motivic spectra and formation of homotopy fixed points.
The former is straightforward because the slice functors are triangulated.
\begin{lemma}
\label{lemma:slices-mod2}
Let $\E$ be a motivic spectrum. 
There is a canonical isomorphism $\s_{q}(\E)/2\iso \s_{q}(\E/2)$.
\end{lemma}

\begin{lemma}\label{lem:hofix-mz-trivial}
There is a naturally induced commutative diagram:
\begin{diagram}
(\MZZ,\id)^{hC_{2}} & &\rTo^\iso &&  
\MZZ\vee\bigvee_{i<0} \Sigma^{2i,0}  \MZZ/2 \\
& \rdTo & & \ldTo_{\pr}& \\
&&\MZZ&&
\end{diagram}
\end{lemma}
\begin{proof}
With the trivial $C_{2}$-action on $\MZZ$ the homotopy fixed points spectrum is given by
\[ 
(\MZZ,\id)^{hC_{2}} 
= 
\sSet_\ast(\RR\PP^{\infty}_{+},\MZZ) 
\cong
\MZZ\vee \sSet_\ast(\RR\PP^{\infty},\MZZ). 
\]
The canonical map $(\MZZ,\id)^{hC_{2}} \rTo \MZZ$ corresponds to the map induced by $\Spec(F)_{+}\rInto\RR\PP^{\infty}_{+}$.
Induction on $n$ shows there are isomorphisms 
\[ 
\sSet_\ast(\RR\PP^n,\MZZ)
\cong
\begin{cases} 
\Sigma^{-n,0}  \MZZ \times \prod_{i=1}^{\frac{n-1}{2}} \Sigma^{-2i,0}  \MZZ/2 & n\equiv 1\bmod 2 \\
\prod_{i=1}^{\frac{n}{2}} \Sigma^{-2i,0}  \MZZ/2 & n\equiv 0\bmod 2.
\end{cases} 
\]
Writing $\sSet_\ast(\RR\PP^{\infty},\MZZ)=\colim_{n}\sSet_\ast(\RR\PP^n,\MZZ)$ we deduce
\[ 
\sSet_\ast(\RR\PP^{\infty},\MZZ) 
\cong
\prod_{i<0} \Sigma^{2i,0}  \MZZ/2. 
\]
Proposition \ref{prop:sum-product} identifies the latter with $\bigvee_{i<0} \Sigma^{2i,0}  \MZZ/2$. 
\end{proof}

\begin{lemma}\label{lem:hofix-mz-nontrivial}
The nontrivial involution on $\MZZ$ and the connecting map $\delta\colon \Sigma^{-1,0}  \MZZ/2\rTo \MZZ$ (on the zeroth summand below) give rise to the commutative diagram:
\begin{diagram}
(\MZZ,\sigma)^{hC_{2}} & & 
\rTo^\iso  & & 
\bigvee_{i\leq 0} \Sigma^{2i-1,0}  \MZZ/2\\
& \rdTo & & \ldTo_{\delta} & \\
& & \MZZ & & 
\end{diagram}
\end{lemma}
\begin{proof}
Let $(G,\sigma)$ be an involutive simplicial abelian group with $\sigma(g)=-g$ and $EC_{2}$ a contractible simplicial set 
with a free $C_{2}$-action.
Then $C_{2}$ acts on the space of fixed points $\sSet(EC_{2},G)^{C_{2}}$ by conjugation,
i.e., 
on the simplicial set of maps from $EC_{2}$ to $(G,\sigma)$. 
Now choose $EC_{2}$ such that its skeletal filtration
$
(S^0,g) 
\subseteq 
(S^1,g) 
\subseteq 
\dotsm 
\subseteq 
(S^n,g) 
\subseteq 
\dotsm 
$
comprises spheres equipped with the antipodal  $C_{2}$-action. 
Here $(S^{n+1},g)$ arises from $(S^n,g)$ by attaching a free $C_{2}$-cell of dimension $n+1$ along the $C_{2}$-map
$C_{2}\times S^n \rTo (S^n,g)$ adjoint to the identity on $S^n$. 
Thus the inclusion $(S^n,g) \subseteq (S^{n+1},g)$ yields a pullback square of simplicial sets:
\begin{equation}
\label{eq:3}
\begin{diagram}
\sSet\bigl((S^{n+1},g),(G,\sigma)\bigr)^{C_{2}} & \rTo & \sSet\bigl((S^{n},g),(G,\sigma)\bigr)^{C_{2}} \\
\dTo & & \dTo \\
\sSet(D^{n+1},G) & \rTo & \sSet(S^{n},G) 
\end{diagram}
\end{equation}
By specializing to $n=0$, 
(\ref{eq:3}) extends to the commutative diagram:
\begin{diagram}
\sSet\bigl((S^{1},g),(G,\sigma)\bigr)^{C_{2}} & \rTo & 
\sSet\bigl((S^{0},g),(G,\sigma)\bigr)^{C_{2}} \iso G & \rTo & 0 \\
\dTo & & \dTo^{x\mapsto (x,-x)} & & \dTo \\
\sSet(D^{1},G) & \rTo^{\phi} & \sSet(S^{0},G)\iso G\times G & \rTo^{+} & G 
\end{diagram}
$\phi$ is a fibrant replacement of the diagonal $G\rTo G\times G$. 
The right hand square is a pullback, 
and addition is a Kan fibration.
Hence $\sSet\bigl((S^{1},g),(G,\sigma)\bigr)^{C_{2}}$ is the homotopy fiber of multiplication by 2 on $G$. 
On the other hand, 
there is a homotopy cofiber sequence $S^1 \rTo^2 S^1\iso \RR\PP^1\rTo \RR\PP^2$ and $\sSet\bigl((S^{1},g),(G,\sigma)\bigr)^{C_{2}}$ 
has a canonical basepoint.
Thus there is a homotopy equivalence
\[ 
\Omega\sSet\bigl((S^{1},g),(G,\sigma)\bigr)^{C_{2}} 
\simeq 
\sSet_\ast(\RR\PP^2,G).
\]
By induction on (\ref{eq:3}), 
we find for every $n\geq 0$ a homotopy equivalence 
\[ 
\Omega \sSet\bigl((S^{n},g),(G,\sigma)\bigr)^{C_{2}} 
\simeq 
\sSet_\ast(\RR\PP^{n+1},G), 
\]
which implies
\begin{equation}
\label{eq:4}
\Omega(G,\sigma)^{hC_{2}}  
\simeq 
\sSet_\ast(\RR\PP^\infty,G).
\end{equation}
A levelwise and sectionwise application of~(\ref{eq:4}) yields the weak equivalences
\begin{equation*}
\Omega^{1,0}(\MZZ,\sigma)^{hC_{2}}  
\cong 
\sSet_\ast(\RR\PP^\infty,\MZZ)
\cong
\prod_{i<0} S^{2i,0}  \MZZ/2. 
\end{equation*}
(For the second weak equivalence see the proof of Lemma~\ref{lem:mz-triv-inv}.) 
The canonical map from $(\MZZ,\sigma)^{hC_{2}}$ to $\MZZ$ corresponds to the map induced by $\RR\PP^{1}\rInto \RR\PP^{\infty}$.
Proposition~\ref{prop:sum-product} concludes the proof.
\end{proof}

\begin{proposition}
\label{proposition:slices-hofix}
Suppose that $\vcd_{2}(F)<\infty$.
The canonical map\[ 
\s_{q}(\KGL^{hC_{2}}) 
\rTo
\s_{q}(\KGL)^{hC_{2}} 
\]
is an isomorphism for every integer $q$.
\end{proposition}
\begin{proof}
Let $F$ be an arbitrary field of characteristic not 2.
Lemma~\ref{lem:map-myst} shows that $\KGL^{hC_2}$ satisfies the properties which were used to determine the slices of $\KQ$. 
As in the proof of Theorem \ref{thm:slices-ko},
this results in a splitting
\[  
\s_{q}(\KGL^{hC_2})
=
\begin{cases}
\bigl(\Sigma^{2q,q}  \MZZ\bigr) 
\vee 
\bigl(\Sigma^{q,q}  \Mnu\bigr) 
\vee
\bigvee_{i<\frac{q}{2}} \Sigma^{2i+q,q}  \MZZ/2 & q\equiv 0 \bmod 2, \\
\bigl(\Sigma^{q,q}  \Mnu\bigr) 
\vee 
\bigvee_{i<\frac{q+1}{2}} \Sigma^{2i+q,q}  \MZZ/2 & q\equiv 1 \bmod 2.
\end{cases}
\]
Here $\Mnu\cong \Sigma^{4,0} \Mnu$ and $\Mnu_{s,t}$ is an $\FF_{2}$-module for all integers $s$ and $t$. 
Hence the vanishing of $\Mnu$ can be deduced from a computation of the zero slice of $\KGL^{hC_{2}}/2=\bigl(\KGL/2\bigr)^{hC_{2}}$.
Moreover, 
the homotopy norm cofiber sequence
\[ \KGL_{hC_{2}}\rTo \KGL^{hC_{2}} \rTo \widehat{\KGL}^{C_{2}}
\rTo \Sigma^{1,0}\KGL_{hC_{2}} \]
from \cite[Diagram (20)]{Hu-Kriz-Ormsby} and Theorem~\ref{thm:slices-homotopy-orbit-K} imply that $\Mnu$ is also a direct summand 
of $\s_0\bigl(\widehat{\KGL}^{C_{2}}\bigr)$.

Suppose that $\vcd_{2}(F)<\infty$. 
Theorem~\ref{thm:homotopy-limit-prime-fields} shows the canonical maps $\KQ/2\rTo \KGL^{hC_{2}}/2$ and $\KT/2\rTo \widehat{\KGL/2}^{C_{2}}$ are equivalences.
Thus the Tate motivic spectrum $\widehat{\KGL/2}^{C_{2}}$ is cellular in the sense of \cite{dugger-isaksen.cell}, since $\KQ$ is cellular \cite{rso.kqcell}.
(The latter is deduced using the model for $\KQ$ from \cite{paninwalterBO}.) 
Hence the zero slice of $\widehat{\KGL/2}^{C_{2}}$ is the motivic Eilenberg-MacLane spectrum associated to a chain complex of abelian groups. This chain complex
does not depend on the base field.
Thus it suffices to compute the zero slice of $\widehat{\KGL/2}^{C_{2}}$ (or, equivalently, of $\KT/2$) over an algebraically closed field of characteristic not two. 
For this one can employ the specific cell presentation given in \cite[Theorem 3.2]{hornbostel.nil} (see also \cite[Section 4]{roendigs.eta}),
which extends to any algebraically closed field of characteristic
not two via the base change techniques used in the proof of
Lemma~\ref{lemma:KGL2differential}. It shows that $\KT$ may be obtained
in two steps from the $\eta$-inverted sphere $\One[\tfrac{1}{\eta}]$.
The first step produces a canonical map $\D_\infty \rTo \KT$
from the colimit $\D_\infty$ of the sequential diagram
\[ \One[\tfrac{1}{\eta}] =\D_1 \rTo \D_2  \rTo \dotsm \rTo \D_n \rTo \dotsm \]
where $\D_{n}\rTo \D_{n+1}$ is the canonical map to the
homotopy cofiber of the unique nontrivial
map $\Sigma^{4n-1,0}\One[\tfrac{1}{\eta}] \rTo \D_n$. 
The second step provides that $\D_\infty \rTo \KT$
is obtained by inverting the unique nontrivial element
$\kappa\in\pi_{4,0}\D_\infty$. 
The computation of the slices of the sphere spectrum \cite[\S8]{Levine:comparison}  
leads to the computation
\[ 
\s_{0}(\One[\tfrac{1}{\eta}])
\cong
\MZZ/2
\vee
\bigvee_{n\geq 2} \Sigma^{n,0}\MZZ/2
\]
stated as \cite[Theorem 4.12]{roendigs.eta}.
The sequential diagram above implies that
\[ 
\s_{0}(\D_\infty)
\cong
\bigvee_{n\geq 0} \Sigma^{2n,0}\MZZ/2
\]
whence, using $\KT\cong \D_\infty[\tfrac{1}{\kappa}]$, the zero slice
of $\KT$ is
\[ 
\s_{0}(\KT) \cong \s_{0}\bigl(\D_\infty[\tfrac{1}{\kappa}]\bigr)
\cong \s_{0}(\D_\infty)[\tfrac{1}{\s_0\kappa}]
\cong
\bigvee_{n\in \ZZ} \Sigma^{2n,0}\MZZ/2.
\]
It follows that $\Mnu$ is contractible, which proves the statement.
\end{proof}

\begin{theorem}
\label{thm:slices-kgl-hofix}
Suppose that $\vcd_{2}(F)<\infty$.
The slices of homotopy fixed point algebraic $K$-theory $\KGL^{hC_{2}}$ are given by
\[  
\s_{q}(\KGL^{hC_{2}})
=
\begin{cases}
\Sigma^{q,q}  \Bigl(\bigl(\Sigma^{q,0}  \MZZ\bigr) \vee \bigvee_{i<\frac{q}{2}} \Sigma^{2i,0}  \MZZ/2\Bigr) & q\equiv 0 \bmod 2 \\
\Sigma^{q,q}  \bigvee_{i<\frac{q+1}{2}} \Sigma^{2i,0}  \MZZ/2 & q\equiv 1 \bmod 2.
\end{cases}
\]
\end{theorem}
\begin{proof}
This follows from Theorem~\ref{thm:slices-k},
Lemmas \ref{lem:hofix-mz-trivial}, \ref{lem:hofix-mz-nontrivial},
and Proposition \ref{proposition:slices-hofix}.
\end{proof}

\begin{lemma}
\label{lem:map-myst}
The map $\KQ \rTo \KGL^{hC_{2}}$ induces the projection map on all slices.
\end{lemma}
\begin{proof}
This is clear from the following claims.
\begin{enumerate}
\item\label{item:1} 
The map $\theta\colon \KQ \rTo \KGL^{hC_{2}}$ fits into a commutative diagram of homotopy cofiber sequences:
\begin{diagram}
\Sigma^{1,1}  \KQ & \rTo^\eta & \KQ & \rTo^f & \KGL & \rTo^{\Sigma^{2,1}  h\circ \beta} & \Sigma^{2,1}  \KQ \\
\dTo^{\Sigma^{2,1}  \theta} & & \dTo^\theta & & \dTo_\id & & \dTo_{\Sigma^{2,1}  \theta} \\
\Sigma^{1,1}  \KGL^{hC_{2}} & \rTo^\eta & \KGL^{hC_{2}} & \rTo^{f^\prime} & \KGL & \rTo^{\Sigma^{2,1}  h^\prime\circ \beta} & \Sigma^{2,1}  \KGL^{hC_{2}} 
\end{diagram}
\item\label{item:2}
The map $s_0\theta\colon s_0\KQ \rTo s_0\KGL^{hC_2}$ is the identity on the summand $\MZZ$.
\item\label{item:3}
The map $h^\prime \circ f^\prime \colon \KGL^{hC_2} \rTo \KGL^{hC_2}$ is multiplication with $1-\epsilon$. 
\item\label{item:4} 
The map $f^\prime \circ h^\prime \colon \KGL\rTo \KGL$ coincides with $1+\Psi^{-1}_\mathrm{st}$.
\end{enumerate}
To prove the first claim, 
consider the homotopy fiber of the canonical map
$\theta \colon \KQ \rTo \KGL^{hC_{2}}$. The Tate diagram
\cite[Diagram (20)]{Hu-Kriz-Ormsby} shows that it coincides
with the homotopy fiber of the canonical map
$\KT \rTo \widehat{\KGL}^{C_{2}}$. Since $\eta$ acts
invertibly on these two motivic spectra, it acts 
invertibly on the homotopy fiber of $\theta$.
This implies that the diagram
\begin{diagram}
\Sigma^{1,1}  \KQ & \rTo^\eta & \KQ  \\
\dTo^{\Sigma^{2,1}  \theta} & & \dTo^\theta \\
\Sigma^{1,1}  \KGL^{hC_{2}} & \rTo^\eta & \KGL^{hC_{2}} 
\end{diagram}
is a homotopy pullback square, and hence the first claim.
See also \cite[Remark 5.9]{Isaksen-Shkembi}.
The second claim follows from the factorization of the unit map for algebraic $K$-theory
\[ 
\One 
\rTo^{\mathrm{unit}} 
\KQ 
\rTo^\theta  
\KGL^{hC_2}  
\rTo^{f^\prime} 
\KGL.
\]

By the proof of Lemma~\ref{lemma:multiplicationbyfh} the third claim follows from the commutative diagram:
\begin{diagram}
\KQ & \rTo^{f} & \KGL & \rTo^h & \KQ \\
\dTo^\theta & & \dTo^\id & & \dTo_\theta\\
\KGL^{hC_2} & \rTo^{f^\prime} & \KGL & \rTo^{h^\prime} & \KGL^{hC_2}
\end{diagram}

The previous commutative diagram and~(\ref{eq:fh}) imply the fourth and final claim.
Using the above one may now compute the slices of $\KGL^{hC_2}$ in the same way as for $\KQ$, 
cf.~Section~\ref{section:slic-herm-algebr}.
That is, 
one obtains an identification 
\[  
\s_{q}(\KGL^{hC_2})
=
\begin{cases}
\bigl(\Sigma^{2q,q}  \MZZ\bigr) 
\vee 
\bigl(\Sigma^{q,q}  \Mnu\bigr) 
\vee
\bigvee_{i<\frac{q}{2}} \Sigma^{2i+q,q}  \MZZ/2 & q\equiv 0 \bmod 2 \\
\bigl(\Sigma^{q,q}  \Mnu\bigr) 
\vee 
\bigvee_{i<\frac{q+1}{2}} \Sigma^{2i+q,q}  \MZZ/2 & q\equiv 1 \bmod 2.
\end{cases}
\]
Here $\Mnu\cong \Sigma^{4,0} \Mnu$ and $\Mnu_{s,t}$ is an $\FF_{2}$-module for 
all integers $s$ and $t$.
Moreover, 
it follows from our first claim above that $\theta\colon \KQ\rTo \KGL^{hC_2}$ 
splits as $(\id,\zeta)$, 
where $\id$ is the identity on the non-mysterious summands of the respective slices, 
and $\zeta\colon \Mmu\rTo \Mnu$ up to suspension with $S^{1,1}$. 
By comparison with the identification of the slices of 
$\KGL^{hC_2}$ in Theorem~\ref{thm:slices-kgl-hofix}, 
the summand $\Mnu$ is contractible, 
which completes the proof.
\end{proof}

\begin{theorem}
\label{theorem:mysteriousvanishing}
The mysterious summand $\Mmu$ is trivial.
\end{theorem}
\begin{proof}
By Corollary~\ref{cor:slice-ess-smooth-base-change}
it suffices to consider prime fields.
Theorems \ref{thm:homotopy-limit-prime-fields}, \ref{thm:slices-kgl-hofix} and Lemma \ref{lem:map-myst} 
show that the mysterious summand is the homotopy fiber of a weak equivalence. 
\end{proof}

\subsection{Higher Witt-theory}
\label{section:slicesWitttheory}
By combining Theorems \ref{theorem:kobalsequence}, 
\ref{thm:slices-homotopy-orbit-K}, 
\ref{thm:slices-ko}, 
and \ref{theorem:mysteriousvanishing} we have enough information to identify the slices of Witt-theory.
An alternate proof with no mention of homotopy orbit $K$-theory follows by Lemma \ref{lemma:sliceshomotopycolimits} 
and the identification of $\s_q(\eta)\colon \s_q(\Sigma^{1,1} \KQ)\rTo \s_q(\KQ)$ worked out in \S \ref{section:slic-herm-algebr}.
\begin{theorem}
\label{thm:slices-witt-theory}
The slices of Witt-theory are given by
\[ 
\s_{q}(\KT) 
\iso 
\Sigma^{q,q}  \bigvee_{i\in \ZZ} \Sigma^{2i,0}  \MZZ/2. 
\]
\end{theorem}

Let $u\colon\KQ\rTo\KT$ be the canonical map.
The next result follows now by an easy inspection.
\begin{proposition}
\label{prop:slices-unit-kt}
Restricting the map $\s_{2q}(u)\colon \s_{2q}(\KQ)\rTo\s_{2q}(\KT)$ to the summand $\Sigma^{4q,2q}  \MZZ$ yields the projection 
$\Sigma^{4q,2q}  \MZZ \rTo \Sigma^{4q,2q}  \MZZ/2$ composed with the inclusion into $\s_{2q}(\KT)$.
Restricting the same map to a suspension of $\MZZ/2$ yields the inclusion into $\s_{2q}(\KT)$.
On odd slices, 
$\s_{2q+1}(u)\colon \s_{2q+1}(\KQ) \rTo \s_{2q+1}(\KT)$ is the inclusion.
\end{proposition}

\section{Differentials}
\label{section:differentialsI}
Having determined the slices we now turn to the problem of computing the 
first differentials in the slice spectral sequences for 
$\KGL_{hC_{2}}$, $\KQ$ and $\KT$.

\subsection{Mod-2 algebraic $K$-theory}
\label{subsection:diff-algebraic-k-theory}
Theorem \ref{thm:slices-k} and Lemma \ref{lemma:slices-mod2} show there is an isomorphism 
\[
\s_{q}(\KGL/2) 
\iso 
\Sigma^{2q,q}  \MZZ/2.
\] 
Hence the differential 
\[
d^{\KGL/2}_{1}
\colon 
\s_{q}(\KGL/2) 
\rTo 
\Sigma^{1,0}   \s_{q+1}(\KGL/2)
\] 
corresponds to a bidegree $(3,1)$ element in the mod-2 motivic Steenrod algebra.
By Bott periodicity it is independent of $q$.
Next we resolve the mod-2 version of a question stated in \cite[Remark 3.12]{Voevodsky:seattle}.
\begin{lemma}
\label{lemma:KGL2differential}
If $\Char(F)\neq 2$ the differential 
\[ 
d^{\KGL/2}_{1} 
\colon 
\MZZ/2\rTo \Sigma^{3,1}  \MZZ/2
\]
equals the Milnor operation $\Qop_{1}=\Sq^{3}+\Sq^{2}\Sq^{1}$.
\end{lemma}
\begin{proof}
By Lemma \ref{lemma:mod2tomod2weight1} and the Adem relation $\Sq^{3}=\Sq^{1}\Sq^{2}$ there exist $a,b\in \ZZ/2$ such that 
\begin{equation}
\label{equation:d1KGL/2}
d^{\KGL/2}_{1} 
= 
a\Sq^{3} +b\Sq^{2}\Sq^{1}.
\end{equation}
Using the Adem relations $\Sq^{3}\Sq^{3}=\Sq^5\Sq^{1}$, 
$\Sq^{2}\Sq^{3}=\Sq^5+\Sq^{4}\Sq^{1}$, 
and $\Sq^{1}\Sq^{1}=0$ we find
\begin{align*}
(a\Sq^{3}+b\Sq^{2}\Sq^{1})^{2}  
= 
&
a^{2}\Sq^{3}\Sq^{3}
+ab\Sq^{3}\Sq^{2}\Sq^{1} 
+ab\Sq^{2}\Sq^{1}\Sq^{3} 
+b^{2} \Sq^{2}\Sq^{1}\Sq^{2}\Sq^{1} \\
= 
& 
(a^{2}+ b^{2})\Sq^5\Sq^{1}.
\end{align*}
Since $d^{\KGL/2}_{1}$ squares to zero this implies $a=b$.
Recall the classes $0\neq\tau\in h^{0,1}$ and $\rho\in h^{1,1}$ represented by $-1\in F$.
Over $\RR$, 
and hence $\QQ$, 
$d^{\KGL/2}_{1}(\tau^{2})=\rho^{3}$ by Suslin's computation of the algebraic $K$-theory of the real numbers \cite{Suslinlocal}, 
while 
\[
d^{\KGL/2}_{1}(\tau^{2})
=
a\Sq^{3}(\tau^{2})+b\Sq^{2}\Sq^{1}(\tau^{2})
=
a\rho^{3}
\]
by (\ref{equation:d1KGL/2}) and Corollary \ref{cor:sq-weight-1}.
These computations imply that $a=b=1$ by base change for all fields of characteristic zero.
To extend this result to fields of odd characteristic, 
we consider
the following diagram of motivic spectra over $\Spec(\ZZ[\tfrac{1}{2}])$.
\begin{equation}
  \label{eq:dedekind}
  \begin{diagram}
    \KGL/2 & \lTo & \MGL/(2,x_2,x_3,\dotsc)\MGL &  \rTo & \MGL/(2,x_1,x_2,\dotsc)\MGL & \rTo & \MZZ/2
  \end{diagram}
\end{equation}
Here $\MGL$ denotes Voevodsky's algebraic cobordism spectrum, $x_n$
denotes the canonical image of Lazard's generator in $\MGL_{\ast,\ast}$, 
$\MZZ/2$ denotes Spitzweck's motivic cohomology spectrum 
with $\mathbb{F}_2$-coefficients \cite{SpitzweckMZ}, 
and the maps are induced by the respective canonical orientations.
All maps in diagram~(\ref{eq:dedekind}) induce
equivalences on zero slices. For the map pointing to the left, this
follows from the description of $\KGL$ as 
$\bigl(\MGL/(x_2,x_3,\dotsc)\MGL\bigr)[x_1^{-1}]$ 
\cite[Theorem 5.2]{Spitzweck:slices}. For the map in the middle, 
this follows from its construction.
By \cite[Theorem 11.3]{SpitzweckMZ}, the rightmost map in diagram~(\ref{eq:dedekind}) is an equivalence, and in particular after applying the
zero slice functor. There results a commutative 
diagram
\begin{equation}\label{eq:dedekind-slice}
\begin{diagram}[width=2cm]
\s_0\KGL/2 & \rTo^{\sliced_1} & \Sigma^{3,1}  \s_0\KGL/2 \\
\dTo^\iso & & \dTo_\iso \\
\s_0 \MZZ/2& \rTo & \Sigma^{3,1} \s_0\MZZ/2
\end{diagram}
\end{equation}
where the map on the top is defined as over a field.
By its construction, the motivic spectrum $\MZZ/2$ satisfies
$\f_{1}\MZZ/2 = \ast$, and the aforementioned  \cite[Theorem 11.3]{SpitzweckMZ}
implies that $\MZZ/2$ is effective over $\Spec(\ZZ[\tfrac{1}{2}])$.
There results a canonical isomorphism 
$\s_0 \MZZ/2 \iso \MZZ/2$ which is a map of motivic ring spectra
by \cite[Theorem 5.19]{GRSO}. 
Base change via $f\colon \Spec(\QQ) \rTo \Spec(\ZZ[\tfrac{1}{2}])$ 
maps diagram~(\ref{eq:dedekind-slice}) 
to the following
diagram over $\Spec(\QQ)$, by Lemma~\ref{lemma:slice-ess-smooth-base-change},
\cite[Lemma 7.5]{SpitzweckMZ} and the previous argument.
\begin{diagram}[width=2cm]
\s_0\KGL/2 & \rTo^{\sliced_1} & \Sigma^{3,1}  \s_0\KGL/2 \\
\dTo^\iso & & \dTo_\iso \\
\MZZ/2& \rTo^{\Qop_{1}} & \Sigma^{3,1} \MZZ/2 \\
\end{diagram}
If $p$ is an odd prime, base change via $i\colon \Spec(\FF_p)\rClosed
\Spec(\ZZ[\tfrac{1}{2}])$ maps 
diagram~(\ref{eq:dedekind-slice}) to the following diagram.
  \begin{diagram}[width=2cm]
    i^\ast\s_0\KGL/2 & \rTo^{i^\ast \sliced_1} & \Sigma^{3,1}  i^\ast\s_0\KGL/2 \\
    \dTo^\iso & & \dTo_\iso \\
    i^\ast\MZZ/2& \rTo & \Sigma^{3,1}  i^\ast\MZZ/2 \\
  \end{diagram}
By \cite[Theorem 9.19]{SpitzweckMZ}, 
$i^\ast \MZZ/2$ is Voevodsky's Eilenberg-MacLane motivic ring spectrum
over $\FF_p$, and the bottom map in the last diagram
is 
the first Milnor operation by \cite[Theorem 11.24]{SpitzweckMZ}.
\end{proof}

Thus the $E^{1}$-page of the $0$th slice spectral sequence for $\KGL/2$ takes the form:
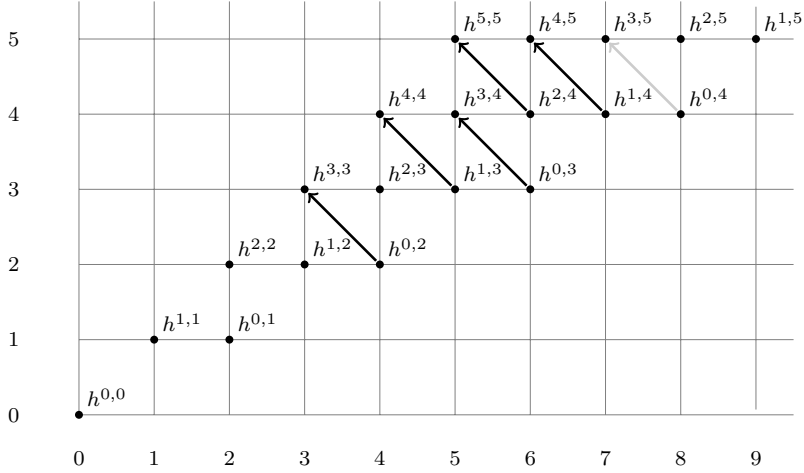
\begin{center}
  \pgfsetshortenend{2pt}
  \pgfsetshortenstart{2pt}
\begin{tikzpicture}[scale=1.0,font=\scriptsize,line width=1pt]
\draw[help lines] (0,0) grid (9.5,5.5);
\foreach \i in {0,...,5} {\node[label=left:$\i$] at (-.5,\i) {};}
\foreach \i in {0,...,9} {\node[label=below:$\i$] at (\i,-.2) {};}

{\draw[fill]     
(0,0) circle (1pt) node[above right=-1pt] {$h^{0,0}$}
(1,1) circle (1pt) node[above right=-1pt] {$h^{1,1}$}
(2,1) circle (1pt) node[above right=-1pt] {$h^{0,1}$}
(2,2) circle (1pt) node[above right=-1pt] {$h^{2,2}$}
(3,2) circle (1pt) node[above right=-1pt] {$h^{1,2}$}
(4,2) circle (1pt) node[above right=-1pt] {$h^{0,2}$}
(3,3) circle (1pt) node[above right=-1pt] {$h^{3,3}$}
(4,3) circle (1pt) node[above right=-1pt] {$h^{2,3}$}
(5,3) circle (1pt) node[above right=-1pt] {$h^{1,3}$}
(6,3) circle (1pt) node[above right=-1pt] {$h^{0,3}$}
(4,4) circle (1pt) node[above right=-1pt] {$h^{4,4}$}
(5,4) circle (1pt) node[above right=-1pt] {$h^{3,4}$}
(6,4) circle (1pt) node[above right=-1pt] {$h^{2,4}$}
(7,4) circle (1pt) node[above right=-1pt] {$h^{1,4}$}
(8,4) circle (1pt) node[above right=-1pt] {$h^{0,4}$}
(5,5) circle (1pt) node[above right=-1pt] {$h^{5,5}$}
(6,5) circle (1pt) node[above right=-1pt] {$h^{4,5}$}
(7,5) circle (1pt) node[above right=-1pt] {$h^{3,5}$}
(8,5) circle (1pt) node[above right=-1pt] {$h^{2,5}$}
(9,5) circle (1pt) node[above right=-1pt] {$h^{1,5}$}
;}

{
\draw[fill,->]
  (4,2) -- (3,3);
\draw[fill,->]
  (5,3) -> (4,4);
\draw[fill,->]
  (6,3) -> (5,4);
\draw[fill,->]
  (6,4) -> (5,5);
\draw[fill,->]
  (7,4) -> (6,5);
\draw[black!20!white,->]
  (8,4) -> (7,5);
}
\end{tikzpicture}
\end{center}

The differential $h^{0,4}\rTo h^{3,5}$ has a nontrivial group
as source and a potentially nontrivial group as target, but
is always zero by Corollary~\ref{cor:sq-weight-1}.

\begin{corollary}
If $\sqrt{-1}\in F$ then $d^{\KGL/2}_{1}\colon \pi_{p,n}\s_q(\KGL/2) \rTo \pi_{p-1,n}\s_{q+1}(\KGL/2)$ is trivial. 
\end{corollary}
\begin{proof}
By assumption $\rho=0$, 
so the assertion follows from Lemma \ref{lemma:KGL2differential} and Corollary \ref{cor:sq-weight-1}.
\end{proof}

\subsection{Hermitian $K$-theory, I}
\label{subsection:diff-hermitian-k-theory}
In what follows we make of use the identification of $d^{\KGL/2}_{1}$ with the first Milnor operation $\Qop_{1}$ to give formulas for the 
$\sliced_{1}$-differential in the slice spectral sequence for $\KQ$.
To begin with, 
consider the commutative diagram:
\begin{equation}
\label{eq:diff-ko-k}
\begin{diagram}
\s_{q}(\KQ) & \rTo^{d^{\KQ}_{1}(q)} & \Sigma^{1,0}  \s_{q+1}(\KQ) \\
\dTo & & \dTo \\
\s_{q}(\KGL) & \rTo^{d^{\KGL}_{1}(q)} & \Sigma^{1,0}  \s_{q+1}(\KGL) \\
\dTo & & \dTo \\
\s_{q}(\KGL/2) & \rTo^{\Qop_{1}} & \Sigma^{1,0}  \s_{q+1}(\KGL/2) 
\end{diagram}
\end{equation}

To proceed from here requires two separate arguments depending on whether the $\sliced_{1}$-differential exits an even or an odd slice.
First we analyze the even slices.
\vspace{0.1in}

Theorems \ref{thm:slices-ko} and \ref{theorem:mysteriousvanishing} show that $\sliced_{1}^{\KQ}(2q)$ is a map from 
\[
\s_{2q}(\KQ)
= 
\Sigma^{2q,2q} \bigl(\Sigma^{2q,0} \MZZ\vee\bigvee_{i<0} \Sigma^{2q+2i,0}  \MZZ/2\bigr)
\]
to
\[
\Sigma^{1,0}  \s_{2q+1}(\KQ)
= 
\Sigma^{2q+1,2q+1} \bigvee_{i\leq 0} \Sigma^{2q+2i+1,0}  \MZZ/2. 
\]
Up to suspension with $S^{4q,2q}$, 
Proposition \ref{prop:slices-forget-hyper} shows $\s_{2q}(\KQ)\rTo\s_{2q}(\KGL)\rTo\s_{2q}(\KGL/2)$ in (\ref{eq:diff-ko-k}) 
is the projection map
\[ 
\MZZ\vee\bigvee_{i<0} \Sigma^{2i,0}  \MZZ/2 
\rTo 
\MZZ 
\rTo
\MZZ/2. 
\]
Likewise, 
up to suspension with $S^{4q,2q}$, 
the composite of the right hand side maps in (\ref{eq:diff-ko-k}) equals
\[ 
\Sigma^{3,1} \bigvee_{i<0} \Sigma^{2i+1,0}  \MZZ/2 
\rTo^{\pr} 
\Sigma^{2,1} \MZZ/2 
\rTo^\delta 
\Sigma^{3,1} \MZZ 
\rTo^{\pr} 
\Sigma^{3,1} \MZZ/2. 
\]
By Lemma \ref{lemma:integraltomod2andmod2tointegralweight1}, 
$\Sq^{2}\circ\mathrm{pr}$ is the only element of $[\MZZ,\Sigma^{2,1} \MZZ/2]$ whose composition with $\Sq^{1}$ and the projection 
$\MZZ\rTo\MZZ/2$ yields $\Qop_{1}\circ\mathrm{pr} = \Sq^{1}\Sq^{2}\circ\mathrm{pr}$. 
When restricting to $\MZZ$ this shows 
\[
d^{\KQ}_{1}(2q) 
= 
(0,\Sq^{2}\circ\mathrm{pr},a(2q)\tau\circ\mathrm{pr});
\;
a(2q)\in h^{0,0}.
\] 
On odd slices the first differential $\sliced_{1}^{\KQ}(2q+1)$ is a map from 
\[
\s_{2q+1}(\KQ)
= 
\Sigma^{2q+1,2q+1} \bigvee_{i\leq 0} \Sigma^{2q+2i,0}  \MZZ/2
\]
to
\[
\Sigma^{1,0}  \s_{2q+2}(\KQ)
= 
\Sigma^{2q+2,2q+2} \bigl(\Sigma^{2q+3,0} \MZZ\vee\bigvee_{i\leq 0} \Sigma^{2q+2i+1,0}  \MZZ/2\bigr).
\]
Restricting $\sliced_{1}^{\KQ}(2q+1)$ to the top summand $\MZZ/2$ corresponding to $i=0$ in the infinite wedge product yields 
-- up to suspension with $S^{4q+1,2q+1}$ -- 
a map
\[ 
\sliced_{1}^{\KQ}(2q+1,0)
\colon 
\MZZ/2 
\rTo 
\Sigma^{4,1}  \MZZ \vee \Sigma^{2,1}  \MZZ/2 \vee \Sigma^{0,1}  \MZZ/2. 
\]
Lemma \ref{lemma:integraltomod2andmod2tointegralweight1} shows $\delta\Sq^{2}\Sq^{1}$ is the only element of $[\MZZ/2,\Sigma^{4,1} \MZZ]$ 
whose composition with the projection map $\mathrm{pr}$ equals $\Qop_{1}\Sq^{1}=\Sq^{3}\Sq^{1}$.
Thus by commutativity of (\ref{eq:diff-ko-k}), 
we obtain
\[
\sliced_{1}^{\KQ}(2q+1) 
= 
(\delta\Sq^{2}\Sq^{1},b(2q+1)\Sq^{2}+\phi(2q+1)\Sq^{1},a(2q+1)\tau),
\]
where $a(2q+1)$, $b(2q+1)\in h^{0,0}$ and $\phi(2q+1)\in h^{1,1}$.
\vspace{0.1in}

In what follows we use the computations in this section to explicitly identify $\sliced_{1}^{\KT}$ and $\sliced_{1}^{\KQ}$ in terms of 
motivic cohomology classes and Steenrod operations.

\subsection{Higher Witt-theory}
Up to suspension with $S^{q,q}$, 
Theorem \ref{section:slicesWitttheory} shows that the differential
\[ 
\sliced_{1}^{\KT}(q)
\colon 
\s_{q}(\KT) 
\rTo \Sigma^{1,0}  \s_{q+1}(\KT) 
\]
takes the form
\[ 
\bigvee_{i\in \ZZ} \Sigma^{2i,0}  \MZZ/2 
\rTo 
\Sigma^{2,1}  \bigvee_{j\in \ZZ} \Sigma^{2j,0} \MZZ/2. 
\]
Lemma \ref{lemma:shift-slices} and $(1,1)$-periodicity $\Sigma^{1,1}  \KT\iso \KT$ for $\KT$ imply 
$\sliced_{1}^{\KT}(q)= \Sigma^{q,q}  \sliced_{1}^{\KT}(0)$.
Let $\sliced_{1}^{\KT}(q,2i)$ denote the restriction of $\sliced_{1}^{\KT}(q)$ to the summand $\Sigma^{q+2i,q}  \MZZ/2$ of $\s_{q}(\KT)$. 
The $(4,0)$-periodicity $\Sigma^{4,0}  \KT\iso \KT$ shows $\sliced_{1}^{\KT}(0)$ is determined by its values on the summands
$\MZZ/2$ and $\Sigma^{2,0}  \MZZ/2$.
That is, 
$\sliced_{1}^{\KT}(q,2i)$ is uniquely determined by $\sliced_{1}^{\KT}(0,0)$ or $\sliced_{1}^{\KT}(0,2)$.
By Proposition~\ref{prop:sum-product} the first differential is determined by its value on each summand in its target. 
\begin{theorem}
\label{thm:diff-witt-theory}
The $\sliced_{1}$-differential in the slice spectral sequence for $\KT$ is given by
\[ 
\sliced_{1}^{\KT}(q,i) 
= 
\begin{cases} 
(\Sq^{3}\Sq^{1},\Sq^{2},0) & i-2q \equiv 0\bmod 4 \\
(\Sq^{3}\Sq^{1},\Sq^{2}+\rho\Sq^{1},\tau) & i-2q \equiv 2\bmod 4.
\end{cases} 
\]
The motivic cohomology classes $0\neq\tau\in h^{0,1}$ and $\rho=[-1]\in h^{1,1}$ are represented by $-1\in F$.
\end{theorem}
\begin{proof}
We have reduced to computing the maps
\begin{diagram}
\sliced_{1}^{\KT}(0,0)\colon \MZZ/2 & \rTo & \Sigma^{4,1}  \MZZ/2 \vee \Sigma^{2,1}  \MZZ/2 \vee \Sigma^{0,1}  \MZZ/2 \\
\sliced_{1}^{\KT}(0,2)\colon \Sigma^{2,0}  \MZZ/2 & \rTo & \Sigma^{6,1}  \MZZ/2 \vee \Sigma^{4,1}  \MZZ/2 \vee \Sigma^{2,1}  \MZZ/2. 
\end{diagram}  
To proceed we invoke the commutative diagram:
\begin{diagram}
\s_{q}(\KQ) & \rTo & \s_{q}(\KT) \\
\dTo^{\sliced_{1}^{\KQ}(q)}& & \dTo_{\sliced_{1}^{\KT}(q)} \\
\Sigma^{1,0}  \s_{q+1}(\KQ) & \rTo & \Sigma^{1,0}  \s_{q+1}(\KT) 
\end{diagram}
By combining Proposition \ref{prop:slices-unit-kt} with the computations in Section~\ref{subsection:diff-hermitian-k-theory} we find
\begin{align*}
\sliced_{1}^{\KT}(0,0) & = (\Sq^{3}\Sq^{1},\Sq^{2}+\phi(0)\Sq^{1},a(0)\tau) \\
\sliced_{1}^{\KT}(0,2) & = (\Sq^{3}\Sq^{1},\Sq^{2}+\phi(2)\Sq^{1},a(2)\tau). 
\end{align*}
Here $a(0)$, $a(2)\in h^{0,0}$ and $\phi(0)$, $\phi(2)\in h^{1,1}$.
Since $\sliced_{1}^{\KT}$ squares to zero, 
we extract the equations
\[
a(2)a(0)=
a(0)\phi(2)=
a(0)(\phi(0)+\phi(2)+\rho)=0,
\;\;
a(0)+a(2)=1,
\]
\[
a(0)\rho+\phi(0)=
a(2)\phi(0)=
a(2)(\phi(0)+\phi(2)+\rho)=
a(2)\rho+\phi(2)=
0.
\]
The proof of Proposition~\ref{prop:convergence-witt-0} discusses
the two possible values for $a(2)$ in the case of an algebraically
closed field. In particular, having $a(2)=0$ would result in 
the group $\pi_{2,0}\KT$ being nontrivial, which would contradict
the vanishing of Balmer's higher Witt groups of fields 
in degrees not congruent to zero modulo four \cite[Theorem 98]{Balmer:handbook}. 
Hence $a(2)=1$ for algebraically closed fields, which extends
to all fields by base change to an algebraic closure.
This implies $a(0)=0$, $\phi(0)=0$, and $\phi(2)=\rho$. 
\end{proof}

\begin{proposition}\label{prop:convergence-witt-0}
Suppose $F$ has mod-$2$ cohomological dimension zero. 
Then the slice filtration for $\KT$ coincides with the fundamental ideal filtration of $W(F)$. 
Moreover,
there are isomorphisms
\[
\pi_{p,0}\f_q(\KT) 
\cong
\begin{cases} 
h^{0,q} & p \equiv q \bmod 4, q\geq 0 \\ 
h^{0,0} & p\equiv 0\bmod 4,q<0 \\
0 & \mathrm{otherwise.}
\end{cases}
\]
The first two isomorphisms are induced by the canonical map $\f_q(\KT)\rTo \s_q(\KT)$.
\end{proposition}

\begin{proof}
By the $(1,1)$- and $(4,0)$-periodicities of $\KT$ and Theorem \ref{thm:slices-witt-theory}, 
it suffices to consider the filtration
\begin{diagram}
\dotsm & & h^{0,n+1} & & 0 & & \\
& & \dTo & & \dTo & & \\
\dotsm & \rTo & \pi_{p,-n}\f_{2}(\KT) & \rTo & \pi_{p,-n}\f_{1}(\KT) & \rTo & \pi_{p,-n}\f_{0}(\KT) \\
& & \dTo & & \dTo & & \dTo \\
\dotsm & & h^{0,n+2} & & 0 & & h^{0,n}
\end{diagram}
for $p$ even, 
and 
\begin{diagram}
\dotsm & & 0 & & h^{0,n} & & \\
& & \dTo & & \dTo & & \\
\dotsm & \rTo & \pi_{p,-n}\f_{2}(\KT) & \rTo & \pi_{p,-n}\f_{1}(\KT) & \rTo & \pi_{p,-n}\f_{0}(\KT) \\
& & \dTo & & \dTo & & \dTo \\
\dotsm   & & 0 & & h^{0,n+1} & & 0
\end{diagram}
for $p$ odd.
When $n=p=0$, 
$\pi_{0,0}\f_{0}(\KT)\rTo h^{0,0}$ is a ring map \cite{GRSO}, \cite{Pelaez},
hence an isomorphism.
It follows that $\pi_{0,0}\f_{1}(\KT)=0$, 
hence $\f_q \pi_{0,0}(\KT) = 0$ for all $q>0$.
In particular, 
the slice filtration is Hausdorff, 
and it coincides with the (trivial) filtration on the Witt ring given by the fundamental ideal.
Theorem \ref{thm:diff-witt-theory} and Lemma~\ref{lem:tau-iso} leave us with two possibilities:
\begin{enumerate}
\item 
The first differential
\[ 
E^{1}_{p,q} 
=
h^{0,q} 
\rTo 
h^{0,q+1} 
=
E^{1}_{p-1,q+1} 
\]
is an isomorphism if $q\equiv p+2\bmod 4$, 
and trivial otherwise.
\item 
The first differential
\[ 
E^{1}_{p,q} 
= 
h^{0,q} 
\rTo 
h^{0,q+1} 
=
E^{1}_{p-1,q+1} 
\]
is an isomorphism if $q\equiv p\bmod 4$, 
and trivial otherwise.
\end{enumerate}
In both cases we find $E^{2}=E^{\infty}$.
Our case distinctions can be recast in the following way.
\begin{enumerate}
\item 
$E^\infty_{p,q}\cong h^{0,0}$ if $q=0$ and $p\equiv 0\bmod 4$, 
and trivial otherwise.
\item
$E^\infty_{p,q}\cong h^{0,0}$ if $q=0$ and $p\equiv 2 \bmod 4 $, 
and trivial otherwise.
\end{enumerate}
The second case contradicts the vanishing of $\pi_{2,0}(\KT)$. 
Computing with the first condition yields the desired result on the filtration.
One extends to arbitrary $n$ by using the commutative diagram:
\begin{diagram}
\pi_{p,n} \f_{q+i}(\KT) & \rTo & \pi_{p,n} \f_{q}(\KT) \\
\dTo^{\iso} & & \dTo_\iso \\
\pi_{p-n,0} \f_{q-n+i}(\KT) & \rTo & \pi_{p-n,0} \f_{q-n}(\KT) 
\end{diagram}
\end{proof}

The following image depicts the first differential when $2q\equiv i\bmod 4$ with degrees along the horizontal axis and weights 
along the vertical axis.
Each dot is a suspension of $\MZZ/2$.
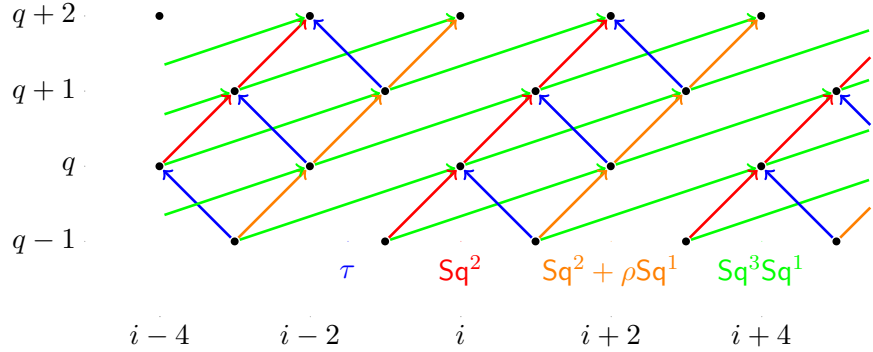
\begin{center}
  \pgfsetshortenend{2pt}
  \pgfsetshortenstart{2pt}
\begin{tikzpicture}[scale=1.0,line width=1pt]
{\draw[fill]     
(0,0) circle (1pt) (-2,0) circle (1pt) (-4,0) circle (1pt) (2,0) circle (1pt) (4,0) circle (1pt) 
(-5,0) circle (0pt) node[left=-1pt] {$q$}
;}
{\draw[fill]     
(1,1) circle (1pt) (-1,1) circle (1pt) (-3,1) circle (1pt) (3,1) circle (1pt) (5,1) circle (1pt) 
(-5,1) circle (0pt) node[left=-1pt] {$q+1$}
;}
{\draw[fill]     
(1,-1) circle (1pt) (-1,-1) circle (1pt) (-3,-1) circle (1pt) (3,-1) circle (1pt) (5,-1) circle (1pt) 
(-5,-1) circle (0pt) node[left=-1pt] {$q-1$}
;}
{\draw[fill]     
(0,2) circle (1pt) (-2,2) circle (1pt) (-4,2) circle (1pt) (2,2) circle (1pt) (4,2) circle (1pt) 
(-5,2) circle (0pt) node[left=-1pt] {$q+2$}
;}
{\draw[fill]     
(0,-2) circle (0pt) node[below=-1pt] {$i$} 
(-2,-2) circle (0pt) node[below=-1pt] {$i-2$}
(-4,-2) circle (0pt) node[below=-1pt] {$i-4$}
(2,-2) circle (0pt) node[below=-1pt] {$i+2$}
(4,-2) circle (0pt) node[below=-1pt] {$i+4$}
;}
{\draw[fill,blue]     
(-1.5,-1) circle (0pt) node[below=5pt] {$\tau$}
;}
{\draw[fill,red]     
(0,-1) circle (0pt) node[below=1pt] {$\Sq^2$}
;}
{\draw[fill,orange]     
(2,-1) circle (0pt) node[below=1pt] {$\Sq^2+\rho\Sq^1$}
;}
{\draw[fill,green]     
(4,-1) circle (0pt) node[below=1pt] {$\Sq^3\Sq^1$}
;}

{\draw[blue,->] 
(-3,-1) -- (-4,0)
;}
{\draw[orange,->] 
(-3,-1) -- (-2,0)
;}
{\draw[green,->] 
(-3,-1) -- (0,0)
;}
{\draw[red,->] 
(-1,-1) -- (0,0)
;}
{\draw[green,->] 
(-1,-1) -- (2,0)
;}
{\draw[blue,->] 
(1,-1) -- (0,0)
;}
{\draw[orange,->] 
(1,-1) -- (2,0)
;}
{\draw[green,->] 
(1,-1) -- (4,0)
;}
{\draw[red,->] 
(3,-1) -- (4,0)
;}
{\draw[green] 
(3,-1) -- (5.5,-.16)
;}
{\draw[blue,->]
(5,-1) -- (4,0)
;}
{\draw[orange]
(5,-1) -- (5.5,-.5)
;}
{\draw[green,->] 
(-4,-.67) -- (-2,0)
;}
{\draw[green,->] 
(-4,.67) -- (-3,1)
;}
{\draw[red,->] 
(-4,0) -- (-3,1)
;}
{\draw[green,->] 
(-4,0) -- (-1,1)
;}
{\draw[blue,->] 
(-2,0) -- (-3,1)
;}
{\draw[orange,->] 
(-2,0) -- (-1,1)
;}
{\draw[green,->] 
(-2,0) -- (1,1)
;}
{\draw[red,->] 
(0,0) -- (1,1)
;}
{\draw[green,->] 
(0,0) -- (3,1)
;}
{\draw[blue,->] 
(2,0) -- (1,1)
;}
{\draw[orange,->] 
(2,0) -- (3,1)
;}
{\draw[green,->] 
(2,0) -- (5,1)
;}
{\draw[red,->] 
(4,0) -- (5,1)
;}
{\draw[green] 
(4,0) -- (5.5,.5)
;}
{\draw[blue,->]
(5.5,0.5) -- (5,1)
;}
{\draw[red,->] 
(-3,1) -- (-2,2)
;}
{\draw[green,->] 
(-3,1) -- (0,2)
;}
{\draw[green,->] 
(-4,1.33) -- (-2,2)
;}
{\draw[blue,->] 
(-1,1) -- (-2,2)
;}
{\draw[orange,->] 
(-1,1) -- (0,2)
;}
{\draw[green,->] 
(-1,1) -- (2,2)
;}
{\draw[red,->] 
(1,1) -- (2,2)
;}
{\draw[green,->] 
(1,1) -- (4,2)
;}
{\draw[blue,->] 
(3,1) -- (2,2)
;}
{\draw[orange,->] 
(3,1) -- (4,2)
;}
{\draw[green] 
(3,1) -- (5.5,1.83)
;}
{\draw[red]
(5,1) -- (5.5,1.5)
;}
{\draw[green] 
(5,1) -- (5.5,1.17)
;}

\end{tikzpicture}
\end{center}

\subsection{Hermitian $K$-theory, II}
\label{section:diff-hermitian-k-theory-2}
Next we determine $\sliced_{1}^{\KQ}$ by combining Proposition \ref{prop:slices-unit-kt} with the formula for $\sliced_{1}^{\KT}$ in Theorem \ref{thm:diff-witt-theory}.
Recall the identifications 
\begin{align*} 
\s_{2q}(\KQ) 
& 
\iso 
\Sigma^{2q,2q}  \bigl(\Sigma^{2q,0}  \MZZ \vee \bigvee_{i<q} \Sigma^{2i,0}  \MZZ/2\bigr) \\
\s_{2q+1}(\KQ) 
& 
\iso \Sigma^{2q+1,2q+1}  \bigvee_{i\leq q} \Sigma^{2i,0}  \MZZ/2. 
\end{align*}

Let $\sliced_{1}^{\KQ}(q,2i)$ denote the restriction of $\sliced_{1}^{\KQ}(q)$ to the $2i$th summand $\Sigma^{2i,0}  \MZZ/2$ of $\s_{q}(\KQ)$, 
where $i\leq\lfloor\frac{q}{2}\rfloor$. 
There are canonical maps $\delta\colon\MZZ/2\rTo \Sigma^{1,0}  \MZZ$ and $\pr\colon\MZZ\rTo\MZZ/2$, 
cf.~Appendix \ref{section:endom-motiv-eilenb}.
Inspection of the computations in Section~\ref{subsection:diff-hermitian-k-theory} and Theorem~\ref{thm:diff-witt-theory} yields: 
\begin{theorem}
\label{thm:diff-ko}
The $\sliced_{1}$-differential in the slice spectral sequence for $\KQ$ is given by
\begin{align*}
\sliced_{1}^{\KQ}(q,i) 
& =  
\begin{cases}
(\Sq^{3}\Sq^{1},\Sq^{2},0) &  q-1> i\equiv 0\bmod 4 \\
(\Sq^{3}\Sq^{1},\Sq^{2}+\rho\Sq^{1},\tau) &  q-1> i\equiv 2\bmod 4 \\
\end{cases} \\
d^{\KQ}_{1}(q,q) 
& =  
\begin{cases}
(0,\Sq^{2}\circ\pr,0) & q\equiv 0 \bmod 4 \\
(0,\Sq^{2}\circ \pr,\tau\circ \pr) & q\equiv 2 \bmod 4 \\
\end{cases} \\
d^{\KQ}_{1}(q,q-1) 
& =  
\begin{cases}
(\delta\Sq^{2}\Sq^{1}, \Sq^{2},0) & q\equiv 1 \bmod 4 \\
(\delta\Sq^{2}\Sq^{1}, \Sq^{2}+\rho\Sq^{1},\tau) & q\equiv 3 \bmod 4. 
\end{cases}  
\end{align*}
\end{theorem}

The following image depicts the first differential for $\KQ$ with degrees along the horizontal axis and weights along the vertical axis.
Each small dot is a suspension of $\MZZ/2$ while a large square is a suspension of $\MZZ$.
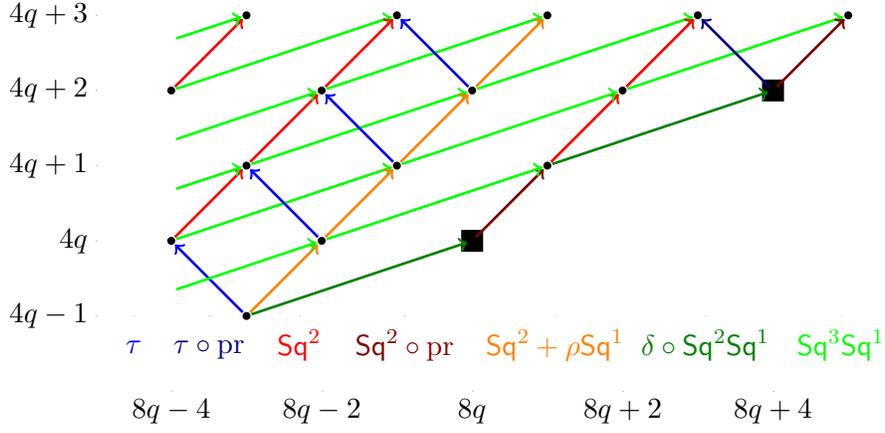
\begin{center}
  \pgfsetshortenend{2pt}
  \pgfsetshortenstart{2pt}
\begin{tikzpicture}[scale=1.0,line width=1pt]
{\draw[fill]     
(-3,-1) circle (1pt) 
(-5,-1) circle (0pt) node[left=-1pt] {$4q-1$}
;}
{\draw[fill]     
(0,0) circle (0pt) (-2,0) circle (1pt) (-4,0) circle (1pt) 
(-5,0) circle (0pt) node[left=-1pt] {$4q$}
;}
{\draw[fill]     
(1,1) circle (1pt) (-1,1) circle (1pt) (-3,1) circle (1pt) 
(-5,1) circle (0pt) node[left=-1pt] {$4q+1$}
;}
{\draw[fill]     
(0,2) circle (1pt) (-2,2) circle (1pt) (-4,2) circle (1pt) (2,2) circle (1pt) (4,2) circle (0pt) 
(-5,2) circle (0pt) node[left=-1pt] {$4q+2$}
;}
{\draw[fill]     
(5,3) circle (1pt) (3,3) circle (1pt) (1,3) circle (1pt) (-1,3) circle (1pt) (-3,3) circle (1pt) 
(-5,3) circle (0pt) node[left=-1pt] {$4q+3$}
;}
{\draw[fill]     
(0,-2) circle (0pt) node[below=-1pt] {$8q$} 
(-2,-2) circle (0pt) node[below=-1pt] {$8q-2$}
(-4,-2) circle (0pt) node[below=-1pt] {$8q-4$}
(2,-2) circle (0pt) node[below=-1pt] {$8q+2$}
(4,-2) circle (0pt) node[below=-1pt] {$8q+4$}
;}
{\draw[fill,blue]     
(-4.5,-1) circle (0pt) node[below=5pt] {$\tau$}
;}
{\draw[fill,blue!50!black]     
(-3.5,-1) circle (0pt) node[below=5pt] {$\tau\circ \pr$}
;}
{\draw[fill,red]     
(-2.3,-1) circle (0pt) node[below=1pt] {$\Sq^2$}
;}
{\draw[fill,red!50!black]     
(-0.9,-1) circle (0pt) node[below=1pt] {$\Sq^2\circ\pr$}
;}
{\draw[fill,orange]     
(1.1,-1) circle (0pt) node[below=1pt] {$\Sq^2+\rho\Sq^1$}
;}
{\draw[fill,green!50!black]     
(3.1,-1) circle (0pt) node[below=1pt] {$\delta \circ \Sq^2\Sq^1$}
;}
{\draw[fill,green]     
(4.9,-1) circle (0pt) node[below=1pt] {$\Sq^3\Sq^1$}
;}
\node at (0,0) [shape=rectangle,draw,fill]{};
\node at (4,2) [shape=rectangle,draw,fill]{};
{\draw[blue,->] 
(-3,-1) -- (-4,0)
;}
{\draw[orange,->] 
(-3,-1) -- (-2,0)
;}
{\draw[green!50!black,->] 
(-3,-1) -- (0,0)
;}
{\draw[green,->] 
(-4,-.67) -- (-2,0)
;}
{\draw[red!50!black,->] 
(0,0) -- (1,1)
;}
{\draw[red,->] 
(-4,0) -- (-3,1)
;}
{\draw[green,->] 
(-4,0) -- (-1,1)
;}
{\draw[green,->] 
(-4,.67) -- (-3,1)
;}
{\draw[blue,->] 
(-2,0) -- (-3,1)
;}
{\draw[orange,->] 
(-2,0) -- (-1,1)
;}
{\draw[green,->] 
(-2,0) -- (1,1)
;}
{\draw[red,->] 
(1,1) -- (2,2)
;}
{\draw[green!50!black,->] 
(1,1) -- (4,2)
;}
{\draw[red,->] 
(-3,1) -- (-2,2)
;}
{\draw[green,->] 
(-3,1) -- (0,2)
;}
{\draw[green,->] 
(-4,1.33) -- (-2,2)
;}
{\draw[blue,->] 
(-1,1) -- (-2,2)
;}
{\draw[orange,->] 
(-1,1) -- (0,2)
;}
{\draw[green,->] 
(-1,1) -- (2,2)
;}
{\draw[red,->] 
(-4,2) -- (-3,3)
;}
{\draw[green,->] 
(-4,2) -- (-1,3)
;}
{\draw[green,->] 
(-4,2.67) -- (-3,3)
;}
{\draw[red,->] 
(-2,2) -- (-1,3)
;}
{\draw[green,->] 
(-2,2) -- (1,3)
;}
{\draw[orange,->] 
(0,2) -- (1,3)
;}
{\draw[blue,->] 
(0,2) -- (-1,3)
;}
{\draw[green,->] 
(0,2) -- (3,3)
;}
{\draw[red,->] 
(2,2) -- (3,3)
;}
{\draw[green,->] 
(2,2) -- (5,3)
;}
{\draw[red!50!black,->] 
(4,2) -- (5,3)
;}
{\draw[blue!50!black,->] 
(4,2) -- (3,3)
;}
\end{tikzpicture}
\end{center}

\section{Milnor's conjecture on quadratic forms}
\label{section:milnorconjecture}

In this section we compute the slice spectral sequence of $\KT$ over any field $F$ of $\Char(F)\neq 2$.
Note that $\KT$ acquires a ring spectrum structure from hermitian $K$-theory $\KQ$ and the tower (\ref{equation:KQKTtower}).
Recall from Theorem \ref{thm:slices-witt-theory},
cf.~Example \ref{example:KT}, 
the identification
\[ 
\s_{0}(\KT)
\cong
\bigvee_{i\in\ZZ}\Sigma^{2i,0}\MZZ/2.
\]
By the periodicity $\Sigma^{1,1}  \KT\iso \KT$ induced by multiplication with the Hopf map, this determines all slices of $\KT$.
Recall that the first differential $\s_q(\KT) \rTo \Sigma^{1,0}  \s_{q+1}(\KT)$ is determined by the motivic Steenrod operations in (\ref{equation:d1KTdifferentials}), 
a formula proven in Theorem~\ref{thm:diff-witt-theory}. 
With the elements $\mathbf{\textcolor{blue}{\tau}}$,
$\mathbf{\textcolor{red}{\Sq^{2}}}$,
$\mathbf{\textcolor{orange}{\Sq^{2}+\rho\Sq^{1}}}$,
and $\mathbf{\textcolor{green}{\Sq^{3}\Sq^{1}}}$ corresponding to their given colors, 
the first differentials can be represented as follows. 
The group in bidegree $(p,q)$ is a direct sum of mod-$2$ motivic cohomology groups positioned on the vertical line above $p$
and inbetween the horizontal lines corresponding to weights $q$ and $q+1$.
The number of direct summands increases linearly with the weight:

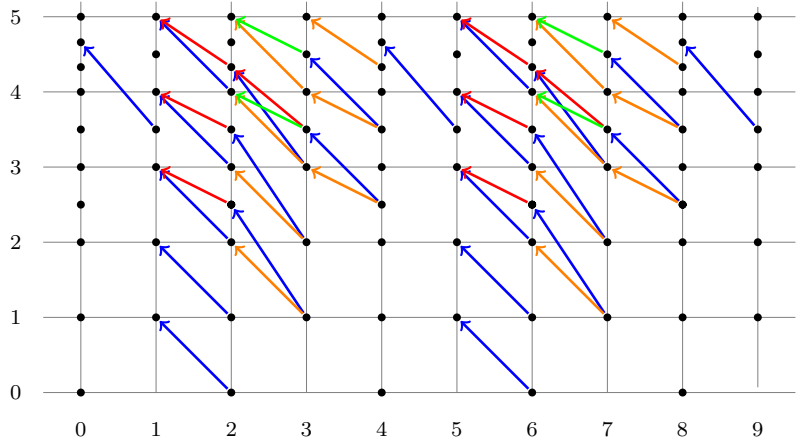
\begin{center}
  \pgfsetshortenend{2pt}
  \pgfsetshortenstart{2pt}
\begin{tikzpicture}[scale=1.0,font=\scriptsize,line width=1pt]
\draw[help lines] (-.5,0) grid (9.5,5.2);
\foreach \i in {0,...,5} {\node[label=left:$\i$] at (-.5,\i) {};}
\foreach \i in {0,...,9} {\node[label=below:$\i$] at (\i,-.1) {};}

\foreach \i in {0,...,5} {\draw[fill] (0,\i) circle (1pt);}
{\draw[fill]     
(0,2.5) circle (1pt) 
(0,3.5) circle (1pt) 
(0,4.33) circle (1pt) 
(0,4.66) circle (1pt) 
;}

{\draw[fill]      
(1,1) circle (1pt) 
(1,2) circle (1pt) 
(1,3) circle (1pt) 
(1,3.5) circle (1pt) 
(1,4) circle (1pt) 
(1,4.5) circle (1pt) 
(1,5) circle (1pt) 
;}

\foreach \i in {0,...,5} {\draw[fill] (2,\i) circle (1pt);}
{\draw[fill]       (2,2.5) circle (1pt) 
(2,2.5) circle (1pt) 
(2,3.5) circle (1pt) 
(2,4.33) circle (1pt) 
(2,4.66) circle (1pt) 
;}

{\draw[fill]      
(3,1) circle (1pt) 
(3,2) circle (1pt) 
(3,3) circle (1pt) 
(3,3.5) circle (1pt) 
(3,4) circle (1pt) 
(3,4.5) circle (1pt) 
(3,5) circle (1pt) 
;}

\foreach \i in {0,...,5} {\draw[fill] (4,\i) circle (1pt);}
{\draw[fill]       
(4,2.5) circle (1pt) 
(4,3.5) circle (1pt) 
(4,4.33) circle (1pt) 
(4,4.66) circle (1pt) 
;}

{\draw[fill]      
(5,1) circle (1pt) 
(5,2) circle (1pt) 
(5,3) circle (1pt) 
(5,3.5) circle (1pt) 
(5,4) circle (1pt) 
(5,4.5) circle (1pt) 
(5,5) circle (1pt) 
;}

\foreach \i in {0,...,5} {\draw[fill] (6,\i) circle (1pt);}
{\draw[fill]       (6,2.5) circle (1pt) 
(6,2.5) circle (1pt) 
(6,3.5) circle (1pt) 
(6,4.33) circle (1pt) 
(6,4.66) circle (1pt) 
;}

{\draw[fill]      
(7,1) circle (1pt) 
(7,2) circle (1pt) 
(7,3) circle (1pt) 
(7,3.5) circle (1pt) 
(7,4) circle (1pt) 
(7,4.5) circle (1pt) 
(7,5) circle (1pt) 
;}

\foreach \i in {0,...,5} {\draw[fill]  (8,\i) circle (1pt) ;}
{\draw[fill]       (8,2.5) circle (1pt) 
(8,2.5) circle (1pt) 
(8,3.5) circle (1pt) 
(8,4.33) circle (1pt)
(8,4.66) circle (1pt)
;}

{\draw[fill]      
(9,1) circle (1pt) 
(9,2) circle (1pt) 
(9,3) circle (1pt) 
(9,3.5) circle (1pt)
(9,4) circle (1pt) 
(9,4.5) circle (1pt) 
(9,5) circle (1pt) 
;}

{\draw[blue,->]
(1,3.5) -- (0,4.66);
\draw[blue,->]
(2,0) -- (1,1);
\draw[blue,->]
(2,1) -- (1,2);
\draw[blue,->]
(2,2) -- (1,3);
\draw[blue,->]
(2,3) -- (1,4);
\draw[blue,->]
(2,4) -- (1,5);
\draw[blue,->]
(3,1) -- (2,2.5);
\draw[blue,->]
(3,2) -- (2,3.5);
\draw[blue,->]
(3,3) -- (2,4.33);
\draw[blue,->]
(4,2.5) -- (3,3.5);
\draw[blue,->]
(4,3.5) -- (3,4.5);
\draw[blue,->]
(5,3.5) -- (4,4.66);
\draw[blue,->]
(6,0) -- (5,1);
\draw[blue,->]
(6,1) -- (5,2);
\draw[blue,->]
(6,2) -- (5,3);
\draw[blue,->]
(6,3) -- (5,4);
\draw[blue,->]
(6,4) -- (5,5);
\draw[blue,->]
(7,1) -- (6,2.5);
\draw[blue,->]
(7,2) -- (6,3.5);
\draw[blue,->]
(7,3) -- (6,4.33);
\draw[blue,->]
(8,2.5) -- (7,3.5);
\draw[blue,->]
(8,3.5) -- (7,4.5);
\draw[blue,->]
(9,3.5) -- (8,4.66);
}

{\draw[->,red]
(2,2.5) -- (1,3);
\draw[red,->]
(2,3.5) -- (1,4);
\draw[red,->]
(2,4.33) -- (1,5);
\draw[red,->]
(3,3.5) -- (2,4.33);
\draw[red,->]
(6,2.5) -- (5,3);
\draw[red,->]
(6,3.5) -- (5,4);
\draw[red,->]
(6,4.33) -- (5,5);
\draw[red,->]
(7,3.5) -- (6,4.33);
;}
{\draw[->,orange] 
(3,1) -- (2,2);
\draw[orange,->]
(3,2) -- (2,3);
\draw[orange,->]
(3,3) -- (2,4);
\draw[orange,->]
(3,4) -- (2,5);
\draw[orange,->]
(4,2.5) -- (3,3);
\draw[orange,->]
(4,3.5) -- (3,4);
\draw[orange,->]
(4,4.33) -- (3,5);
\draw[orange,->]
(7,1) -- (6,2);
\draw[orange,->]
(7,2) -- (6,3);
\draw[orange,->]
(7,3) -- (6,4);
\draw[orange,->]
(7,4) -- (6,5);
\draw[orange,->]
(8,2.5) -- (7,3);
\draw[orange,->]
(8,3.5) -- (7,4);
\draw[orange,->]
(8,4.33) -- (7,5);
;}
{\draw[->,green]
(3,3.5) -- (2,4);
\draw[green,->]
(3,4.5) -- (2,5);
\draw[green,->]
(7,3.5) -- (6,4);
\draw[green,->]
(7,4.5) -- (6,5)
;}
\end{tikzpicture}
\end{center}

The next statement is an immediate consequence of Voevodsky's proof of Milnor's conjecture on Galois cohomology \cite{Voevodsky:Z/2}.
Recall the classes $\tau\in h^{0,1}$ and $\rho\in h^{1,1}$ are represented by $-1\in F$.  
\begin{lemma}
\label{lem:tau-iso}
For $0\leq p\leq q$ cup-product with $\tau$ yields an isomorphism
$
\tau\colon
h^{p,q}
\rTo^{\iso}  
h^{p,q+1}. 
$
\end{lemma}

\begin{corollary}
\label{cor:sq-weight-1}
If $a\in h^{p,q}$ where $0\leq p\leq q$, 
write $a=\tau^{q-p}c$ where $c\in h^{p,p}$ and let $n=q-p$.
The Steenrod squares of weight $\leq 1$ act on the mod-$2$ motivic cohomology ring $h^{\ast,\ast}$ by 
\begin{align*}
\Sq^{1}(\tau^nc) 
= &  
\begin{cases} 
\rho\tau^{n-1}c    & n\equiv 1\bmod 2 \\
0                  & n\equiv 0\bmod 2
\end{cases} 
& 
\Sq^{2}(\tau^nc) 
= & 
\begin{cases} 
\rho^{2}\tau^{n-1}c    & n\equiv 2,3\bmod 4\\
0                     & n\equiv 0,1\bmod 4
\end{cases} 
\\
\Sq^{2}\Sq^{1}(\tau^nc) 
& = 
\begin{cases} 
\rho^{3}\tau^{n-2}c    & n\equiv 3\bmod 4\\
0                     & n\equiv 0,1,2\bmod 4 
\end{cases}
&
\Sq^{3}\Sq^{1}(\tau^nc) 
= &   
\begin{cases} 
\rho^{4}\tau^{n-3}c    & n\equiv 3\bmod 4\\
0                     & n\equiv 0,1,2\bmod 4.
\end{cases} 
\end{align*}
\end{corollary}
\begin{proof}
This follows from Lemma \ref{lem:tau-iso}, 
the computation of $\Sq^i(\tau^n)$ for $i\in\{1,2\}$ and the Cartan formula \cite[Proposition 9.6]{Voevodsky:Steenrod}. 
\end{proof}

We consider an element $\phi\in h^{m,n}$ as a stable motivic cohomology operation of bidegree $(m,n)$ with the same name via multiplication with $\phi$ on the left. 
Only elements of bidegrees $(0,1)$, $(1,1)$, and $(1,2)$ will be relevant here.
The Adem relations in weight less than or equal to 2 are given by
\[
\Sq^1\Sq^1=0,
\;
\Sq^{1}\tau=\tau\Sq^{1}+\rho,
\;
\Sq^{1}\rho=\rho\Sq^{1},
\;
\Sq^{1}\Sq^{2}=\Sq^{3},
\;
\Sq^1\Sq^3=0,
\;
\Sq^{2}\tau=\tau\Sq^{2}+\tau\rho\Sq^{1}, 
\]
\[
\Sq^{2}\rho=\rho\Sq^{2}, 
\;
\Sq^{2}\Sq^{2}=\tau\Sq^{3}\Sq^{1},
\;
\Sq^{2}\Sq^{3}=\Sq^{5}+\Sq^{4}\Sq^{1},
\;
\Sq^{3}\Sq^{2}=\rho\Sq^{3}\Sq^{1},
\;
\Sq^{3}\Sq^{3}=\Sq^{5}\Sq^{1}.
\]

This concludes the prerequisites for our proof of the following result.
\begin{theorem}
\label{thm:e2-witt}
The 0th slice spectral sequence for $\KT$ collapses at its $E^{2}$-page, 
and 
\begin{equation*}
E^{\infty}_{p,q}(\KT)
\cong
\begin{cases}
h^{q,q} 
&
p\equiv 0\bmod 4 \\
0
&
\mathrm{otherwise.} \\
\end{cases}
\end{equation*}
\end{theorem}
\begin{proof}
The $E^{1}$-page takes the form
\[ 
E^{1}_{p,q} 
= 
\pi_{p,0}\s_{q}(\KT) 
= 
\bigdirectsum_{i\in \ZZ} h^{2i+(q-p),q} 
= 
\begin{cases}
\bigdirectsum_{j=0}^{\lfloor\frac{q}{2}\rfloor}h^{2j,q} & q\equiv p\bmod 2 \\
\bigdirectsum_{j=0}^{\lfloor\frac{q-1}{2}\rfloor}h^{2j+1,q} & q\not\equiv p \bmod 2.
\end{cases}
\]
The group $h^{2i+(q-p),q}$ in the sum 
$\bigdirectsum_{i\in \ZZ} h^{2i+(q-p),q}$ 
arises from the $2i$th summand of $\s_{q}(\KT)$.
The vanishing of $h^{p,q}$ for $p<0$ and $p>q$ shows the sum is finite.
Hence for every element $a\in E^1_{p,q}$ there exists a unique collection
of elements
$\{a_j\in h^{j,q}\}$ such that
\[ 
a= 
\begin{cases} 
(a_{q},a_{q-2},\dotsc,a_0)\in E^{1}_{p,q} & q\equiv 0 \equiv p \bmod 2\\ 
(a_{q-1},a_{q-3},\dotsc,a_0)\in E^{1}_{p,q} & q\equiv 1 \equiv p \bmod 2 \\
(a_{q},a_{q-2},\dotsc,a_1)\in E^{1}_{p,q} & q\equiv 1 \not\equiv p \bmod 2 \\
(a_{q-1},a_{q-3},\dotsc,a_1)\in E^{1}_{p,q} & q\equiv 0 \not\equiv p \bmod 2. 
\end{cases}
\]
 By inspection of (\ref{equation:d1KTdifferentials}) the components of $d_{1}^{\KT}$ are given by
\begin{equation}
\label{eq:d1kt} 
d_{1}^{\KT}(a)_{j} =
\begin{cases} 
\Sq^{3}\Sq^{1}a_{j-4}+\Sq^{2}a_{j-2}+\rho\Sq^{1}a_{j-2} & j\equiv q-p\bmod 4 \\ 
\Sq^{3}\Sq^{1}a_{j-4}+\Sq^{2}a_{j-2}+\tau a_{j} & j\equiv q-p+2\bmod 4.
\end{cases} 
\end{equation}
Here $j$ is the dimension index of the target group. 
The formula can be read off from the parity of the dimension index of the source group. 
Note that $\rho \Sq^1$, $\Sq^2$ and $\tau$ shift the dimension by $2$, $2$, and $0$, 
respectively. 
We compute the $E^{2}$-page by repeatedly using the Steenrod square computations and Adem relations given in the beginning of Section~\ref{section:milnorconjecture}.

\begin{description}
\item[\underline{$p\equiv 0\bmod 4$}]
For $a=(a_q,a_{q-2},\dotsc)$, 
Corollary~\ref{cor:sq-weight-1} implies 
\[
d_{1}^{\KT}(a)_j
= 
\begin{cases}
\Sq^{2}a_{j-2} & j\equiv q \bmod 4 \\
\tau a_{j}    & j\equiv q+2 \bmod 4. 
\end{cases}
\]
If $d_{1}^{\KT}(a)=0$ the injectivity part of Lemma~\ref{lem:tau-iso} implies $0=a_{q-2}=a_{q-6}=\dotsm$,
so that 
\[ 
\mathrm{ker}(d_{1}^{\KT}) 
= 
h^{q,q} \directsum h^{q-4,q}\directsum \dotsm.
\]
We may assume $q>0$.
For $b\in E^1_{p+1,q-1}$ the entering differential is given by
\[ 
d_{1}^{\KT}(b)_{j} 
= 
\begin{cases} 
\Sq^{3}\Sq^{1}b_{j-4}+\Sq^{2}b_{j-2}+\tau b_{j} & j\equiv q-p\bmod 4 \\
\Sq^{3}\Sq^{1}b_{j-4}+\Sq^{2}b_{j-2}+\rho\Sq^{1}b_{j-2}   & j\equiv q-p+2\bmod 4. 
\end{cases} 
\]
Corollary~\ref{cor:sq-weight-1} simplifies this formula to
\[ 
d_{1}^{\KT}(b)_{j} 
= 
\begin{cases} 
\Sq^{3}\Sq^{1}b_{j-4}+\tau b_{j} & j\equiv q-p\bmod 4 \\
0   & j\equiv q-p+2\bmod 4. 
\end{cases} 
\]
For example, if $p\equiv 0 \bmod 4$, then $j\equiv q-p \bmod 4$ implies
$j\equiv q \bmod 4$. Thus for $b_{j-2}\in h^{j-2,q-1}$ one has
$q-1-(j-2) \equiv 1\bmod 4$, whence $\Sq^2 b_{j-2} = 0$.
It follows that $E^2_{p,q}$ is the homology of the complex
\[ 
h^{q-4,q-1}\directsum h^{q-8,q-1}\directsum \dotsm
\rTo^{\alpha}  
h^{q,q}\directsum h^{q-4,q} \directsum \dotsm \rTo 0, 
\]
where $\alpha(b_{q-4},b_{q-8},\dotsc,b_m)=(\Sq^{3}\Sq^{1}b_{q-4},\Sq^{3}\Sq^{1}b_{q-8}+\tau b_{q-4},\dotsc,\tau b_m)$.
Here $m\equiv q\bmod 4$ and $0\leq m\leq 3$.
Lemma~\ref{lem:tau-iso} implies $\alpha$ is split injective by mapping $(a_{q},a_{q-4},\dotsc,a_m)$ to
\[ 
\bigl(\tau^{-1}a_{q-4}+\phi(a_{q-4}+\tau\phi(a_{q-8}+\dotsm+\tau\phi(a_{m}))),\dotsc,\tau^{-1}a_{m+4}+\phi(a_{m}),\tau^{-1}a_{m}\bigr), 
\]
where $\phi$ is the composite map
\[ 
h^{p,q}
\rTo^{\tau^{-1}} 
h^{p,q-1} 
\rTo^{\Sq^{3}\Sq^{1}}
h^{p+4,q}
\rTo^{\tau^{-1}} 
h^{p+4,q-1}. 
\] 
It follows that $E^{2}_{p,q}\iso h^{q,q}$ for all $p\equiv 0\bmod 4$.
        
\item[\underline{$p\equiv 1\bmod 4$}]
For $a=(a_{q-1},a_{q-3},\dotsc)$,
Corollary~\ref{cor:sq-weight-1} implies
\[ 
d_{1}^{\KT}(a)_j 
= 
\begin{cases} 
0                              & j\equiv q-1\bmod 4 \\
\Sq^{3}\Sq^{1}a_{j-4}+\tau a_{j} & j\equiv q+1\bmod 4. 
\end{cases} 
\]
If $d_{1}^{\KT}(a)=0$ then $0=a_{q-3}=a_{q-7} = \dotsm$ by applying inductively the injectivity statement in Lemma~\ref{lem:tau-iso}. 
Thus we have 
\[ 
\mathrm{ker}(d_{1}^{\KT}) 
= 
h^{q-1,q} \directsum h^{q-5,q}\directsum \dotsm.
\]
For $b\in E^1_{p+1,q-1}$ the entering differential is given by
\[ 
d_{1}^{\KT}(b)_{j} 
= 
\begin{cases} 
\Sq^{3}\Sq^{1}b_{j-4}+\Sq^{2}b_{j-2}+\tau b_{j} & j\equiv q-p\bmod 4 \\
\Sq^{3}\Sq^{1}b_{j-4}+\Sq^{2}b_{j-2}+\rho\Sq^{1}b_{j-2}   & j\equiv q-p+2\bmod 4. 
\end{cases} 
\]
Corollary~\ref{cor:sq-weight-1} simplifies this formula to
\[ 
d_{1}^{\KT}(b)_{j} 
= 
\begin{cases} 
\Sq^{2}b_{j-2}+\tau b_{j} & j\equiv q-p\bmod 4 \\
0                        & j\equiv q-p+2\bmod 4. 
\end{cases} 
\]
Thus $E^2_{p,q}$ is the homology of the complex
\[ 
h^{q-1,q-1}\directsum h^{q-3,q-1} \directsum \dotsm
\rTo^{d_1^{\KT}}  
h^{q-1,q}\directsum h^{q-5,q} \directsum \dotsm 
\rTo 
0. 
\]
Since the restriction of $d_{1}^{\KT}$ to $h^{q-1,q-1}\directsum h^{q-5,q-1}\directsum \dotsm$ is surjective by Lemma \ref{lem:tau-iso},
$E_{p,q}^{2}=0$.
        
\item[\underline{$p\equiv 2\bmod 4$}]
For $a=(a_{q},a_{q-2},\dotsc)$, 
Corollary~\ref{cor:sq-weight-1} implies
\[ 
d_{1}^{\KT}(a)_{j} 
= 
\begin{cases} 
0                        & j\equiv q-2\bmod 4 \\
\Sq^{2}a_{j-2}+\tau a_{j} & j\equiv q\bmod 4. 
\end{cases} 
\]
Hence the subgroup $\mathrm{ker}(d_{1}^{\KT})$
can be identified with
\[
\{ 
(a_q,a_{q-2},\dotsc)\in h^{q,q}\directsum h^{q-2,q}\directsum \dotsm 
\colon 
\tau a_{j}
=
\Sq^2 a_{j-2} \mathrm{\ for\ all\ } j\equiv q \bmod 4, 0\leq j \leq q\}.
\]
For $b\in E^1_{p+1,q-1}$ the entering differential is given by
\[ 
d_{1}^{\KT}(b)_{j} 
= 
\begin{cases} 
\Sq^{3}\Sq^{1}b_{j-4}+\Sq^{2}b_{j-2}+\rho\Sq^{1}b_{j-2}   & j\equiv q \bmod 4 \\
\Sq^{3}\Sq^{1}b_{j-4}+\Sq^{2}b_{j-2}+\tau b_{j} & j\equiv q-2\bmod 4. 
\end{cases} 
\]
Corollary~\ref{cor:sq-weight-1} simplifies this formula to
\[ 
d_{1}^{\KT}(b)_{j} 
= 
\begin{cases} 
\Sq^{3}\Sq^{1}b_{j-4}+\rho\Sq^{1}b_{j-2}  & j\equiv q \bmod 4  \\
\Sq^{2}b_{j-2}+\tau b_{j}                & j\equiv q-2\bmod 4. 
\end{cases} 
\]
If $a\in \mathrm{ker}(d_{1}^{\KT})$, 
Lemma \ref{lem:tau-iso} shows there exists elements $b_{j,q-1}\in h^{j,q-1}$ for all $0\leq j<q$ where $j\equiv q-2 \bmod 4$ and
$\tau b_{j,q-1} =a_j$. 
For these indices $j$, 
the Adem relation $\Sq^2 \tau = \tau \Sq^2 + \tau \rho \Sq^1$ 
and $d_{1}^{\KT}(a)=0$ imply
\[ 
\tau a_{j+2} 
= 
\Sq^{2} a_{j} 
= 
\Sq^{2}\tau b_{j} 
=
\tau \Sq^{2}b_{j}+\tau \rho \Sq^{1} b_{j} =\tau \rho \Sq^{1}b_{j}.
\]
Thus $\rho \Sq^{1}b_{j}=a_{j+2}$ by Lemma~\ref{lem:tau-iso}. 
It follows that 
\[ 
d_{1}^{\KT}(b_{q-2},0,b_{q-6},\dotsc) 
= 
(\rho\Sq^{1}b_{q-2},\tau b_{q-2},\rho\Sq^{1} b_{q-6},\tau b_{q-6}) 
= (a_{q},a_{q-2},a_{q-4},a_{q-6},\dotsc).
\]
This shows that $E_{p,q}^{2}$ is trivial.

\item[\underline{$p\equiv 3\bmod 4$}]
For $a=(a_{q-1},a_{q-3},\dotsc)$ the exiting differential simplifies to 
\[ 
d_{1}^{\KT}(a)_{j} 
= 
\begin{cases} 
\Sq^{3}\Sq^{1}a_{j-4}+\rho\Sq^{1}a_{j-2}   & j\equiv q-3\bmod 4 \\
\Sq^{2}a_{j-2}+\tau a_{j}                 & j\equiv q-1\bmod 4 
\end{cases} 
\]
by Corollary~\ref{cor:sq-weight-1}. 
Applying $\Sq^2$ to $\Sq^2 a_{j-2}+\tau a_{j}$ yields
\[ 
\Sq^2 (\Sq^2 a_{j-2}+\tau a_{j})
=
\tau \Sq^3\Sq^1 a_{j-2}+\tau \Sq^2 a_{j}+\tau\rho\Sq^{1}a_{j}. 
\]
Thus $\mathrm{ker}(d_{1}^{\KT})$ is comprised of tuples $(a_{q-1},a_{q-3},\dotsc)\in h^{q-1,q}\directsum h^{q-3,q}\directsum \dotsm$
for which $\tau a_{j}=\Sq^2 a_{j-2}$ whenever $j\equiv q -1 \bmod 4$, $0\leq j<q$.
For $b\in E^1_{p+1,q-1}$ the entering differential simplifies to 
\[ 
d_{1}^{\KT}(b)_{j} 
= 
\begin{cases} 
\tau b_{j}      & j\equiv q-3\bmod 4 \\
\Sq^{2}b_{j-2}   & j\equiv q-1\bmod 4. 
\end{cases} 
\]
If $a\in \mathrm{ker}(d_{1}^{\KT})$, 
Lemma~\ref{lem:tau-iso} shows there exist elements $b_{j,q-1}\in h^{j,q-1}$ for all $0\leq j<q$ where $j\equiv q-3 \bmod 4$ and 
$\tau b_{j,q-1}=a_j$. 
For these $j$, 
$d_{1}^{\KT}(a)=0$ implies
\[ 
\tau a_{j+2} 
= 
\Sq^{2} a_{j} 
= 
\Sq^{2}\tau b_{j} 
=\tau \Sq^{2}b_{j}+\tau \rho \Sq^{1} b_{j} =\tau \Sq^{2}b_{j}.
\]
Hence $\Sq^{2}b_{j} = a_{j+2}$ by Lemma~\ref{lem:tau-iso}. 
It follows that 
\[ 
d_{1}^{\KT}(b_{q-3},0,b_{q-7},\dotsc) 
= 
(\Sq^{2}b_{q-3},\tau b_{q-3},\Sq^{2} b_{q-7},\tau b_{q-7}) 
= 
(a_{q-1},a_{q-3},a_{q-5},a_{q-7},\dotsc).
\]
This shows that $E_{p,q}^{2}$ is trivial.
\end{description}
An inspection of the $E^{2}$-page shows that the 0th slice spectral sequence for $\KT$ collapses with $E^{2}=E^{\infty}$-page: 
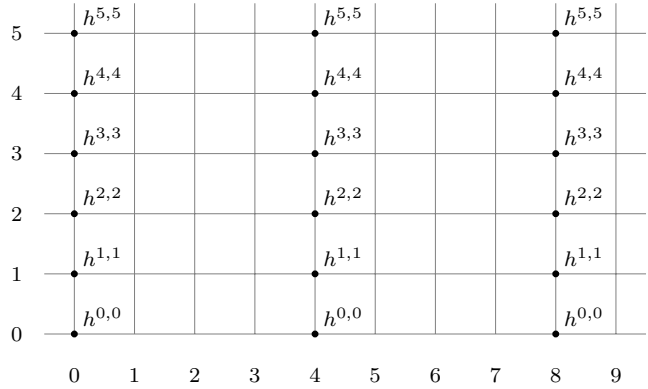
\begin{center}
\begin{tikzpicture}[scale=0.8,font=\scriptsize,line width=1pt]
\draw[help lines] (-.5,0) grid (9.5,5.5);
\foreach \i in {0,...,5} {\node[label=left:$\i$] at (-.5,\i) {};}
\foreach \i in {0,...,9} {\node[label=below:$\i$] at (\i,-.2) {};}
\foreach \i in {0,...,5} 
     {\draw[fill] (0,\i) circle (1pt) node[above right=-1pt] {$h^{\i,\i}$};}
\foreach \i in {0,...,5} 
     {\draw[fill] (4,\i) circle (1pt) node[above right=-1pt] {$h^{\i,\i}$};}
\foreach \i in {0,...,5} 
     {\draw[fill] (8,\i) circle (1pt) node[above right=-1pt] {$h^{\i,\i}$};}
\end{tikzpicture}
\end{center}
\end{proof}

Let $I(F)$ denote the fundamental ideal of even dimensional quadratic forms in the Witt ring $W(F)$. 
To conclude our proof of Milnor's conjecture on quadratic forms, 
it remains to identify the slice filtration on the Witt ring with the filtration given by the fundamental ideal.
\begin{lemma}\label{lem:f1-ideal}
The slice filtration of $\KT$ induces a commutative diagram:
\begin{diagram}
\pi_{0,0}\f_{1}(\KT) & \rInto & \pi_{0,0}\f_{0}(\KT) \\
\dTo^{\iso} & & \dTo_\iso \\
I(F) & \rInto & W(F)
\end{diagram}
\end{lemma}
\begin{proof}
By \cite[Corollary 5.18]{GRSO} the canonical map $ \f_0(\KT)\rTo \s_0(\KT)$ induces a ring homomorphism 
\[ 
W(F) 
\iso 
\pi_{0,0}\f_{0}(\KT) 
\rTo 
\pi_{0,0}\s_{0}(\KT) 
\iso 
h^{0,0}
\iso
\FF_{2}. 
\]
It is natural with respect to separable field extensions according to Corollary \ref{cor:slice-ess-smooth-base-change}.
Comparing with an algebraic closure of $F$ shows this ring homomorphism is induced by sending a quadratic form to its rank. 
Since the group $\pi_{1,0}\s_{0}(\KT)=h^{1,0}$ is trivial,  
the map 
\[ 
\pi_{0,0}\f_{1}(\KT) 
\rTo 
\pi_{0,0}\f_{0}(\KT)
\]
is injective, 
which proves the result.
\end{proof}

\begin{corollary}\label{cor:f1-ideal}
The identification $\pi_{0,0}\KT\iso W(F)$ induces an inclusion $I(F)^{q}\subseteq\f_{q}\pi_{0,0}\KT$.
\end{corollary}
\begin{proof}
Since $\KT$ is a ring spectrum the claim follows from the multiplicative structure of the slice filtration \cite[Theorem 5.15]{GRSO}, 
\cite[Theorem 3.6.9]{Pelaez}, 
and Lemma \ref{lem:f1-ideal}.
\end{proof}

Any rational point $u\in \AA^{1}\minus \{0\}(F)$ defines a map of motivic spectra $[u]\colon \One\rTo S^{1,1}$. 
We are interested in the effect of the map $[u]$ on motivic cohomology and $\KT$.

\begin{lemma}
\label{lem:unit-H}
The map
\begin{diagram}
H^{0,0}
= 
[\One,\MZZ]
= 
[\One,\One\smash\MZZ] 
&
\rTo^{([u]\smash\MZZ)_{\ast}}
& 
[\One,\Sigma^{1,1}\MZZ]
= 
H^{1,1}
\end{diagram}
sends $1$ to $u\in \AA^{1}\minus\{0\}(F)= F^{\times}\iso H^{1,1}$.
\end{lemma}
\begin{proof}
Let $\ZZ(1)$ denote the Tate object in the derived category of motives $\DM_F$ over $F$.
The assertion follows from the canonically induced diagram:
\begin{diagram}
\Hom_{\Sm_F}(F,\AA^{1}\minus \{0\}) & 
\rTo & 
[\One,\Sigma^{1,1} \MZZ] \\
\dTo & & 
\dTo_\iso \\
\Hom_{\DM_F}(\ZZ,\mathcal{O}^\units) & 
\lTo^\iso & 
\Hom_{\DM_F}(\ZZ,\ZZ(1)[1])
\end{diagram}
The lower horizontal isomorphism follows from \cite[Theorem 4.1]{MVW}.
\end{proof}

\begin{corollary}\label{cor:unit-h}
Suppose $u_{1},\dotsc,u_{q}\in \AA^{1}\minus \{0\}(F)$ are rational points. 
The map $H^{0,0} \rTo H^{q,q}$ induced by the smash product $[u_{1}]\smash \dotsm \smash [u_{q}]$ sends $1$ to  
$\{u_{1},\dotsc,u_{q}\}\in K^{M}_q\iso H^{q,q}$,
and likewise for $h^{0,0}\rTo h^{q,q}$.
\end{corollary}

\begin{lemma}\label{lem:unit-kt}
The composition $\One\rTo^{[u]} S^{1,1}\rTo^\eta \One$ induces multiplication by $\langle u \rangle -1$ on $\pi_{0,0}\One$, 
where $\langle u\rangle$ is the class of the rank one quadratic form in the Grothendieck-Witt ring defined by $u$.
\end{lemma}
\begin{proof}
This follows from \cite[Corollary 1.24]{Morel:2052}.
\end{proof}

The unit map for $\KT$ induces the canonical map from the Grothendieck-Witt ring to the Witt ring $\pi_{0,0}\One\rTo \pi_{0,0}\KT$.
By definition, 
multiplication by $\eta$ is an isomorphism on $\pi_{\ast,\ast}\KT$.
Thus the ``multiplication by $[u]$'' map on $\pi_{0,0}\KT$ is determined by its effect on the Grothendieck-Witt ring $\pi_{0,0}\One$. 
Recall that $\f_{q}\pi_{p,n}\E$ denotes the image of the canonical map $\pi_{p,n}\f_{q}(E)\rTo\pi_{p,n}\E$.

\begin{lemma}
\label{lem:ses-conv}
There is a canonically induced short exact sequence
\[ 
0
\rTo 
\f_{q+1}\pi_{0,0}\KT 
\rTo^{j_q} 
\f_{q}\pi_{0,0}\KT 
\rTo 
h^{q,q} 
\rTo 
0.\]
\end{lemma}
\begin{proof} 
Theorem~\ref{thm:e2-witt} shows the exact sequence in the proof of \cite[Lemma 7.2]{Voevodsky:open} takes the form
\[ 
0 \rTo 
\f_{q}\pi_{0,0}\KT/\f_{q+1}\pi_{0,0}\KT \rTo^\alpha h^{q,q} 
\rTo 
\bigcap_{i\geq 1} \f_{q+i} \pi_{-1,0} \f_q(\KT) 
\rTo 
\bigcap_{i\geq 0} \f_{q+i} \pi_{-1,0}\KT \rTo 0.
\]
Moreover, 
$\pi_{-1,0}\KT$ is the trivial group.
Corollary~\ref{cor:f1-ideal} furnishes a map
\[ 
I(F)^q/I(F)^{q+1} 
\rTo^\beta 
\f_{q}\pi_{0,0}\KT/\f_{q+1}\pi_{0,0}\KT. 
\]
Combined with the canonical surjective map
$
k^{M}_q 
\rTo^\gamma  
I(F)^q/I(F)^{q+1} 
$
from the $q$-th mod-$2$ Milnor K-theory group defined in \cite{Milnor}, 
we obtain the composite map 
\begin{equation}
\label{equation:abg}
\alpha\circ\beta\circ\gamma 
\colon 
k^{M}_q
\rTo 
h^{q,q}.
\end{equation}
Lemmas \ref{lem:unit-H} and \ref{lem:unit-kt} show that (\ref{equation:abg}) coincides with Suslin's isomorphism between Milnor $K$-theory 
and the diagonal of motivic cohomology \cite[Lecture 5]{MVW}. 
In particular, 
$\alpha$ is surjective.
\end{proof}

The above shows that $\gamma$ is injective, 
which gives an alternate proof of the main result in \cite{Orlov-Vishik-Voevodsky}. 
\begin{theorem}
\label{thm:kmilnorwitt}
The canonical map $k^{M}_q \rTo I(F)^q/I(F)^{q+1}$ is an isomorphism for $q\geq 0$.
\end{theorem}

\begin{corollary}
\label{cor:sliceisideal}
The identification $\pi_{0,0}\KT\iso W(F)$ induces an equality $I(F)^{q} = \f_{q}\pi_{0,0}\KT$ for $q\geq 0$. 
\end{corollary}
\begin{proof}
By the definition of $\beta$ there is a commutative diagram: 
\begin{diagram}
0 &
\rTo &
I(F)^{q+1} &
\rTo &
I(F)^{q} &
\rTo &
I(F)^q/I(F)^{q+1} &
\rTo &
0 \\
& & \dTo & & \dTo & & \dTo_\beta & &  \\
0 &
\rTo &
\f_{q+1}\pi_{0,0}\KT &
\rTo^{j_q} &
\f_{q}\pi_{0,0}\KT &
\rTo &
h^{q,q} &
\rTo &
0 
\end{diagram}
Note that $\beta$ is an isomorphism by Lemma \ref{lem:ses-conv}, Theorem \ref{thm:kmilnorwitt}, and the isomorphism (\ref{equation:abg}).
Thus the result follows by induction using the identification $I(F)=\f_1\pi_{0,0}\KT=\pi_{0,0}\f_1\KT$ in Lemma~\ref{lem:f1-ideal}. 
\end{proof}

This finishes our proof of Milnor's conjecture on quadratic forms.

\begin{theorem}
\label{thm:convergence-witt}
The image of $\pi_{4p+q,q}\f_{n}(\KT)$ in $\pi_{4p+q,q}(\KT)\iso W(F)$ coincides with $I^{n-q}(F)$, 
where $I(F)\subseteq W(F)$ is the fundamental ideal. 
Thus the slice spectral sequence for $\KT$ converges to the filtration of the Witt ring given by the fundamental ideal.
\end{theorem}
\begin{proof}
This follows from Corollary~\ref{cor:sliceisideal} and the main result in \cite{Arason-Pfister}, 
which shows the filtration of $W(F)$ by $I(F)$ is Hausdorff. 
For completeness we analyze the map $\pi_{p,0} \f_{q+1}(\KT)\rTo \pi_{p,0} \f_{q}(\KT)$ for columns $p\equiv 1,2,3 \bmod 4$ in Lemma~\ref{lem:filt-zero-123} below.
\end{proof}

\begin{lemma}\label{lem:filt-zero-123}
Let $q\geq 0$. 
The canonical map $\f_{q+1}(\KT) \rTo \f_{q}(\KT)$ induces the trivial map 
\[ 
\pi_{p,0} \f_{q+1}(\KT) 
\rTo 
\pi_{p,0} \f_{q}(\KT) 
\]
for $p\equiv 1,2,3 \bmod 4$.
\end{lemma}
\begin{proof}
The statement is clear for $q=0$. 
Suppose that $p\equiv 1,2 \bmod 4$ and $x\in \pi_{p,0}\f_{q+1}(\KT)$.  
Its image $y\in \pi_{p,0}\f_{q}(\KT)$ lies in the kernel of the map $\pi_{p,0}\f_{q}(\KT)\rTo \pi_{p,0}\s_q(\KT)$. 
By induction, 
the map $\pi_{p+1,0} \s_{q-1}(\KT)\rTo \pi_{p,0}\f_{q}(\KT)$ is surjective. 
Hence there is an element $z\in \pi_{p+1,0} \s_{q-1}(\KT)$ mapping to $y$. 
Thus $z$ lies in the kernel of $d_{1}({p+1,q-1})$.
Theorem~\ref{thm:e2-witt} shows the kernel of $d_{1}({p+1,q-1})$ coincides with the image of $d_{1}({p+2,q})$. 
In particular, 
there is an element in $\pi_{p+2,q} \s_q(\KT)$ whose image is $z$, 
showing that $y=0$.

Suppose now that $p\equiv 3 \bmod 4$ and $x\in \pi_{p,0}\f_{q+1}(\KT)$.  
Consider its image $y\in \pi_{p,0} \s_{q+1}(\KT)$. 
Since its image under $d_{1}({p,q+1})$ is trivial,
Theorem~\ref{thm:e2-witt} furnishes an element $z\in\pi_{p+1,0}\s_{q+2}(\KT)$ whose image under $d_{1}({p+1,q+2})$ is precisely $y$.  
Consider the difference $x-w$, 
where $w$ is the image of $z$ in $\pi_{p,0}\f_{q+1}(\KT)$. 
The image of $x-w$ in $\pi_{p,0}\f_q(\KT)$ then coincides with the image of $x$.  
Since the image of $x-w$ in $\pi_{p,0}\s_{q+1}(\KT)$ is zero, 
there is an element $v\in \pi_{p,0}\f_{q+2}(\KT)$ mapping to $x-w$. 
Proceeding inductively yields an element 
\[ 
e
\in 
\bigcap_{i\geq 1}\f_{q+i}\pi_{p,0}\f_{q}(\KT).
\] 
However, 
using that $\pi_{p,0} \KT=0$,
this group is trivial by the exact sequence
\[ 
0 
\rTo 
\f_{q}\pi_{p+1,0}\KT/\f_{q+1}\pi_{p+1,0}\KT 
\rTo 
h^{q,q} 
\rTo
\bigcap_{i\geq 1} \f_{q+i} \pi_{p,0} \f_q(\KT) 
\rTo 
\bigcap_{i\geq 0}\f_{q+i} \pi_{p,0}\KT 
\rTo 
0
\] 
from the proof of \cite[Lemma 7.2]{Voevodsky:open}. 
Hence $e=0$, 
and the image of $x$ in $\pi_{p,0}\f_q(\KT)$ is zero.
\end{proof}

\begin{corollary}
\label{cor:filt-zero-123}
Let $q\geq 0$ and $p\equiv 2,3\bmod 4$. 
There is a canonically induced split short exact sequence of $\FF_2$-modules
\[ 
0
\rTo 
\pi_{p,0} \f_{q}(\KT) 
\rTo 
\pi_{p,0} \s_{q}(\KT) 
\rTo 
\pi_{p-1,0}\f_{q+1}(\KT) 
\rTo 
0. 
\]
\end{corollary}

\begin{corollary}
\label{cor:pifqkt}
For $q\geq 0$ there are canonically induced isomorphisms
\[ 
\pi_{p,0} \f_q(\KT) 
\iso 
\begin{cases} 
h^{q-1,q}\oplus h^{q-5,q}\oplus \dotsm & p \equiv 1 \bmod 4 \\
h^{q-2,q}\oplus h^{q-6,q}\oplus \dotsm & p \equiv 2 \bmod 4 \\
h^{q-3,q}\oplus h^{q-7,q}\oplus \dotsm & p \equiv 3 \bmod 4. 
\end{cases}
\]
\end{corollary}
\begin{proof}
Use Theorem \ref{thm:e2-witt}, Lemma~\ref{lem:filt-zero-123}, and Corollary~\ref{cor:filt-zero-123}. 
\end{proof}

\begin{corollary}
\label{cor:pi0fqkt}
For $q\geq 0$ the canonical map $\Sigma^{1,0}  \s_q(\KT) \rTo \f_{q+1}(\KT)$ induces a split short exact sequence
\[ 
0 
\rTo 
h^{q-3,q}\oplus h^{q-7,q} 
\oplus 
\dotsm 
\rTo 
\pi_{0,0}\f_{q+1}(\KT)
\rTo 
\f_{q+1} \pi_{0,0}\KT = I(F)^{q+1} 
\rTo 0. 
\]
Moreover, 
the map 
\[ 
\pi_{0,0}\f_q (\KT) 
\rTo 
\pi_{0,0}\f_{q-1}(\KT) 
\]
is injective on the image of $\pi_{0,0}\f_{q+1}(\KT)$.
\end{corollary}
\begin{proof}
The latter claim follows by a diagram chase and Theorem~\ref{thm:e2-witt}, since $E^2_{p,q} = 0$ if $p\equiv 1 \bmod 4$. 
Hence the exact sequence
\[ 
\dotsm 
\rTo^\beta 
\pi_{1,0} \s_q(\KT) 
\rTo^\alpha 
\pi_{0,0}\f_{q+1}(\KT) 
\rTo 
\pi_{0,0}\f_q(\KT) 
\rTo 
\dotsm \]
induces the short exact sequence
\[ 
0 
\rTo 
\pi_{1,0} \s_q(\KT)/\Ker(\alpha) 
\rTo 
\pi_{0,0}\f_{q+1}(\KT) 
\rTo 
\f_{q+1}\pi_{0,0} (\KT) 
\rTo 
0. 
\]
Since $\Ker(\alpha) = \Img(\beta) = \Img(d_{2,q-1}) = h^{q-1,q} \oplus h^{q-5,q}\oplus \dotsm$ by Theorem~\ref{thm:e2-witt},
the sequence is short exact. 
It splits by Lemma~\ref{lem:tau-iso}, 
since the composition of 
\[ 
h^{q-3,q}\oplus h^{q-7,q} \oplus \dotsm 
\rTo 
\pi_{0,0}\f_{q+1} 
\]
and 
\[ 
\pi_{0,0} \f_{q+1}(\KT) 
\rTo 
\pi_{0,0}\s_{q+1}(\KT) 
\iso
h^{q+1,q+1} \oplus h^{q-1,q+1}\oplus \dotsm \rTo^{\mathrm{pr}} h^{q-3,q+1} \oplus h^{q-7,q+1} \oplus \dotsm 
\]
is given by multiplication with $\tau\in h^{0,1}$. 
\end{proof}

Theorems~\ref{thm:e2-witt} and \ref{thm:convergence-witt} imply Theorem \ref{thm:main} stated in the introduction.
If $X\in\Sm_F$ is a semilocal scheme and $F$ a field of characteristic zero, 
our computations and results extend to the Witt ring $W(X)$ with fundamental ideal $I(X)$ and the mod-$2$ motivic cohomology of $X$.
Our reliance on the Milnor conjecture for Galois cohomology \cite{Voevodsky:Z/2} can be replaced by \cite[\S2.2]{Hoobler} or \cite[Theorem 7.8]{Kerz}, 
while the isomorphism (\ref{equation:abg}) holds for $X$ by \cite[Theorem 7.6]{Kerz}.
The rest of the proof is identical to the one given for fields.
Kerz proved a closely related result in \cite[Theorem 7.10]{Kerz}.
By periodicity of $\KT$ there is an evident variant of Theorem \ref{thm:e2-witt} for the $n$th slice spectral sequence of $\KT$ for every $n\in \ZZ$.
We note the following result from \cite{Baeza} is transparent from our computation of the slice spectral sequence for $\KT$. 
\begin{corollary}
If $X\in\Sm_F$ is a semilocal scheme of geometric origin then $W(X)$ contains no elements of odd order.
If $X$ is not formally real then $W(X)$ is a $2$-primary torsion group. 
\end{corollary}

\section{Hermitian $K$-groups}
\label{section:hk-groups}
According to Theorem \ref{thm:slices-ko} the $E^{1}$-page of the $0$th slice spectral sequence for $\KQ$ takes the form

\[ 
E^{1}_{p,q} 
= 
\pi_{p,0}\s_{q}(\KQ) 
= 
\begin{cases}
H^{2q-p,q}\directsum \bigdirectsum_{i<\frac{q}{2}} h^{2i+(q-p),q} & q\equiv 0 \bmod 2 \\
\bigdirectsum_{i\leq\frac{q-1}{2}} h^{2i+(q-p),q} & q\equiv 1 \bmod 2. 
\end{cases}
\]

Using the formula in Theorem~\ref{thm:diff-ko} for the first differentials we are ready to perform low-degree computations in the slice spectral sequence for $\KQ$.
We assume throughout that $F$ is a field of $\Char(F)\neq 2$.
\vspace{0.1in}

The Beilinson-Soule vanishing conjecture predicts the integral motivic cohomology group $H^{2q-p,q}$ is trivial if $p>2q$, 
which holds for instance for finite fields and number fields.
(The same group is uniquely divisible if $p\geq 2q$, except when $p=q=0$).  
We note that $H^{4-p,2}=0$ for $p\geq 4$ by comparison with Lichtenbaum's weight-two motivic complex $\Gamma(2)$ \cite[\S 7]{BL}, \cite{Lichtenbaum}.
\vspace{0.1in}

In the following table for the $E^1$-page of the slice spectral sequence for $\KQ$, 
the group $E_{p,q}^1$ in bidegree $(p,q)$ is a direct sum of motivic cohomology groups positioned on the vertical line above $p$ and inbetween the horizontal lines 
corresponding to the weights $q$ and $q+1$.

\begin{center}
\begin{tikzpicture}[scale=1.2,font=\scriptsize,line width=1pt]
\draw[help lines] (-2.5,0) grid (9.5,7.2);
\foreach \i in {0,...,7} {\node[label=left:$\i$] at (-2.5,\i) {};}
\foreach \i in {-2,...,9} {\node[label=below:$\i$] at (\i,-.2) {};}

\foreach \i in {0,...,6} {\draw[fill] (-2,\i) circle (1pt)
                                        node[above right=-1pt] {$h^{\i,\i}$};}

{\draw[fill]     
(-2,2.5) circle (1pt) node[above right=-1pt] {$h^{0,2}$}
(-2,3.5) circle (1pt) node[above right=-1pt] {$h^{1,3}$}
(-2,4.33) circle (1pt) node[above right=-1pt] {$h^{2,4}$}
(-2,5.33) circle (1pt) node[above right=-1pt] {$h^{3,5}$}
(-2,6.25) circle (1pt) node[above right=-1pt] {$h^{4,6}$}
(-2,6.75) circle (1pt) node[above right=-1pt] {$h^{0,6}$}
;}
{\draw[fill]
(-2,4.66) circle (1pt) node[above right=-1pt] {$h^{0,4}$}
(-2,5.66) circle (1pt) node[above right=-1pt] {$h^{1,5}$}
(-2,6.5) circle (1pt) node[above right=-1pt] {$h^{2,6}$}
;}

{\draw[fill]      
(-1,3.5) circle (1pt) node[above right=-1pt] {$h^{0,3}$}
(-1,4.5) circle (1pt) node[above right=-1pt] {$h^{1,4}$}
(-1,5.33) circle (1pt) node[above right=-1pt] {$h^{2,5}$}
(-1,6.33) circle (1pt) node[above right=-1pt] {$h^{3,6}$}
;}

{\draw[fill]
(-1,1) circle (1pt) node[above right=-1pt] {$h^{0,1}$}
(-1,2) circle (1pt) node[above right=-1pt] {$h^{1,2}$}
(-1,3) circle (1pt) node[above right=-1pt] {$h^{2,3}$}
(-1,4) circle (1pt) node[above right=-1pt] {$h^{3,4}$}
(-1,5) circle (1pt) node[above right=-1pt] {$h^{4,5}$}
(-1,5.66) circle (1pt) node[above right=-1pt] {$h^{0,5}$}
(-1,6) circle (1pt) node[above right=-1pt] {$h^{5,6}$}
(-1,6.66) circle (1pt) node[above right=-1pt] {$h^{1,6}$}
;}

\foreach \i in {1,...,6} {\draw[fill] (0,\i) circle (1pt)
                                        node[above right=-1pt] {$h^{\i,\i}$};}

{\draw[fill]     
(0,0) circle (0pt) node[above right=1pt] {{$H^{0,0}$}}
(0,4.66) circle (1pt) node[above right=-1pt] {$h^{0,4}$}
(0,5.66) circle (1pt) node[above right=-1pt] {$h^{1,5}$}
(0,6.5) circle (1pt) node[above right=-1pt] {$h^{2,6}$}
;}

{\draw[fill]
(0,2.5) circle (1pt) node[above right=-1pt] {$h^{0,2}$}
(0,3.5) circle (1pt) node[above right=-1pt] {$h^{1,3}$}
(0,4.33) circle (1pt) node[above right=-1pt] {$h^{2,4}$}
(0,5.33) circle (1pt) node[above right=-1pt] {$h^{3,5}$}
(0,6.25) circle (1pt) node[above right=-1pt] {$h^{4,6}$}
(0,6.75) circle (1pt) node[above right=-1pt] {$h^{0,6}$}
;}

{\draw[fill]      
(1,1) circle (1pt) node[above right=-1pt] {$h^{0,1}$}
(1,2) circle (1pt) node[above right=-1pt] {$h^{1,2}$}
(1,3) circle (1pt) node[above right=-1pt] {$h^{2,3}$}
(1,4) circle (1pt) node[above right=-1pt] {$h^{3,4}$}
(1,5) circle (1pt) node[above right=-1pt] {$h^{4,5}$}
(1,5.66) circle (1pt) node[above right=-1pt] {$h^{0,5}$}
(1,6) circle (1pt) node[above right=-1pt] {$h^{5,6}$}
(1,6.66) circle (1pt) node[above right=-1pt] {$h^{1,6}$}
;}

{\draw[fill]
(1,3.5) circle (1pt) node[above right=-1pt] {$h^{0,3}$}
(1,4.5) circle (1pt) node[above right=-1pt] {$h^{1,4}$}
(1,5.33) circle (1pt) node[above right=-1pt] {$h^{2,5}$}
(1,6.33) circle (0pt) node[above right=1pt] {$h^{3,6}$}
;}

\foreach \i in {3,...,6} {\draw[fill] (2,\i) circle (1pt)
                                       node[above right=-1pt] {$h^{\i,\i}$};}

{\draw[fill]     
(2,2.5) circle (1pt) node[above right=-1pt] {$h^{0,2}$}
(2,3.5) circle (1pt) node[above right=-1pt] {$h^{1,3}$}
(2,4.33) circle (1pt) node[above right=-1pt] {$h^{2,4}$}
(2,5.33) circle (1pt) node[above right=-1pt] {$h^{3,5}$}
(2,6.25) circle (1pt) node[above right=-1pt] {$h^{4,6}$}
(2,6.75) circle (1pt) node[above right=-1pt] {$h^{0,6}$}
;}
{\draw[fill]
(2,2) circle (0pt) node[above right=1pt] {{$H^{2,2}$}}
(2,4.66) circle (1pt) node[above right=-1pt] {$h^{0,4}$}
(2,5.66) circle (1pt) node[above right=-1pt] {$h^{1,5}$}
(2,6.5) circle (1pt) node[above right=-1pt] {$h^{2,6}$}
;}

{\draw[fill]      
(3,3.5) circle (1pt) node[above right=-1pt] {$h^{0,3}$}
(3,4.5) circle (1pt) node[above right=-1pt] {$h^{1,4}$}
(3,5.33) circle (1pt) node[above right=-1pt] {$h^{2,5}$}
(3,6.33) circle (1pt) node[above right=-1pt] {$h^{3,6}$}
;}

{\draw[fill]
(3,2) circle (0pt) node[above right=1pt] {$H^{1,2}$}
(3,3) circle (1pt) node[above right=-1pt] {$h^{2,3}$}
(3,4) circle (1pt) node[above right=-1pt] {$h^{3,4}$}
(3,5) circle (1pt) node[above right=-1pt] {$h^{4,5}$}
(3,5.66) circle (1pt) node[above right=-1pt] {$h^{0,5}$}
(3,6) circle (1pt) node[above right=-1pt] {$h^{5,6}$}
(3,6.66) circle (1pt) node[above right=-1pt] {$h^{1,6}$}
;}

\foreach \i in {5,...,6} {\draw[fill] (4,\i) circle (1pt)
                                        node[above right=-1pt] {$h^{\i,\i}$};}

{\draw[fill]     
(4,4) circle (0pt) node[above right=1pt] {{$H^{4,4}$}}
(4,4.66) circle (1pt) node[above right=-1pt] {$h^{0,4}$}
(4,5.66) circle (1pt) node[above right=-1pt] {$h^{1,5}$}
(4,6.5) circle (1pt) node[above right=-1pt] {$h^{2,6}$}
;}

{\draw[fill]
(4,3.0) circle (1pt) node[above right=-1pt] {$h^{1,3}$}
(4,4.33) circle (1pt) node[above right=-1pt] {$h^{2,4}$}
(4,5.33) circle (1pt) node[above right=-1pt] {$h^{3,5}$}
(4,6.25) circle (1pt) node[above right=-1pt] {$h^{4,6}$}
(4,6.75) circle (1pt) node[above right=-1pt] {$h^{0,6}$}
;}

{\draw[fill]      
(5,4) circle (0pt) node[above right=1pt] {$H^{3,4}$}
(5,5) circle (1pt) node[above right=-1pt] {$h^{4,5}$}
(5,5.66) circle (1pt) node[above right=-1pt] {$h^{0,5}$}
(5,6) circle (1pt) node[above right=-1pt] {$h^{5,6}$}
(5,6.66) circle (1pt) node[above right=-1pt] {$h^{1,6}$}
;}

{\draw[fill]
(5,3.0) circle (1pt) node[above right=-1pt] {$h^{0,3}$}
(5,4.5) circle (1pt) node[above right=-1pt] {$h^{1,4}$}
(5,5.33) circle (1pt) node[above right=-1pt] {$h^{2,5}$}
(5,6.33) circle (1pt) node[above right=-1pt] {$h^{3,6}$}
;}

{\draw[fill]     
(6,4.0) circle (0pt) node[above right=1pt] {{$H^{2,4}$}}
(6,5.0) circle (1pt) node[above right=-1pt] {$h^{3,5}$}
(6,6.25) circle (1pt) node[above right=-1pt] {$h^{4,6}$}
(6,6.75) circle (1pt) node[above right=-1pt] {$h^{0,6}$}
;}
{\draw[fill]
(6,4.5) circle (1pt) node[above right=-1pt] {$h^{0,4}$}
(6,5.5) circle (1pt) node[above right=-1pt] {$h^{1,5}$}
(6,6) circle (0pt) node[above right=1pt] {{$H^{6,6}$}}
(6,6.5) circle (1pt) node[above right=-1pt] {$h^{2,6}$}
;}

{\draw[fill]      
(7,4.0) circle (0pt) node[above right=1pt] {$H^{1,4}$}
(7,5.0) circle (1pt) node[above right=-1pt] {$h^{2,5}$}
(7,6.33) circle (1pt) node[above right=-1pt] {$h^{3,6}$}
;}

{\draw[fill]
(7,5.5) circle (1pt) node[above right=-1pt] {$h^{0,5}$}
(7,6) circle (0pt) node[above right=1pt] {$H^{5,6}$}
(7,6.66) circle (1pt) node[above right=-1pt] {$h^{1,6}$}
;}

{\draw[fill]     
(8,4.0) circle (0pt) node[above right=1pt] {{$H^{0,4}$}}
(8,5.0) circle (1pt) node[above right=-1pt] {$h^{1,5}$}
(8,6.33) circle (1pt) node[above right=-1pt] {$h^{2,6}$}
;}

{\draw[fill]
(8,6.0) circle (0pt) node[above right=1pt] {{$H^{4,6}$}}
(8,6.66) circle (1pt) node[above right=-1pt] {$h^{0,6}$}
;}

{\draw[fill]      
(9,4.0) circle (0pt) node[above right=1pt] {$H^{-1,4}$}
(9,5.0) circle (1pt) node[above right=-1pt] {{$h^{0,5}$}}
(9,6.33) circle (1pt) node[above right=-1pt] {$h^{1,6}$}
;}

{\draw[fill]
(9,6.0) circle (2pt) node[above right=-1pt] {{$H^{3,6}$}}
;}
\node at (0,0) [shape=rectangle,draw,fill]{};
\node at (2,2) [shape=rectangle,draw,fill]{};
\node at (3,2) [shape=rectangle,draw,fill]{};
\node at (4,4) [shape=rectangle,draw,fill]{};
\node at (5,4) [shape=rectangle,draw,fill]{};
\node at (6,4) [shape=rectangle,draw,fill]{};
\node at (7,4) [shape=rectangle,draw,fill]{};
\node at (8,4) [shape=rectangle,draw,fill]{};
\node at (9,4) [shape=rectangle,draw,fill]{};
\node at (6,6) [shape=rectangle,draw,fill]{};
\node at (7,6) [shape=rectangle,draw,fill]{};
\node at (8,6) [shape=rectangle,draw,fill]{};
\node at (9,6) [shape=rectangle,draw,fill]{};

\end{tikzpicture}
\end{center}

A comparison with $\KT$ via Proposition \ref{prop:slices-unit-kt} implies:
\begin{lemma}
\label{lemma:BScomputation}
Suppose that $p\leq 3$. Then the canonical map induces an identification
$E^{2}_{p,q}(\KQ)=E^{2}_{p,q}(\KT)$ for all $q\geq p+1$ if $p$ is even and for all
$q\geq p+2$ if $p$ is odd. 
If $F$ satisfies the Beilinson-Soule vanishing conjecture, the identification
holds also for $p>3$.
\end{lemma}

Lemma \ref{lemma:BScomputation} combined with the classical computation of $\pi_{0,0}\KQ$ as the Grothendieck-Witt group 
of $F$ basically determines the $0$th and $1$st columns of the $0$th slice spectral sequence of $\KQ$.
In the $1$st column we note:
There are no entering or exiting differentials in weight $q\leq 2$.
In weight three the exiting differential has kernel $h^{2,3}$ by Corollary~\ref{cor:sq-weight-1}. 
Moreover,
the differential 
\[ 
E^{1}_{2,2}(\KQ)
=
H^{2,2}\directsum h^{0,2}
\rTo 
h^{2,2}\directsum h^{0,2} 
\rTo^{\tau+\Sq^{2}} 
h^{2,3} 
\]
is surjective by Lemma~\ref{lem:tau-iso} and Corollary~\ref{cor:sq-weight-1}.
All $d_{r}$-differentials exiting $E^{r}_{p,q}(\KQ)$ for $p=0$, $p=1$ and $r\geq 2$ are trivial.

\begin{lemma}
There are isomorphisms
\begin{equation*}
E^{\infty}_{p,q}(\KQ)
\cong
\begin{cases}
H^{0,0}  & p=q=0 \\
h^{q,q}  & p=0, q>0 \\
h^{0,1}  & p=q=1 \\
h^{1,2}  & p=1, q=2 \\
0 & p=1, q\neq 1,2.
\end{cases}
\end{equation*}
\end{lemma}
\begin{remark}
There are isomorphisms $H^{0,0}\cong\ZZ$, 
$h^{0,1}\cong\ZZ/2$,
and $h^{1,2}\cong F^{\times}/2$.
\end{remark}

By computing in the $2$nd column we obtain the motivic cohomological description of the second orthogonal $K$-group $KO_{2}(F)$
stated in Theorem \ref{thm:mainhermitian}.
\begin{lemma}
There are isomorphisms
\begin{equation*}
E^{\infty}_{2,q}(\KQ)
\cong
\begin{cases}
\mathrm{ker}(\tau\circ\pr+\Sq^{2}\colon H^{2,2}\directsum h^{0,2}\rTo h^{2,3}) & q=2 \\
0 & q\neq 2.
\end{cases}
\end{equation*}
\end{lemma}
\begin{proof}
The claim for $q=2$ follows from Theorem \ref{thm:diff-ko}.
We note that the kernel of the surjection $\tau\circ\pr+\Sq^{2}$ is the preimage of the subgroup $\{0,\rho^{2}\}\subseteq h^{2,2}$ 
under $\pr\colon H^{2,2}\rTo h^{2,2}$.
If $\rho^{2}=0$, 
this group coincides with $2H^{2,2}$. 
In weight three, 
the kernel of 
\[ 
\tau+\Sq^{2}
\colon
h^{3,3}\directsum h^{1,3}
\rTo
h^{3,4} 
\] 
is isomorphic to $h^{1,1}$ via $\phi\rMapsto (\rho^{2}\phi,\tau^{2}\phi)$. 
The image of the entering differential corresponds to the image of $H^{1,1}$ in $h^{1,1}$ under this isomorphism,
hence it coincides with $h^{1,1}$.
The remaining vanishing follows from Lemma \ref{lemma:BScomputation}.
\end{proof}
\begin{remark}
There are isomorphisms $H^{2,2}\cong K_{2}(F)$, 
$h^{0,2}\cong\ZZ/2$,
and $h^{2,3}\cong \: _{2}Br(F)$ the $2$-torsion subgroup of the Brauer group of equivalence classes of central simple $F$-algebras.
\end{remark}

\begin{lemma}
There is a short exact sequence
$
0
\to
h^{0,3}
\to
KO_{3}(F)
\to
2H^{1,2}
\to 
0
$
and isomorphisms
\begin{equation*}
E^{\infty}_{3,q}(\KQ)
\cong
\begin{cases}
2H^{1,2} & q=2 \\
h^{0,3} & q=3 \\
0 & q\neq 2, 3.
\end{cases}
\end{equation*}
\end{lemma}
\begin{proof}
Note that $\Sq^{2}$ acts trivially on $h^{1,2}$ by Corollary~\ref{cor:sq-weight-1}.
In weight $2$ we look at the kernel of $\tau\circ\pr\colon H^{1,2}\rTo h^{1,3}$.
By Lemma~\ref{lem:tau-iso} this equals the kernel of $\pr$, 
i.e.,
$2H^{1,2}\subseteq H^{1,2}$.
In weight $3$, 
the kernel of the exiting differential is isomorphic to $h^{0,0}$ via $c\rMapsto (\tau\rho^{2}c,\tau^{3}c)$. 
In weight $4$, 
the kernel of the exiting differential is isomorphic to $h^{1,1}$ via $\phi\rMapsto (\tau\rho^{2}\phi,\tau^{3}\phi)$.
The entering differential surjects onto the kernel of the latter map.
In weight $5$, 
the group coincides with the corresponding group for $\KT$, 
as the change from $h^{4,4}$ to $H^{4,4}$ does not affect the spectral sequence, 
due to the triviality of the differential exiting bidegree $(4,4)$.
The remaining claims follow from Lemma \ref{lemma:BScomputation}.
\end{proof}
\begin{remark}
There are isomorphisms $H^{1,2}\cong K_{3}^{ind}(F)$;
the cokernel of $K_{3}^{M}(F)\to K_{3}(F)$, 
i.e., the $K_{3}$-group of indecomposable elements, 
and $h^{0,3}\cong\ZZ/2$.
By Lichtenbaum's weight-two motivic complex $\Gamma(2)$ \cite{Lichtenbaum} there is a short exact sequence
$
0
\to
h^{0,2}
\to
H^{1,2}
\to
2H^{1,2}
\to 
0.
$
By comparing with the forgetful map $\KQ\to\KGL$ and identifying $h^{0,2}$ with $h^{0,3}$ one concludes that $KO_{3}(F)$ is isomorphic to $K_{3}^{ind}(F)$.
\end{remark}

\begin{lemma}
There are isomorphisms
\begin{equation*}
E^{2}_{4,q}(\KQ)
\cong
\begin{cases}
0           & q=2, 3 \\
H^{4,4} & q=4 \\
h^{q,q} & q\geq 5.
\end{cases}
\end{equation*}
\end{lemma}
\begin{proof}
As noted above the group $H^{0,2}$ is trivial.
Lemma~\ref{lem:tau-iso} shows the exiting differential in weight $3$ is injective.
In weight $4$, 
the kernel of the exiting differential is $H^{4,4}\directsum h^{0,4}$.
The nontrivial element in the image of the entering differential is $(\delta(\rho^{3}\tau),\tau^{4})$. 
Thus the quotient can be identified with $H^{4,4}$. 
In weight $5$, 
this follows from the case of $\KT$ because $\Sq^{2}\circ\pr\colon H^{3,4}\rTo h^{3,4}\rTo h^{5,5}$ is trivial.
Lemma \ref{lemma:BScomputation} finishes the proof.
\end{proof}

In principle one can continue with a similar analysis of the next columns.
To summarize the computations above, 
we note that in low-degrees the $E^{2}$-page takes the form:

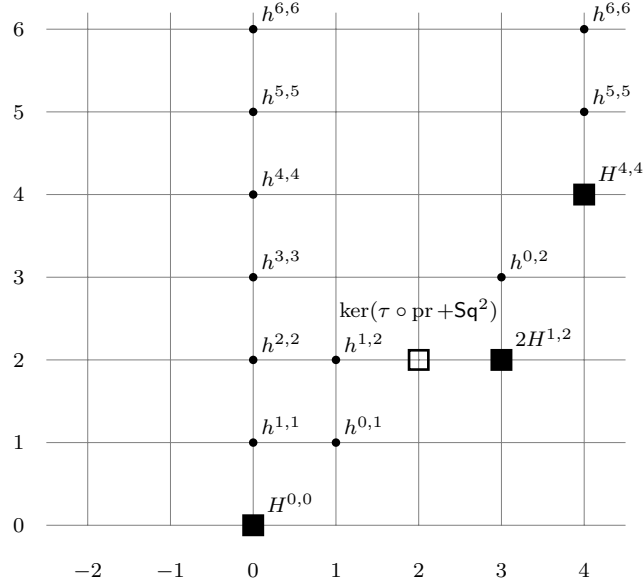
\begin{center}
\begin{tikzpicture}[scale=1.1,font=\scriptsize,line width=1pt]
\draw[help lines] (-2.5,0) grid (4.5,6.2);
\foreach \i in {0,...,6} {\node[label=left:$\i$] at (-2.5,\i) {};}
\foreach \i in {-2,...,4} {\node[label=below:$\i$] at (\i,-.2) {};}
\foreach \i in {1,...,6} {\draw[fill] (0,\i) circle (1pt)
                                        node[above right=-1pt] {$h^{\i,\i}$};}
{\draw[fill]     
(0,0) circle (0pt) node[above right=1pt] {{$H^{0,0}$}}
;}

{\draw[fill]      
(1,1) circle (1pt) node[above right=-1pt] {$h^{0,1}$}
(1,2) circle (1pt) node[above right=-1pt] {$h^{1,2}$}
;}

{\draw[fill]     
(2,2) circle (0pt) node[above=10pt] {$\mathrm{ker}(\tau\circ\pr + \Sq^{2})$}
;}

{\draw[fill]      
(3,3) circle (1pt) node[above right=-1pt] {$h^{0,2}$}
(3,2) circle (0pt) node[above right=1pt] {$2H^{1,2}$}
;}

\foreach \i in {5,...,6} {\draw[fill] (4,\i) circle (1pt)
                                        node[above right=-1pt] {$h^{\i,\i}$};}
{\draw[fill]     
(4,4) circle (0pt) node[above right=1pt] {{$H^{4,4}$}}
;}
\node at (0,0) [shape=rectangle,draw,fill]{};
\node at (2,2) [shape=rectangle,draw]{};
\node at (3,2) [shape=rectangle,draw,fill]{};
\node at (4,4) [shape=rectangle,draw,fill]{};

\end{tikzpicture}
\end{center}

\begin{corollary}
The group $KO_{4}(F)$ surjects onto $H^{4,4}\cong K_{4}^{M}$.
If $K_{4}^{M}=0$ then $KO_{4}(F)$ is the trivial group.
\end{corollary}

The symplectic $K$-groups $KSp_{\ast}(F)$ of $F$ are the filtered target groups of the second slice spectral sequence for $\KQ$
on account of the isomorphism $\pi_{p,2}\KQ\iso KSp_{p-4}(F)$.
Computations similar to the above yields the $E^{2}$-page:

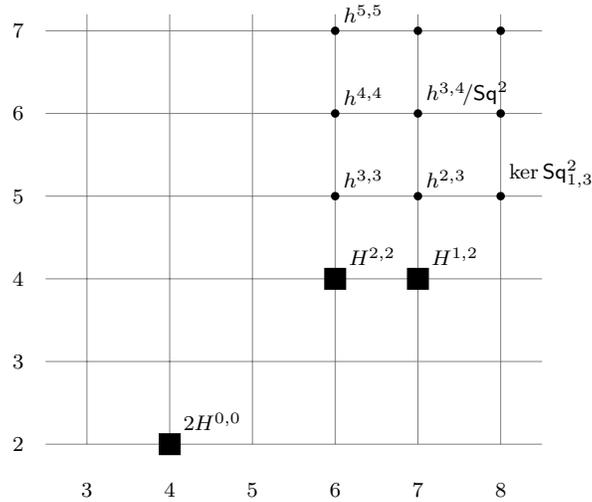
\begin{center}
\begin{tikzpicture}[scale=1.1,font=\scriptsize,line width=1pt]
\draw[help lines] (2.5,2) grid (8.5,7.2);
\foreach \i in {2,...,7} {\node[label=left:$\i$] at (2.5,\i) {};}
\foreach \i in {3,...,8} {\node[label=below:$\i$] at (\i,1.8) {};}
{\draw[fill]     
(4,2) circle (0pt) node[above right=1pt] {{$2H^{0,0}$}}
;}
{\draw[fill]     
(6,4) circle (0pt) node[above right=1pt] {$H^{2,2}$}
(6,5) circle (1pt) node[above right=-1pt] {$h^{3,3}$}
(6,6) circle (1pt) node[above right=-1pt] {$h^{4,4}$}
(6,7) circle (1pt) node[above right=-1pt] {$h^{5,5}$}
(7,4) circle (0pt) node[above right=1pt] {$H^{1,2}$}
(7,5) circle (1pt) node[above right=-1pt] {$h^{2,3}$}
(7,6) circle (1pt) node[above right=-1pt] {$h^{3,4}\! /\Sq^2$}
(7,7) circle (1pt) node[above right=-1pt] {}
(8,5) circle (1pt) node[above right=-1pt] {$\ker \Sq^2_{1,3}$}
(8,6) circle (1pt) node[above right=-1pt] {}
(8,7) circle (1pt) node[above right=-1pt] {}
;}
\node at (4,2) [shape=rectangle,draw,fill]{};
\node at (6,4) [shape=rectangle,draw,fill]{};
\node at (7,4) [shape=rectangle,draw,fill]{};
\end{tikzpicture}
\end{center}

We read off that $KSp_{0}(F)\cong 2H^{0,0}$ is infinite cyclic and $KSp_{1}(F)$ is the trivial group.
It follows that all the classes in the sixth column are infinite cycles and $E^{2}_{p,q}=E^{\infty}_{p,q}$ when $(p,q)=(6,4), (6,5)$.
Hence we obtain a surjection $KSp_{2}(F)\rTo H^{2,2}$.  
Its kernel is isomorphic to $I(F)^{3}$ as shown by Suslin \cite[\S 6]{SuslinK2}.
If $\cd_{2}(F)\leq 2$, 
so that $h^{p,q}=0$ when $p\geq 3$, 
the seventh column degenerates to a short exact sequence
$
0
\to
h^{2,3}
\to
KSp_{3}(F)
\to
H^{1,2}
\to 
0.
$

\begin{remark}
The computations in this section hold for smooth semilocal rings containing a field of characteristic zero, 
cf.~the generalization of Milnor's conjecture on quadratic forms discussed in the introduction.
\end{remark}

\appendix
\section{Maps between motivic Eilenberg-MacLane spectra}
\label{section:endom-motiv-eilenb}
Throughout this section the base scheme is essentially smooth over a field of characteristic unequal to $2$ \cite[Definition 2.9]{HKPAO}.
In the following series of results, 
we identify weight zero and weight one endomorphisms of motivic Eilenberg-MacLane spectra in $\SHH$ in terms of 
Steenrod operations $\Sq^{i}$ and the motivic cohomology classes $\rho$, $\tau$.
When the base scheme is a field of characteristic zero, 
these identifications follow from Voevodsky's work on the motivic Steenrod algebra \cite{Voevodsky:Steenrod},
while the generalization to our set-up relies on \cite{HKPAO}.
We use square brackets to denote maps in $\SHH$.
\begin{lemma} 
\begin{align*}
[\MZZ/2,\Sigma^{p,0}  \MZZ/2] 
= &
\begin{cases} 
\FF_{2} & p = 0 \\ 
\FF_{2}\{\Sq^{1}\} & p=1\\
0 & \mathrm{otherwise.}
\end{cases} 
\end{align*}
\end{lemma}
Recall that $\Sq^{1}$ is the canonical composite map $\MZZ/2 \rTo^{\delta} \Sigma^{1,0}  \MZZ\rTo^{\mathrm{pr}} \Sigma^{1,0}  \MZZ/2$.
\begin{lemma} 
\label{lemma:mod2tomod2weight1}
\begin{align*}
[\MZZ/2,\Sigma^{p,1}  \MZZ/2] 
= &
\begin{cases} 
\tau h^{0,0}  & p=0 \\ 
h^{1,1}\directsum h^{0,0}\{\tau\Sq^{1}\} & p=1 \\
h^{1,1}\Sq^{1} \directsum \, h^{0,0}\{\Sq^{2}\} & p=2\\
h^{0,0}\{\Sq^{2}\Sq^{1}\} \directsum h^{0,0}\{\Sq^{1}\Sq^{2}\} & p=3 \\
h^{0,0}\{\Sq^{1}\Sq^{2}\Sq^{1}\} & p=4 \\
0 & \mathrm{otherwise.}
\end{cases} 
\end{align*}
\end{lemma}

\begin{lemma} 
\label{lemma:MZtoMZ/2weightzero}
\begin{align*} 
[\MZZ,\Sigma^{p,0}  \MZZ/2] 
= &
\begin{cases} 
\FF_{2}\{\mathrm{pr}\} & p=0 \\
0 & p\neq 0 
\end{cases} 
& 
[\MZZ/2,\Sigma^{p,0}  \MZZ] 
= &
\begin{cases} 
\FF_{2}\{\delta\} & p=1 \\
0 & p\neq 1. 
\end{cases}
\end{align*} 
\end{lemma}

\begin{lemma} 
\label{lemma:integraltomod2andmod2tointegralweight1}
\begin{align*} 
[\MZZ,\Sigma^{p,1}  \MZZ/2] 
= &
\begin{cases} 
\tau h^{0,0}\circ \mathrm{pr} & p=0 \\
h^{1,1}\circ \mathrm{pr} & p=1 \\
h^{0,0}\{\Sq^{2}\circ \mathrm{pr}\} & p=2\\
h^{0,0}\{\Sq^{1}\Sq^{2}\circ \mathrm{pr} \} & p=3 \\
0 & \mathrm{otherwise}
\end{cases} \\
[\MZZ/2,\Sigma^{p,1}  \MZZ] 
= & 
\begin{cases} 
\delta \circ \tau h^{0,0} & p=1 \\
\delta \circ h^{1,1} & p=2 \\
\FF_{2}\{\delta \circ \Sq^{2}\} & p=3\\
\FF_{2}\{\delta \circ \Sq^{2}\Sq^{1}\} & p=4 \\
0 & \mathrm{otherwise.}
\end{cases} 
\end{align*} 
\end{lemma}
Here $\Sq^{1}(\tau)=\rho$, $\Sq^{2}(\tau)=0$ and $\Sq^{2}(\tau^{2})=\tau\rho^{2}$.
It follows that
\begin{align*}
\Sq^{1}(\tau^n)
= &
\begin{cases} 
\rho\tau^{n-1} & n\equiv 1\bmod 2\\
0              & n\equiv 0\bmod 2 
\end{cases}
&
\Sq^{2}(\tau^n) 
= & 
\begin{cases} 
\rho^{2} \tau^{n-1} & n\equiv 2,3\bmod 4 \\
0                   & n\equiv 0,1\bmod 4. 
\end{cases}
\end{align*}

\begin{proposition}\label{prop:sum-product}
For every subset $A\subseteq \ZZ$ there is a canonical weak equivalence 
\[ 
\alpha\colon
\bigvee_{i\in A} \Sigma^{i,0}  \MZZ/2 
\rTo 
\prod_{i\in A} \Sigma^{i,0}  \MZZ/2.
\]
\end{proposition}
\begin{proof}
It suffices to show $[\Sigma^{p,q}  X_{+},\alpha]$ is an isomorphism for all $p,q\in\ZZ$ and $X\in\Sm_{F}$. 
This is the canonical map
\[ 
\bigdirectsum_{i\in A}h^{i-p,-q}(X_{+}) 
\rTo 
\prod_{i\in A} h^{i-p,-q}(X_{+}). 
\]
Work of Suslin-Voevodsky \cite{Suslin-Voevodsky:banff} shows the group $h^{i-p,q}(X_{+})$ is nonzero only if $0\leq i-p \leq 2q$,
see \cite[Corollary 2.14]{HKPAO} for base schemes essentially smooth over a field.
\end{proof}

\begin{footnotesize}


\end{footnotesize}
\vspace{0.1in}

\begin{center}
Institut f\"ur Mathematik, Universit\"at Osnabr\"uck, Germany.\\
e-mail: oroendig@uni-osnabrueck.de
\end{center}
\begin{center}
Department of Mathematics, University of Oslo, Norway.\\
e-mail: paularne@math.uio.no
\end{center}

\end{document}